\documentclass{amsart}
\usepackage{amsfonts}

\newtheorem{theorem}{Theorem}
\theoremstyle{plain}

\newtheorem{conjecture}{Conjecture}
\newtheorem{corollary}{Corollary}

\newtheorem{definition}{Definition}

\newtheorem{lemma}{Lemma}

\newtheorem{proposition}{Proposition}
\newtheorem{remark}{Remark}

\numberwithin{equation}{section}

\newcommand{\B}{\mathbb{B}}
\newcommand{\R}{\mathbb{R}}
\newcommand{\fblv}{f_{\overline{\lambda}_0,v_0}}
\newcommand{\Blvp}{\B^+_{\lambda,v_0}}
\newcommand{\Bblvp}{\B^+_{\overline{\lambda}_0,v_0}}

\newcommand{\cna}{c_{n,\alpha}}
\newcommand{\CNA}{C(n,\alpha)}
\newcommand{\ipt}{(1-2r\cos\phi+r^2)}
\newcommand{\flv}{f_{\lambda,v_0}}
\newcommand{\Sph}{\mathbb{S}}
\newcommand{\pk}{\widetilde{p}}
\newcommand{\PK}{\widetilde{P}}
\newcommand{\pka}{\widetilde{p}_\alpha}
\newcommand{\PKA}{\widetilde{P}_\alpha}
\newcommand{\plv}{\phi_{\lambda,v_0}}
\newcommand{\upka}{p_\alpha}
\newcommand{\UPKA}{P_\alpha}

\begin{document}
\title{on a family of integral operators on the ball}
\author[W. Tian]{Wenchuan Tian}
\address{Department of Mathematics\\University of California\\ Santa Barbara, CA 93106}
\email{wtian@math.ucsb.edu}

\maketitle

\begin{abstract}
In this work, we transform the equation in the upper half space first studied by Caffarelli and Silvestre to an equation in the Euclidean unit ball $\mathbb{B}^n$. We identify the Poisson kernel for the equation in the unit ball. Using the Poisson kernel, we define the extension operator. We prove an extension inequality in the limit case and prove the uniqueness of the extremal functions in the limit case using the method of moving spheres. In addition we offer an interpretation of the limit case inequality as a conformally invariant generalization of Carleman's inequality. 
\end{abstract}

\section{\protect\bigskip Introduction}
\subsection{Conformally Invariant Generalization of Carleman's Inequality}
In \cite{C} Carleman proved the following:
\begin{theorem}\cite{C}
For any $u\in C^\infty(\overline{\B^2})$ such that $u$ is harmonic in $\B^2$ with respec to the Euclidean metric then we have
\begin{equation}\label{CarlemanInequality}
\int_{\B^2} e^{2u}dx\leq \frac{1}{4\pi}\left(\int_{\Sph^1} e^u d\theta\right)^2.
\end{equation}
Where equality holds for either $u(x)=c$ or $u(x)=-2\ln|x-x_0|+c$ where $c\in \R$ is any constant and $x_0\in \R^2\backslash \overline{\B}^2$.
\end{theorem}
Note that the inequality (\ref{CarlemanInequality}) is conformally invariant and that it also holds for subharmonic functions. We will refer to (\ref{CarlemanInequality}) as Carleman's inequality through out this article. 

Using inequality ($\ref{CarlemanInequality}$) Carleman proved that isoperimetric inequality holds in two dimensional minimal surfaces in $\R^3$ \cite{C}. Beckenbach and Rado also used inequality ($\ref{CarlemanInequality}$) to prove that the isoperimetric inequality holds in analytic surfaces in $\R^3$ with nonpositive gaussian curvature \cite{BR}.

There are several generalizations of Carleman's inequality to higher dimensional unit ball. For example, Hang, Wang and Yan \cite{HWY2} proved proved the following generalization for harmonic functions in higher dimensional unit ball:
\begin{corollary}\cite[Corollary 3.1]{HWY2}
	Assume $n\geq 3$, then for $\widetilde{f}\in L^\infty(\Sph^{n-1})$,
	\[\left\|e^{\widetilde{P} \widetilde{f}}\right\|_{L^{\frac{n}{n-1}}(\B^n)}\leq n^{-1}\omega_n^{-\frac{1}{n}}\left\|e^{\widetilde{f}}\right\|_{L^1(\Sph^{n-1})}.\]
	Here
	\begin{equation*}
\widetilde{P}\widetilde{f}(x)=\frac{1}{n\omega_n}\int_{\Sph^{n-1}}\frac{1-|x|^2}{|x-\xi|^n}\widetilde{f}(\xi)d\xi	
\end{equation*}
is the harmonic extension of $\widetilde{f}$, $\omega_{n}$ is the volume of the unit ball in $\R^n$ with the Euclidean metric. Moreover, equality holds if and only if $\widetilde{f}$ is constant.
\end{corollary}
Note that the inequality in the corollary also works for subharmonic functions but it is not invariant under conformal transformation.

In \cite{Ch} Chen proposed another way to generalize the Carleman's inequality in dimension $4$. He considered the extension inequality related to the equation studied by Caffarelli and Silvestre \cite{CS}, as a limit case of the inequality he proved the following:
\begin{corollary}\cite[Corollary 1]{Ch}
For any $u:\B^4\to\R$ satisfying $\Delta^2 u\leq 0$ and $-\frac{\partial u}{\partial \nu}\leq 1$, 
\[\left(\int_{\B^4} e^{4u}dx\right)^\frac{1}{4}\leq S\left(\int_{\Sph^3} e^{3u}d\xi\right)^{\frac{1}{3}}.\]
Note that here $\Delta$ is the Laplacian in $\B^4$ with the Euclidean metric, $\nu$ is the outer unit normal vector with respect to the Euclidean metric. Here $S$ is the sharp constant, and is assumed by the solution to the equation
\begin{equation}\label{ChenBihamonicEquation}
\begin{cases}
	\Delta^2 u=0,\text{ in }\B^4\\
	u=0,\text{ on }\Sph^3\\
	-\frac{\partial u}{\partial \nu}=1,\text{ on }\Sph^3.\\
\end{cases}	
\end{equation}
\end{corollary}
Note that Chen did not prove uniqueness of extremal function for this generalization, and as he pointed out at the end of \cite{Ch} that this generalization works well because the Green's function of equation (\ref{ChenBihamonicEquation}) is positive. As a result will be difficult for us to find similar generalizations in higher dimensions.

In this article, we propose another way to generalize the Carleman's inequality:
\begin{corollary}
Assume $n\geq 3$, then for any $\widetilde{f}\in L^\infty(\Sph^{n-1})$
\[\left\|e^{\widetilde{I}_n+\widetilde{P}_{2-n}\widetilde{f}}\right\|_{L^n(\B^n)}\leq S_n\left\|e^{\widetilde{f}}\right\|_{L^{n-1}(\Sph^{n-1})}.\]
The sharp constant
$S_n=\frac{\left\|e^{\widetilde{I}_n}\right\|_{L^n(\B^n)}}{\left|\Sph^{n-1}\right|^{\frac{1}{n-1}}}$. Moreover, equality holds if and only if $\widetilde{f}(\xi)=-\ln|1-\zeta\cdot \xi|+C$. Where $C\in\R$ is a constant and $\zeta\in\B^n$.
\end{corollary}
Note that this inequality is invariant under conformal transformation and that it also holds for hyperbolic subharmonic functions.

The function $\widetilde{I}_n$ shows up naturally in the proof. When $n$ is even, we can think of $\|e^{\widetilde{I}_n+\widetilde{P}_{2-n}\widetilde{f}}\|_{L^n(\B^n)}$ as the $L^n$ norm of $e^{\widetilde{P}_{2-n}\widetilde{f}}$ measured using the Fefferman-Graham metric 
\[g=e^{2\widetilde{I}_n}dx^2\]
where $dx^2$ is the standard Euclidean metric on the unit ball. For further information related to the Fefferman-Graham metric, we refer the reader to \cite{CC}, \cite{AC} and \cite{FG}.

\subsection{Notations}
 Through out this article, we let 
\[\R^n_+=\{(y',y_n)\in \R^n\text{ such that }y'\in \R^{n-1},\ y_n>0\},\]
and
\[\B^n=\{x\in \R^n\text{ such that } |x|<1\},\]
here $|x|$ denotes the norm of $x$ with respect to the Euclidean metric. We also use the notation
\[\Sph^{n-1}=\{x\in \R^n\text{ such that } |x|=1\}\]
to denote the unit sphere in $\R^{n}$.

Through out this article we use notations like $(y',y_n)$ and $(x',x_n)$ to denote points in $\R^n$ where $y',\ x'\in \R^{n-1}$ and $y_n,\ x_n\in\R$.

In \cite{CS} Caffarelli and Silvestre considered an interesting generalization of Laplace equation in the upper half space. They considered for $-1<\alpha<1$ the equation
\begin{equation}\label{divform}
\begin{cases}
\text{div}\left( y_{n}^{\alpha }\nabla u\right) =0, \text{ for }y\in \R^{n}_+,\\
u(y',0)=f(y'), \text{ for } y'\in \R^{n-1}.
\end{cases}
\end{equation}%
In this article we want to consider the case $2-n\leq \alpha <1$ and transform equation \ref{divform} from the upper half space to the unit ball.

Let $\Psi :\mathbb{R}_{+}^{n}\rightarrow \mathbb{B}^{n}$ be the projection map defined by%
\begin{equation}\label{PsiDefinition}
\begin{split}
\left( x^{\prime },x_{n}\right) &=\Psi \left( y^{\prime },y_{n}\right)
=\left( \frac{2y^{\prime }}{1+2y_{n}+\left\vert y\right\vert ^{2}},\frac{%
-1+\left\vert y\right\vert ^{2}}{1+2y_{n}+\left\vert y\right\vert ^{2}}%
\right) \\
&=\left( \frac{2y^{\prime }}{\left( 1+y_{n}\right) ^{2}+\left\vert
y^{\prime }\right\vert ^{2}},1-\frac{2\left( 1+y_{n}\right) }{\left(
1+y_{n}\right) ^{2}+\left\vert y^{\prime }\right\vert ^{2}}\right)
\end{split}
\end{equation}%
with the inverse $\Psi ^{-1}:\mathbb{B}^{n}\rightarrow \mathbb{R}_{+}^{n}$ 
\begin{equation*}
\left( y^{\prime },y_{n}\right) =\Psi ^{-1}\left( x^{\prime },x_{n}\right)
=\left( \frac{2x^{\prime }}{\left( 1-x_{n}\right) ^{2}+\left\vert x^{\prime
}\right\vert ^{2}},\frac{1-\left\vert x\right\vert ^{2}}{\left(
1-x_{n}\right) ^{2}+\left\vert x^{\prime }\right\vert ^{2}}\right) .
\end{equation*}%
It is useful to record 
\begin{equation*}
\frac{1-2x_{n}+\left\vert x\right\vert ^{2}}{2}=\frac{2}{1+2y_{n}+\left\vert
y\right\vert ^{2}}
\end{equation*}%
Which means that if we define $[\Psi(y)]_n$ to be the $n-$th component of $\Psi(y)$ (note that by definition $[\Psi(y)]_n=\frac{-1+|y|^2}{1+2y_n+|y|^2}$), then we have
\begin{equation}\label{EquationConformalFactorTransformation}
\frac{1-2[\Psi(y)]_n+\left\vert \Psi(y)\right\vert ^{2}}{2}=\frac{2}{1+2y_{n}+\left\vert
y\right\vert ^{2}}.
\end{equation}

The restriction of $\Psi $ on $y_{n}=0$ is the stereographic projection $%
\mathbb{R}^{n-1}\rightarrow \mathbb{S}^{n-1}$. From the calculation in Proposition \ref{HowTheOperatorTranforms}, in particular (\ref{ConformalFactorPsi}), we see that 
\begin{equation*}
\Psi ^{\ast }dx^{2}=\frac{4dy^{2}}{\left( 1+2y_{n}+\left\vert y\right\vert
^{2}\right) ^{2}}.
\end{equation*}
It means that $\Psi :\R^n_+ \to \B^n$ is a conformal transformation. Here the conformal factor is very important for our calculation, through out this article we will use $|\Psi'(y)|$ to denote the conformal factor, in particular for any $y\in \R^n_+$ we have
\begin{equation}\label{PsiConformalFactorExampleUpperHalf}
|\Psi'(y)|= \frac{2}{1+2y_n+|y|^2},	
\end{equation}
and for any $w\in \R^{n-1}=\partial \R^n_+$ we have
\begin{equation}\label{PsiConformalFactorExampleBoundary}
|\Psi'(w)|=\frac{2}{1+|w|^2}.	
\end{equation}

For a function $\widetilde{f}$ on $\mathbb{B}^{n}$ or $\mathbb{S}^{n-1}$ we
define 
\begin{equation}\label{PsiTransformation}
f\left( y^{\prime },y_{n}\right) =\widetilde{f}\circ \Psi \left( y^{\prime
},y_{n}\right) \left( \frac{2}{1+2y_{n}+\left\vert y\right\vert ^{2}}\right)
^{\frac{n-2+\alpha }{2}}.
\end{equation}%
It is easy to check that this map is an isometry from $L^{\frac{2\left(
n-1\right) }{n-2+\alpha }}\left( \mathbb{S}^{n-1}\right) $ to $L^{\frac{%
2\left( n-1\right) }{n-2+\alpha }}\left( \mathbb{R}^{n-1}\right) $ and from $%
L^{\frac{2n}{n-2+\alpha }}\left( \mathbb{B}^{n}\right) $ to $L^{\frac{2n}{%
n-2+\alpha }}\left( \mathbb{R}_{+}^{n}\right) $. The inverse map is 
\begin{equation}\label{PsiTransformationInverse}
\widetilde{f}\left( x^{\prime },x_{n}\right) =f\circ \Psi ^{-1}\left(
x^{\prime },x_{n}\right) \left( \frac{2}{1-2x_{n}+\left\vert x\right\vert
^{2}}\right) ^{\frac{n-2+\alpha }{2}}.
\end{equation}  

In the limit case when $\alpha=2-n$ we still use notations $\widetilde{f}$ and $f$ to denote functions on $\Sph^{n-1}$ and $\R^{n-1}$ respectively, but the relation between them is different. Given any $\widetilde{f}: \Sph^{n-1}\to \R$, such that $e^{\widetilde{f}}\in L^{n-1}(\Sph^{n-1})$, define
\begin{equation}\label{flimitdef}
f(w)=\widetilde{f}\circ\Psi(w)+\ln|\Psi'(w)|	,
\end{equation}
then it is easy to see that we have $\left\|e^{\widetilde f}\right\|_{L^{n-1}(\Sph^{n-1})}=\left\|e^{f}\right\|_{L^{n-1}(\R^{n-1})}$.

\subsection{Main Results}
In this article, we revisit the extension problem studied in \cite{Ch}. We derive an explicit formula for $\PKA$ in (\ref{ExtensionFormulaAlpha}) and then carry out the analysis on $\B^n$.

In Theorem \ref{maininequality}, we prove that the following inequality has constant function as optimizers.

\begin{theorem}\label{MainInequalityTian}
	Assume $n\geq 3$ and $\alpha\in (2-n,1)$. For every $f\in L^{\frac{2(n-1)}{n-2+\alpha}}(\Sph^{n-1})$, we have
	\begin{equation*}
		\left\|\PKA f\right\|_{L^{\frac{2n}{n-2+\alpha}}(\B^n)}\leq S_{n,\alpha}\left\|f\right\|_{L^{\frac{2(n-1)}{n-2+\alpha}}(\Sph^{n-1})}.
	\end{equation*}
	Where $S_{n,\alpha}$ is a constant that only depends on $n$ and $\alpha$.  Up to conformal transformation any constant is an optimizer..
\end{theorem}

Our proof of the existence of optimizer uses subcritical analysis which is similar to the proof in \cite{HWY2}. Note that we do not proof uniqueness in this theorem. Uniqueness is proved in \cite{Ch}.

In the limit case $\alpha\to 2-n$. We prove

\begin{theorem}\label{LimitCaseInequalityTian}
For dimension $n\geq 2$, and any function $\widetilde{F}\in L^\infty(\Sph^{n-1})$ we have
\begin{equation}\label{LimitCaseMainInequality}
\left\|e^{\widetilde{I}_n+\widetilde{P}_{2-n}\widetilde{F}}\right\|_{L^n(\B^n)}\leq S_n \left\|e^{\widetilde{F}}\right\|_{L^{n-1}(\Sph^{n-1})}.	
\end{equation}

Where $\widetilde{I}_n(x)=2\frac{d\widetilde{P}_\alpha 1}{d\alpha}\bigr|_{\alpha=2-n}$. When $n$ is even we have
\[\widetilde{I}_n(x)=\sum_{k=1}^{n/2-1} \frac{1}{2k}\cdot \frac{\Gamma\left(\frac{n-2}{2}\right)\Gamma\left(n-k-1\right)}{\Gamma(n-2)\Gamma\left(\frac{n}{2}-k\right)} (1-|x|^2)^k.\]
The sharp constant
$S_n=\frac{\left\|e^{\widetilde{I}_n}\right\|_{L^n(\B^n)}}{\left|\Sph^{n-1}\right|^{\frac{1}{n-1}}}$
\end{theorem}
Our proof of the limit case inequality is very similar to that of \cite{Ch}. What is different is that we found a very useful induction relation concerning the function $\widetilde{I}_n$. The induction relation is proved in Lemma \ref{InInductionLemma}. In the case when $n$ is even we found an explicit formula for the function $\widetilde{I}_n$ using the induction relation.

We also consider the variational problem
\[S_n=\sup\left\{\left\|e^{\widetilde{I}_n+\widetilde{P}_{2-n}\widetilde{f}}\right\|: \widetilde{f}\in L^{\infty}(\Sph^{n-1}), \left\|e^{\widetilde{f}}\right\|_{L^{n-1}(\Sph^{n-1})}=1 \right\},\]
and derived the Euler Lagrange equation
\[e^{(n-1)\widetilde{f}(\xi)}=\int_{\B^n}e^{n\widetilde{I}_n +n\widetilde{P}_{2-n}\widetilde{f}}\widetilde{p}_{2-n}(x,\xi)dx.\]
We prove the following uniqueness result
\begin{theorem}\label{MainTheoremLimitCaseUniqueness}
For any integer $n\geq 2	$, if $\widetilde{f}\in L^\infty(\Sph^{n-1})$ satisfies the equation
\[e^{(n-1)\widetilde{f}(\xi)}=\int_{\B^n}e^{n\widetilde{I}_n +n\widetilde{P}_{2-n}\widetilde{f}}\widetilde{p}_{2-n}(x,\xi)dx,\]
then for all $\xi\in\Sph^{n-1}$
\[\widetilde{f}(\xi)=\ln \frac{1-|\zeta|^2}{|\xi-\zeta|^2}+C_n,\]
where $\zeta \in\B^{n}$ and $C_n=-\frac{1}{n-1} \ln \left|\Sph^{n-1}\right|$ is a constant. Here $\left|\Sph^{n-1}\right|$ denotes the volume of the standard sphere.
\end{theorem}

This uniqueness result in the limit case is new; the proof uses the moving sphere method.

The method of moving spheres is a powerful tool to prove uniqueness of solutions to equations that have conformal symmetry. The method relies on maximum principle and the conformal symmetry of the equation. The moving sphere method was invented by Li and Zhu in \cite{LZ}. For further information related to the moving sphere method we refer the reader to \cite{L} and \cite{FKT}. The method of moving spheres can be considered as a powerful generalization of the method of moving planes. For more information about the method of moving planes we refer the readers to the  articles \cite{ChenLiOu}, \cite{GNN} and \cite{GuoWang}.

This article is organized as follows: in section \ref{PoissonKernelUnitBall} we transform the equation studied by Caffarelli and Silvestre \cite{CS} from the upper half space to the unit ball. We also identify the Poisson kernel of the corresponding equation in the unit ball and study how the Poisson kernel transforms under conformal transformation of the unit ball. In section \ref{FamilyOfInequalities} we prove a family of conformally invariant extension inequalities. Note that this is the same result as in Chen's work \cite{Ch}. For the proof, we use the method from \cite{HWY2} which is different from Chen's proof. In section \ref{LimitCaseInequality} we take limit $\alpha\to 2-n$ to obtain a limit case inequality. Our proof of the limit case inequality is very similar to Chen's proof in \cite{Ch}, but we use a slightly better way to estimate the extension of constant function. We also prove important results about the function $\widetilde{I}_n$ in section \ref{LimitCaseInequality}, these results are used in section \ref{UniquenessThroughMovingSpheres} to establish uniqueness of the limit case inequality. In the proof of uniqueness, we used the method of moving spheres.

\section{The Poisson Kernel in the Unit Ball}\label{PoissonKernelUnitBall}
In this chapter we transform the equation (\ref{divform}) from the upper half space to the unit ball. We also identify the Poisson kernel of the corresponding equation in the unit ball and study how the Poisson kernel transforms under conformal changes.

 \subsection{The Equation in the Unit Ball}
Now we are ready to transform the equation (\ref{divform}) from the upper half space to the unit ball. For any $2-n\leq\alpha<1$ define the operator $\mathcal{L}$ in $\B^n$ such that for any $x=(x',x_n)\in\B^n$ and for any $\widetilde{u}\in C^2(\B^n)$
\begin{eqnarray*}
\mathcal{L}\widetilde{u} &=&\left( \frac{1-2x_{n}+\left\vert x\right\vert ^{2}}{2}\right)
^{\left( n+2-\alpha \right) /2} \\ 
& &\cdot\left[ div\left[ \left( \frac{1-\left\vert
x\right\vert ^{2}}{2}\right) ^{\alpha }\nabla \widetilde{u}\right] +\frac{%
\alpha \left( 2-n-\alpha \right) }{2}\left( \frac{1-\left\vert x\right\vert
^{2}}{2}\right) ^{\alpha -1}\widetilde{u}\right].
\end{eqnarray*}%
Note that in $\B^n$ we have $\mathcal{L}\widetilde{u}=0$  if and only if
\begin{equation}\label{EquationBall}
div\left[ \left( \frac{1-\left\vert x\right\vert ^{2}}{2}\right) ^{\alpha
}\nabla \widetilde{u}\right] +\frac{\alpha \left( 2-n-\alpha \right) }{2}%
\left( \frac{1-\left\vert x\right\vert ^{2}}{2}\right) ^{\alpha -1}%
\widetilde{u}=0.
\end{equation}

\begin{proposition}\label{HowTheOperatorTranforms}
\bigskip (How the operator tranforms) For any $2-n\leq\alpha<1$ and any $u\in C^2(\R^n_+)$ define $\widetilde{u}$ using (\ref{PsiTransformationInverse}) then we have
\[div(y_n^\alpha\nabla u)=0,\text{ in }\R^n_+\]
if and only
\[\mathcal{L}\widetilde{u}=0,\text{ in }\B^n\]
\end{proposition}

\begin{proof}
In the following, for any $y=(y',y_n)\in \R^n$, let $\rho =\left( 1+2y_{n}+\left\vert y\right\vert
^{2}\right) /2$. Then we have $\nabla \rho =\left( y^{\prime },1+y_{n}\right) $. Let $a,\ b,\ c,\ d=1,\ 2,\ ...,\ n$ be indices. Suppose $x=\Psi(y)$, then by direct calculation we have, when $c\neq n$,
\begin{eqnarray*}
	\frac{\partial x_{c}}{\partial y_{a}}
&=&\left\{ 
\begin{array}{cc}
-\frac{4y_a y_c}{(1+2y_n+|y|^2)^2}, & a\neq c \text{ and } a\neq n, \\ 
\frac{2}{1+2y_n+|y|^2}-\frac{4y_c^2}{(1+2y_n+|y|^2)^2}, & c=a\neq n,\\
-\frac{4y_c(y_n+1)}{(1+2y_n+|y|^2)^2}, & a=n,
\end{array}%
\right.  \\
\end{eqnarray*}
and when $c=n$
\begin{eqnarray*}
\frac{\partial x_{n}}{\partial y_{a}}
&=&\left\{ 
\begin{array}{cc}
\frac{2y_a}{1+2y_n+|y|^2}+\frac{2y_a(1-|y|^2)}{(1+2y_n+|y|^2)^2}, & a\neq n, \\ 
\frac{2y_n}{1+2y_n+|y|^2}+\frac{2(1-|y|^2)(1+y_n)}{(1+2y_n+|y|^2)^2}, & a=n.\\
\end{array}%
\right.  \\
\end{eqnarray*}
From it we have
\begin{eqnarray}\label{ConformalFactorPsi}
\sum_{a=1}^n\frac{\partial x_{c}}{\partial y_{a}} \frac{\partial x_{d}}{\partial y_{a}} &=&\left\{ 
\begin{array}{cc}
\rho^{-2}, & c=d, \\ 
0, & c\neq d.%
\end{array}%
\right.	
\end{eqnarray}

\begin{eqnarray*}
\sum_{a=1}^n\frac{\partial x_{c}}{\partial y_{a}}\frac{\partial \rho }{\partial y_{a}}
&=&\left\{ 
\begin{array}{cc}
-\rho ^{-1}y_{c}, & c\neq n, \\ 
\rho ^{-1}\left( 1+y_{n}\right) & c=n.%
\end{array}%
\right. , \\
\sum_{a=1}^n\frac{\partial x_{c}}{\partial y_{a}^{2}} &=&\left\{ 
\begin{array}{cc}
-\left( n-2\right) \rho ^{-2}y_{c}, & c\neq n, \\ 
\left( n-2\right) \rho ^{-2}\left( 1+y_{n}\right) & c=n.%
\end{array}%
\right.
\end{eqnarray*}%
We calculate 
\begin{equation*}
\frac{\partial u}{\partial y_{a}}=\rho ^{-\left( n+\alpha \right) /2}\left[
\rho \frac{\partial \widetilde{u}}{\partial x_{c}}\frac{\partial x_{c}}{%
\partial y_{a}}+\left( 1-\frac{n+\alpha }{2}\right) \widetilde{u}\frac{%
\partial \rho }{\partial y_{a}}\right] ,
\end{equation*}%
\begin{eqnarray*}
div (y_n^\alpha \nabla u) &=&\frac{\partial }{\partial y_{a}}\left( y_{n}^{\alpha }\frac{%
\partial u}{\partial y_{a}}\right) \\
&=&y_{n}^{\alpha }\rho ^{-\left( n+\alpha \right) /2}\bigg[ \rho \frac{%
\partial ^{2}\widetilde{u}}{\partial x_{c}\partial x_{d}}\frac{\partial x_{c}%
}{\partial y_{a}}\frac{\partial x_{d}}{\partial y_{a}}+\rho \frac{\partial 
\widetilde{u}}{\partial x_{c}}\frac{\partial ^{2}x_{c}}{\partial y_{a}^{2}} \\
&&+\left( 2-\frac{n+\alpha }{2}\right) \frac{\partial \rho }{\partial y_{a}}%
\frac{\partial \widetilde{u}}{\partial x_{c}}\frac{\partial x_{c}}{\partial
y_{a}}+n\left( 1-\frac{n+\alpha }{2}\right) \widetilde{u}\bigg] \\
&&-\frac{n+\alpha }{2}y_{n}^{\alpha }\rho ^{-\left( n+\alpha \right) /2-1}%
\frac{\partial \rho }{\partial y_{a}}\left[ \rho \frac{\partial \widetilde{u}%
}{\partial x_{c}}\frac{\partial x_{c}}{\partial y_{a}}+\left( 1-\frac{%
n+\alpha }{2}\right) \widetilde{u}\frac{\partial \rho }{\partial y_{a}}%
\right] \\
&&+\alpha y_{n}^{\alpha -1}\rho ^{-\left( n+\alpha \right) /2}\left[
\rho \frac{\partial \widetilde{u}}{\partial x_{c}}\frac{\partial x_{c}}{%
\partial y_{n}}+\left( 1-\frac{n+\alpha }{2}\right) \widetilde{u}\frac{%
\partial \rho }{\partial y_{n}}\right] \\
&=&y_{n}^{\alpha }\rho ^{-\left( n+\alpha \right) /2}\bigg[ \rho ^{-1}\Delta 
\widetilde{u}+\rho  \frac{\partial \widetilde{u}}{\partial x_{c}}\frac{%
\partial ^{2}x_{c}}{\partial y_{a}^{2}} \\ 
&&+\left( 2-n-\alpha \right) \rho
^{-1}\left( \frac{\partial \widetilde{u}}{\partial x_{n}}\left(
1+y_{n}\right) -\frac{\partial \widetilde{u}}{\partial x_{j}}y_{j}\right) +n\left( 1-\frac{n+\alpha }{2}\right) \widetilde{u}\bigg] \\
&&-\left( n+\alpha \right) \left( 1-\frac{n+\alpha }{2}\right) y_{n}^{\alpha
}\rho ^{-\left( n+\alpha \right) /2}\widetilde{u} \\
&&+\alpha y_{n}^{\alpha
-1}\rho ^{-\left( n+\alpha \right) /2}\left[ \rho \frac{\partial \widetilde{u%
}}{\partial x_{c}}\frac{\partial x_{c}}{\partial y_{n}}+\left( 1-\frac{%
n+\alpha }{2}\right) \widetilde{u}\frac{\partial \rho }{\partial y_{n}}%
\right] \\
&=&y_{n}^{\alpha }\rho ^{-\left( n+\alpha \right) /2} \\
&&\cdot\left[ \rho ^{-1}\Delta 
\widetilde{u}-\alpha \rho ^{-1}\left( \frac{\partial \widetilde{u}}{\partial
x_{n}}\left( 1+y_{n}\right) -\frac{\partial \widetilde{u}}{\partial x_{j}}%
y_{j}\right) -\alpha \left( 1-\frac{n+\alpha }{2}\right) \widetilde{u}\right]
\\
&&+\alpha y_{n}^{\alpha -1}\rho ^{-\left( n+\alpha \right) /2}\bigg[ \rho
^{-1}\left( 1+y_{n}\right) \left( \frac{\partial \widetilde{u}}{\partial
x_{n}}\left( 1+y_{n}\right) -\frac{\partial \widetilde{u}}{\partial x_{j}}%
y_{j}\right) \\
&&-\frac{\partial \widetilde{u}}{\partial x_{n}}+\left( 1-\frac{%
n+\alpha }{2}\right) \widetilde{u}\left( 1+y_{n}\right) \bigg] \\
&=&y_{n}^{\alpha }\rho ^{-\left( n+\alpha \right) /2}\rho ^{-1}\Delta 
\widetilde{u}+\alpha y_{n}^{\alpha -1}\rho ^{-\left( n+\alpha \right) /2}%
\\
&&\cdot\bigg[ \rho ^{-1}\left( \frac{\partial \widetilde{u}}{\partial x_{n}}\left(
1+y_{n}\right)  -\frac{\partial \widetilde{u}}{\partial x_{j}}y_{j}\right)-\frac{\partial \widetilde{u}}{\partial x_{n}}+\left( 1-\frac{n+\alpha }{2}%
\right) \widetilde{u}\bigg] \\
&=&y_{n}^{\alpha -1}\rho ^{-\left( n+\alpha \right) /2}\left[ \frac{%
1-\left\vert x\right\vert ^{2}}{2}\Delta \widetilde{u}-\alpha x_{a}\frac{%
\partial \widetilde{u}}{\partial x_{a}}+\frac{\alpha \left( 2-n-\alpha
\right) }{2}\widetilde{u}\right] \\
&=&\left( \frac{1-\left\vert x\right\vert ^{2}}{2}\right) ^{\alpha -1}\left( 
\frac{1-2x_{n}+\left\vert x\right\vert ^{2}}{2}\right) ^{\left( n+2-\alpha
\right) /2} \\
& &\cdot \left[ \frac{1-\left\vert x\right\vert ^{2}}{2}\Delta \widetilde{u}-\alpha x_{a}\frac{\partial \widetilde{u}}{\partial x_{a}}+\frac{\alpha
\left( 2-n-\alpha \right) }{2}\widetilde{u}\right]
\end{eqnarray*}
\end{proof}

\begin{remark}
For any integer $n\geq 2$ and any $\alpha\in (2-n,0)$, we can apply theorem 1.1 in \cite{WZ} to show that solution to the equation
\[
\begin{cases}
	\mathcal{L}u=0, \text{ in }\B^n\\
	u=f, \text{ in }\Sph^{n-1}
\end{cases}	
\]
	is unique in $C^2(\B^n)\cap C^0(\overline{\B^n})$.
\end{remark}

\begin{remark}\label{EquivalenceCG}
Note that the equation (\ref{EquationBall}) is a special case of equation (4.2)
 in \cite{CG} if we take $\overline{g}$ as the Euclidean metric in the unit ball, $g^+$ as the hyperbolic metric in the unit ball, and $\rho=\frac{1-|x|^2}{2}$ as the defining function.
 \end{remark}

\subsection{Poisson Kernel in the Unit Ball}
Caffarelli and Silvestre \cite{CS} found a Poisson kernel that solves the Dirichlet problem (\ref{divform}). For any $-1<\alpha<1$, $y\in \R^n_+$ and any $\xi\in \R^{n-1}$  
\begin{equation}\label{PoissonKernelUpperHalfSpace}
\upka(y,w)=\cna \frac{y_n^{1-\alpha}}{|y-w|^{\frac{n-\alpha}{2}}}.
\end{equation} 
Here $\cna$ is the constant given in (\ref{cna}).
For any $f:\R^{n-1}\to \R$ regular enough, we can define
\[P_\alpha f=\int_{\R^{n-1}}p_{\alpha}(y,\xi)f(w)dw,\]
such that $P_\alpha f$ solves the Dirichlet problem (\ref{divform}).

We want to find the corresponding Poisson kernel in the unit ball. For any $\alpha\in[2-n,1)$ define
\begin{equation}\label{PoissonKernelUnitBall}
\pka(x,\xi)=2^{\alpha-1}\cna \frac{(1-|x|^2)^{1-\alpha}}{|x-\xi|^{n-\alpha}},
\end{equation}

here
\begin{equation}\label{cna}
c_{n,\alpha }^{-1} =\left\vert \mathbb{S}^{n-2}\right\vert \int_{0}^{\infty }%
\frac{r^{n-2}dr}{\left( 1+r^{2}\right) ^{\frac{n-\alpha }{2}}}
=\frac{\Gamma (\frac{1-\alpha }{2})\Gamma (\frac{n-1}{2})}{2\Gamma (\frac{%
n-\alpha }{2})}\left\vert \mathbb{S}^{n-2}\right\vert.
\end{equation}
For any function $\widetilde{\phi}:\Sph^{n-1}\to \R$, define
\begin{equation}\label{ExtensionFormulaAlpha}
\PKA \widetilde{\phi}(x)=\int_{\Sph^{n-1}} \pka (x,\xi)\widetilde{\phi}(\xi)d\xi.	
\end{equation}

Then we have the following proposition:

\begin{proposition}
	For any integer $n\geq 2$, any $\alpha\in (2-n,1)$, any $f\in L^{\frac{2(n-1)}{n-2+\alpha}}(\R^{n-1}) $, define $\widetilde{f}$ as in (\ref{PsiTransformationInverse}), then we have
	\[\widetilde{\UPKA f}= \PKA \widetilde{f}.\]
	Here $\widetilde{\UPKA f}$ is the transformation of $\UPKA f$ as defined in 
	(\ref{PsiTransformationInverse}), and $\PKA \widetilde{f}$ is the extension of $\widetilde{f}$ in the unit ball as defined in (\ref{ExtensionFormulaAlpha}).
\end{proposition}

\begin{proof}
The proof is by direct calculation. Note that for any $w\in \R^{n-1}$ using the fact that (\ref{PsiTransformation}) and (\ref{PsiTransformationInverse}) are inverse to each other, we have
\[f(w)=\widetilde{f}\circ \Psi (w) \left(\frac{2}{1+|w|^2}\right)^{\frac{n-2+\alpha}{2}}.\]
As a result, we have
\begin{eqnarray*}
&&\left( P_{\alpha }f\right) \circ \Psi ^{-1}\left( x^{\prime },x_{n}\right)
\\
&=&c_{n,\alpha }\int_{\mathbb{R}^{n-1}}\frac{\left( \frac{1-\left\vert
x\right\vert ^{2}}{1-2x_{n}+\left\vert x\right\vert ^{2}}\right) ^{1-\alpha }%
}{\left( \left\vert \frac{2x^{\prime }}{\left( 1-x_{n}\right)
^{2}+\left\vert x^{\prime }\right\vert ^{2}}-w\right\vert ^{2}+\left( \frac{%
1-\left\vert x\right\vert ^{2}}{1-2x_{n}+\left\vert x\right\vert ^{2}}%
\right) ^{2}\right) ^{\frac{n-\alpha }{2}}} \\
&&\cdot\widetilde{f}\circ \Psi \left(
w\right) \left( \frac{2}{1+\left\vert w\right\vert ^{2}}\right) ^{\frac{%
n-2+\alpha }{2}}dw \\
&=&c_{n,\alpha }\int_{\mathbb{S}^{n-1}}\frac{\left( \frac{1-\left\vert
x\right\vert ^{2}}{1-2x_{n}+\left\vert x\right\vert ^{2}}\right) ^{1-\alpha }%
}{\left( \left\vert \frac{2x^{\prime }}{\left( 1-x_{n}\right)
^{2}+\left\vert x^{\prime }\right\vert ^{2}}-\frac{\xi ^{\prime }}{1-\xi _{n}%
}\right\vert ^{2}+\left( \frac{1-\left\vert x\right\vert ^{2}}{%
1-2x_{n}+\left\vert x\right\vert ^{2}}\right) ^{2}\right) ^{\frac{n-\alpha }{%
2}}} \\
&&\cdot\widetilde{f}\left( \xi \right) \left( 1-\xi _{n}\right) ^{\frac{\alpha
-n}{2}}d\xi \\
&=&2^{\left( \alpha -n\right) /2}c_{n,\alpha }\left( 1-2x_{n}+\left\vert
x\right\vert ^{2}\right) ^{\frac{n-2+\alpha }{2}}\int_{\mathbb{S}^{n-1}}%
\frac{\left( 1-\left\vert x\right\vert ^{2}\right) ^{1-\alpha }}{\left\vert
x-\xi \right\vert ^{n-\alpha }}\widetilde{f}\left( \xi \right) d\xi.%
\end{eqnarray*}%
Divide both sides by $\left(\frac{1-2x_n+|x|^2}{2}\right)^{\frac{n-2+\alpha}{2}}$ then we are done.
\end{proof}

\begin{remark}\label{RemarkExtensionSpecialCase}
For any integer $n\geq 2$ and $\alpha=2-n$ we can prove similar result. For any $f\in L^{\infty}(\R^{n-1})$, define
\[\widetilde{f}=f\circ \Psi^{-1},\]
then the same calculation as in the previous proposition show that
\[\left(P_{2-n}f\right)\circ\Psi^{-1}=\widetilde{P}_{2-n}\left(f\circ \Psi^{-1}\right)=\widetilde{P}_{2-n}\widetilde{f}.\]
\end{remark}

\begin{remark}
We note that
\begin{equation*}
\int_{\mathbb{S}^{n-1}}\widetilde{p}_{\alpha }\left( x,\xi \right) d\xi
\end{equation*}%
is not constant in $x$ except when $\alpha =0$ or $2-n$.
\end{remark}

\begin{proposition}\label{alphahamonic}
For any $\alpha \in [2-n,1)$ and for any $f\in C\left( \mathbb{S}^{n-1}\right) $%
\begin{eqnarray*}
u\left( x\right) :=
\begin{cases}
\int_{\mathbb{S}^{n-1}}\widetilde{p}_{\alpha }\left(
x,\xi \right) f\left( \xi \right) d\xi , x\in\B^n\\
f(x) , x\in\Sph^{n-1}
\end{cases}
\end{eqnarray*}%
defines a continuous function on $\overline{\mathbb{B}^{n}}$ which is smooth
in $\mathbb{B}^{n}$ and satisfies $\mathcal{L}u=0$.
\end{proposition}

\begin{proof}
The integral $\int_{\Sph^{n-1}}\pka (x,\xi)d\xi=\frac{1}{r^{n-1}}\int_{\Sph^{n-1}_r}\pka (x,\xi)dx$ is a function that only depends on $|x|$. Define $h(|x|)=\int_{\Sph^{n-1}}\pka (x,\xi)d\xi$, then by Lemma \ref{integration}, and remark \ref{lowerbound} we know that for $r\in [0,1]$
\[\left(\frac{2}{1+r}\right)^{n-2+\alpha}\frac{\Gamma\left(\frac{n-\alpha}{2}\right)\Gamma\left(\frac{n-1}{2}\right)}{\Gamma(n-1)\Gamma\left(\frac{1-\alpha}{2}\right)}\leq h(r)\leq \left(\frac{2}{1+r}\right)^{n-2+\alpha} \leq 2^{n-2+\alpha} .\]

By dominated convergence theorem as in Remark \ref{DCT}, we know that for $r\in [0,1]$, $h(r)$ is continuous and that
\[\lim_{r\to 1}h(r)=1.\]

By the continuity of $f$ on $\Sph^{n-1}$, we can choose $\delta>0$ small, such that when $|\xi_1-\xi_2|\leq \delta$, we have $|f(\xi_1)-f(\xi_2)|\leq \epsilon$.  By the continuity of $h(r)$ on the interval $[0,1]$ and the fact that it is strictly positive on $[0,1]$ we can choose $\delta>0$  smaller if needed, such that when $|x-x_0|\leq \delta$ we have $\left\lvert\frac{1}{h(|x|)}-\frac{1}{h(|x_0|)}\right\rvert<\epsilon$. Define
\begin{equation*}
	M:=\|f\|_{L^\infty(\Sph^{n-1})}\frac{\Gamma(n-1)\Gamma\left(\frac{1-\alpha}{2}\right)}{\Gamma\left(\frac{n-\alpha}{2}\right)\Gamma\left(\frac{n-1}{2}\right)}.
\end{equation*}
Note that we have
\[0<\frac{\Gamma\left(\frac{n-\alpha}{2}\right)\Gamma\left(\frac{n-1}{2}\right)}{\Gamma(n-1)\Gamma\left(\frac{1-\alpha}{2}\right)}\leq h(r),\]
for all $r\in [0,1]$.

Suppose $x_0\in \Sph^{n-1}$, and $x\in \B^n$ such that $|x-x_0|<\delta/2$ consider
\begin{eqnarray*}
|u(x)-u(x_0)|&=&\left\lvert\int_{\mathbb{S}^{n-1}}\widetilde{p}_{\alpha }\left(
x,\xi \right) f\left( \xi \right) d\xi-\int_{\mathbb{S}^{n-1}}\widetilde{p}_{\alpha }\left(
x,\xi \right) \frac{u(x_0)}{h(|x|)} d\xi\right\rvert	 \\
&\leq& \int_{|\xi-x_0|\leq \delta} \pka(x,\xi)\left (\left\lvert f(\xi)-f(x_0)\right\rvert+\left\lvert f(x_0)-\frac{f(x_0)}{h(|x|)}\right\rvert \right)d\xi
 \\ &&+\int_{|\xi-x_0|> \delta}\pka(x,\xi)\left\lvert f(\xi)-\frac{f(x_0)}{h(|x|)}\right\rvert d\xi\\
 &\leq & C(n,\alpha)\epsilon+\frac{C(n,\alpha)M (1-|x|^2)^{1-\alpha}}{\delta^{n-\alpha}}
\end{eqnarray*}
As a result $u(x)$ is continuous at $x_0$. 

The part that $\mathcal{L}u=0$ follows from dominated convergence theorem and direct calculation.
\end{proof}

Based on the Martin theory for harmonic functions, we make the following conjecture:

\begin{conjecture}
(The representation theorem) Let $\widetilde{u}:\B^n\to \R$ be a positive solution of

\[\mathcal{L} \widetilde{u}=0.\]
Then there exists a Borel measure $\nu $ on $\mathbb{S}^{n-1}$ s.t.%
\begin{equation*}
\widetilde{f}\left( x\right) =\int_{\mathbb{S}^{n-1}}\widetilde{p}_{\alpha
}\left( x,\xi \right) d\nu \left( \xi \right) .
\end{equation*}
\end{conjecture}

\subsection{Poisson Kernel under Conformal Transformation}
 We want to know how the Poisson kernel transforms under conform transformation. We prove the following:
\begin{proposition}\label{PropositionPoissonKernelConformalTransformationPsi}
For any integer $n\geq 2$, any $\alpha\in [2-n,1)$, any $y\in\R^n_+$ and any $w\in \R^{n-1}$ we have
\begin{equation}\label{pkPsi}
\pka(\Psi(y),\Psi(w))=p_\alpha(y,w)|\Psi'(y)|^{(2-n-\alpha)/2}|\Psi'(w)|^{(\alpha-n)/2}.
\end{equation}
Here $\Psi$ is the conformal transformation defined in (\ref{PsiDefinition}), $|\Psi'(y)|$ and $|\Psi'(w)|$ are the conformal factors in (\ref{PsiConformalFactorExampleUpperHalf}) and (\ref{PsiConformalFactorExampleBoundary}) respectively.
\end{proposition}

\begin{proof}
For any $x\in \B^n$ and any $\xi\in \Sph^{n-1}$, by definition we have
\[\pka (x,\xi)=2^{\alpha-1}\cna \frac{(1-|x|^2)^{1-\alpha}}{|x-\xi|^{n-\alpha}}.\]	
From this we have for any $y\in\R^n_+$ and any $w\in \R^{n-1}$
\[\pka(\Psi(y),\Psi(w))=2^{\alpha-1}\cna \frac{(1-|\Psi(y)|^2)^{1-\alpha}}{|\Psi(y)-\Psi(w)|^{n-\alpha}}.\]
Through direct calculation we have
\[1-|\Psi(y)|^2=\frac{4y_n}{1+2y_n+|y|^2}=2y_n|\Psi'(y)|\]
and
\begin{eqnarray*}
|\Psi(y)-\Psi(w)|^2 &=& \left|\frac{2}{1+2y_n+|y|^2}(y',-y_n-1)-\frac{2}{1+|w|^2}(w,-1)\right|^2\\
&=&\frac{4}{(1+|w|^2)(1+2y_n+|y|^2)} |y-w|^2\\
&=&|y-w|^2 |\Psi'(w)||\Psi'(y)|
\end{eqnarray*}
Note that here we use the notation $(y',-y_n-1)$ and $(w,-1)$ to denote points in $\R^n$ and use the notation $\langle w,y'\rangle$ to denote the Euclidean inner product in $\R^{n-1}$. Combine these calculations together, then we can get (\ref{pkPsi}).
\end{proof}

We also want to consider how the Poisson kernel transform under the isometry group of $(\B^n,g_h)$. Here $g_h=\frac{4}{(1-|x|^2)}dx^2$ denotes the hyperbolic metric in the unit ball. We use the notation $SO(n,1)$ to denote the isometry group of the unit ball with the hyperbolic metric. For any $\Phi\in SO(n,1)$ any $x\in\B^n$ and any $\xi\in \Sph^{n-1}$ we use $|\Phi'(x)|$ and $|\Psi'(\xi)|$ to denote the conformal factors in $\B^n$ and $\Sph^{n-1}$ respectively. We prove the following:

\begin{proposition} For any integer $n\geq 2$ any $\alpha \in [2-n,1)$, any $x\in \B^n$, any $\xi\in \Sph^{n-1}$ and any $\Phi\in SO(n,1)$ we have
\begin{equation}\label{pkPhi}
\widetilde{p}_{\alpha }\left( \Phi \left( x\right) ,\Phi \left( \xi \right)
\right) =\widetilde{p}_{\alpha }\left( x,\xi \right) \left\vert \Phi
^{\prime }\left( x\right) \right\vert ^{\left( 2-n-\alpha \right)
/2}\left\vert \Phi ^{\prime }\left( \xi \right) \right\vert ^{\left( \alpha
-n\right) /2}.
\end{equation}
\end{proposition}
\begin{proof}
	For any $\Phi \in SO\left( n,1\right)$, since it is an isometry of $\B^n$ with the hyperoblic metric $g_{h}=\frac{4}{\left(
1-\left\vert x\right\vert ^{2}\right) ^{2}}dx^{2}$, it is
a conformal transformation with respect to the Euclidean metric. We have for any $x\in \B^n$
\[\Phi^\ast \left(\frac{4}{(1-|x|^2)^2}dx^2\right)=\frac{4}{(1-|\Phi(x)|^2)^2}\Phi^\ast (dx^2)=\frac{4}{(1-|x|^2)^2}dx^2,\]
and
\[\Phi^\ast dx^2=|\Phi'(x)|^2dx^2. \]
Here $|\Phi'(x)|$ is the conformal factor, it is a notation similar to $|\Psi'(y)|$. From this we conclude that for any $x\in \B^n$
\begin{equation}\label{PhiPrime}
\left\vert \Phi ^{\prime }\left( x\right) \right\vert =\frac{1-\left\vert
\Phi \left( x\right) \right\vert ^{2}}{1-\left\vert x\right\vert ^{2}}.
\end{equation}

For any $x,\ z\in \B^n$, define $d(x,z)$ as the distance between the two points measured by the hyperbolic metric. Then we have
\begin{equation*}
\cosh d\left( x,z\right) =1+2\frac{\left\vert x-z\right\vert ^{2}}{\left(
1-\left\vert x\right\vert ^{2}\right) \left( 1-\left\vert z\right\vert
^{2}\right) }.
\end{equation*}
Since $\Phi$ is an isometry with respect to the hyperbolic metric, we have
Thus 
\begin{equation*}
\frac{\left\vert x-z\right\vert ^{2}}{\left( 1-\left\vert x\right\vert
^{2}\right) \left( 1-\left\vert z\right\vert ^{2}\right) }=\frac{\left\vert
\Phi \left( x\right) -\Phi \left( z\right) \right\vert ^{2}}{\left(
1-\left\vert \Phi \left( x\right) \right\vert ^{2}\right) \left(
1-\left\vert \Phi \left( z\right) \right\vert ^{2}\right) }
\end{equation*}%
Letting $z\rightarrow \xi \in \mathbb{S}^{n-1}$ yields%
\begin{equation}\label{PhiDistance}
\left\vert \Phi \left( x\right) -\Phi \left( \xi \right) \right\vert
^{2}=\left\vert \Phi ^{\prime }\left( \xi \right) \right\vert \left\vert
\Phi ^{\prime }\left( x\right) \right\vert \left\vert x-\xi \right\vert ^{2}.
\end{equation}

Plug (\ref{PhiPrime}) and (\ref{PhiDistance}) into $\pka (\Phi(x),\Phi(\xi))$ we have
\begin{eqnarray*}
	\widetilde{p}_{\alpha }\left( \Phi \left( x\right) ,\Phi \left( \xi \right)
\right) &=& 2^{\alpha -1}\cna \frac{(1-|\Phi(x)|^2)^{1-\alpha}}{|\Phi(x)-\Phi(\xi)|^{n-\alpha}}\\
&=& 2^{\alpha -1}\cna \frac{(1-|x|^2)^{1-\alpha}|\Phi'(x)|^{1-\alpha}}{|x-\xi|^{n-\alpha}|\Phi'(x)|^{(n-\alpha)/2}|\Phi'(\xi)|^{(n-\alpha)/2}}\\
&=&\pka(x,\xi)|\Phi'(x)|^{(2-n-\alpha)/2}|\Phi'(\xi)|^{(\alpha-n)/2}.
\end{eqnarray*}
\end{proof}

It is also important to know how the Poisson kernel in the upper half space transforms under the group $SO(n,1)$. Using the definition of $\Psi$ as in (\ref{PsiDefinition}), through direct calculation, we have
\[|\Psi'(y)|=\frac{1-|\Psi(y)|^2}{2y_n}.\]

\begin{lemma}
For any $\Phi\in SO(n,1)$, define 
\[\phi=\Psi^{-1}\circ \Phi\circ\Psi,\] 
then for any $y\in \R^n_+$ and any $u\in \R^{n-1}$ we have
\begin{equation}\label{pkphi}
	p_{2-n}(\phi(y),\phi(w))=p_{2-n}(y,u)|\phi'(w)|^{1-n},
\end{equation}
as a result, for any $f:\R^n_+\to\R$ we have
\begin{equation}\label{PKphi}
	(P_{2-n}f)\circ\phi=P_{2-n}(f\circ\phi)
\end{equation}
\end{lemma}

\begin{proof}
Since we have
\[\Psi\circ\Psi^{-1}=Id,\]
using chain rule, we have
\[\Psi'(\Psi^{-1}(x))\cdot(\Psi^{-1})' (x)=Id\]
for all $x\in \B^n$. Note that we think of the left hand side of the equation as matrix multiplication. As a result we have

\begin{equation}\label{inverse}
|\Psi'(\Psi^{-1}(x))|=|(\Psi^{-1})'(x)|^{-1}	
\end{equation}

Using both (\ref{pkPsi})	 and (\ref{pkPhi}) we have
\begin{eqnarray*}
p_{2-n}(\phi(y),\phi(w))	 &=& p_{2-n}(\Psi^{-1}\circ \Phi\circ\Psi(y),\Psi^{-1}\circ \Phi\circ\Psi(w))
\\&=& \widetilde{p}_{2-n}(\Phi\circ\Psi(y),\Phi\circ\Psi(w))|\Psi'(\Psi^{-1}\circ \Phi\circ\Psi(w))|^{n-1}
\\&=& \widetilde{p}_{2-n}(\Psi(y),\Psi(w))|\Phi'(\Psi(w))|^{1-n}|\Psi'(\Psi^{-1}\circ \Phi\circ\Psi(w))|^{n-1}
\\&=& p_{2-n}(y,u)|\Psi'(w)|^{1-n}|\Phi'(\Psi(w))|^{1-n}|\Psi'(\Psi^{-1}\circ \Phi\circ\Psi(w))|^{n-1}
\\&=& p_{2-n}(y,w)|\phi'(w)|^{1-n}.
\end{eqnarray*}
Note that in the last step we used chain rule and (\ref{inverse}) by plugging in $x=\Phi\circ \Psi(w)$.

Using (\ref{pkphi}) we have
\begin{eqnarray*}
(P_{2-n}f)\circ\phi (y) &=& c_{n,2-n}\int_{\R^{n-1}}p_{2-n}(\phi(y),w)f(w)dw
\\&=& c_{n,2-n}\int_{\R^{n-1}}p_{2-n}(\phi(y),\phi(w))f(\phi(w))|\phi'(w)|^{n-1}dw
\\&=& c_{n,2-n}\int_{\R^{n-1}}p_{2-n}(y,w)f(\phi(w))dw
\\&=& P_{2-n}(f\circ\phi)(y).
\end{eqnarray*}

\end{proof}

\begin{lemma}

\begin{equation}\label{Extension}
\ln|\phi'(y)|=P_{2-n}\ln|\phi'(w)|	
\end{equation}
\end{lemma}

\begin{proof}
	Since $\Phi=\Psi\circ\phi\circ\Psi^{-1}$, by (\ref{InBall}), we have
	\[\widetilde{I}_n\circ\Phi(x)+\ln|\Phi'(x)|=\widetilde{I}_n(x)+\widetilde{P}_{2-n}(\ln|\Phi'(\xi)|).\]
	If we define $y=\Psi^{-1}(x)$, then using (\ref{Inconformal}) we have
	\begin{eqnarray*}
\widetilde{I}_n(x) &=& \widetilde{I}_n\circ\Psi(y)
\\&=& P_{2-n}(\ln|\Psi'(u)|)	-\ln|\Psi'(y)|,
	\end{eqnarray*}
	and
	\begin{eqnarray*}
	\widetilde{I}_n\circ\Phi(x)  &=&\widetilde{I}_n\circ\Phi\circ\Psi(y)
	\\&=&\widetilde{I}_n\circ\Psi(\phi(y))
	\\&=&(P_{2-n}\ln|\Psi'(w)|)\circ\phi(y)-\ln|\Psi'(\phi(y))|
	\\&=&P_{2-n}\ln|\Psi'(\phi(w))|(y)-\ln|\Psi'(\phi(y))|,
	\end{eqnarray*}
where the last step follows from (\ref{PKphi}).

On the other hand, using chain rule, we have
\begin{eqnarray*}
\ln|\Phi'(x)|&=&\ln|\Psi'(\phi\circ\Psi^{-1}(x))|+\ln|\phi'(\Psi^{-1}(x))|+\ln|(\Psi^{-1})'(x)|
\\&=&\ln|\Psi'(\phi(y))|+\ln|\phi'(y)|+\ln|(\Psi^{-1})'(\Psi(y))|.
\end{eqnarray*}
If we define $w=\Psi^{-1}(\xi)$, then using chain rule, we have
\[
\ln|\Phi'(\xi)|=\ln|\Psi'(\phi\circ\Psi^{-1}(\xi))|+\ln|\phi'(\Psi^{-1}(\xi))|+\ln|(\Psi^{-1})'(\xi)|,
\]
and by (\ref{pkPsi}) we have
\[(\widetilde{P}_{2-n}\ln|\Phi'(\xi)|)\circ\Psi=P_{2-n}\ln|\Psi'(\phi(w))|+P_{2-n}\ln|\phi'(w)|+P_{2-n}\ln|(\Psi^{-1})'(\Psi(w))|.\]
Putting everything together and using (\ref{inverse}) we have
\[\ln|\phi'(y)|=P_{2-n}\ln|\phi'(w)|.\]

\end{proof}

For any $\Phi \in SO\left( n,1\right) $ and any $\widetilde{f}\in L^{\frac{2(n-1)}{n-2+\alpha}}(\Sph^{n-1})$ we define
\begin{equation}\label{ConformalTransformationPhiSphere}
\widetilde{f}_{\Phi }\left( \xi \right) =\widetilde{f}\circ \Phi \left( \xi
\right) \left\vert \Phi ^{\prime }\left( \xi \right) \right\vert ^{\left(
n-2+\alpha \right) /2}.
\end{equation}%
Then it is easy to see that $\widetilde{f}_{\Phi}\in L^{\frac{2(n-1)}{n-2+\alpha}}(\Sph^{n-1})$ and that 
\[\|\widetilde{f}_{\Phi}\|_{L^{\frac{2(n-1)}{n-2+\alpha}}(\Sph^{n-1})}=\|\widetilde{f}\|_{L^{\frac{2(n-1)}{n-2+\alpha}}(\Sph^{n-1})}.\]
Moreover, the extension of $\widetilde{f}_{\Phi}$ as defined in (\ref{ExtensionFormulaAlpha}) transforms in the following way:
\begin{equation}\label{ConformalTransformationPhiBall}
\widetilde{P}_{\alpha }\left( \widetilde{f}_{\Phi }\right) \left( x\right)
=\left( \widetilde{P}_{\alpha }\widetilde{f}\right) \circ \Phi \left(
x\right) \left\vert \Phi ^{\prime }\left( x\right) \right\vert ^{\left(
n-2+\alpha \right) /2}.
\end{equation}

\section{A Family of Conformally Invariant Extension Inequalities}\label{FamilyOfInequalities}

\subsection{Compactness}
The goal of this subsection is to prove that the extension operator $\PKA: L^{p}(\Sph^{n-1})\to L^q(\B^n)$ is compact for certain choices of $p$ and $q$. This is done in  Corollary \ref{compact}. Before we can prove Corollary \ref{compact}, we need to prove an estimate in Proposition \ref{WeakEstimate}. 

We need the following definition:
\begin{definition}\label{SphereBallDefinition}
For $r\in (0,1]$, define $\Sph^{n-1}_r=\{x\in \B^n: |x|=r \}$ and $\B^n_r=\{x\in \B^n: |x|<r\}.$	
\end{definition}

The proof of Proposition \ref{WeakEstimate} depends on the following important technical lemma, which is also used in section \ref{PoissonKernelUnitBall}:
\begin{lemma}\label{integration}
For any $n\geq 2$, $\alpha\in [2-n,1)$	and $r\in (0,1)$
\begin{equation*}
\int_{\Sph^{n-1}_r}\pka (x,\xi)dx \leq 1,	
\end{equation*}
The notation $\Sph^{n-1}_r$ is as in definition \ref{SphereBallDefinition}.
\end{lemma}

\begin{proof}
\begin{eqnarray*}
&&\int_{\Sph^{n-1}_r}\pka (x,\xi)dx
\\ &=&	2^{\alpha-1} c_{n,\alpha}|\Sph^{n-2}|r^{n-1}(1-r^2)^{1-\alpha}\int_0^\pi\frac{\sin^{n-2}(\phi)}{(r^2+1-2r\cos(\phi))^{(n-\alpha)/2}}d\phi.
\end{eqnarray*}

Using u-substitution, take $u=\tan(\phi/2)$, we get
\begin{equation}\label{IntegrationLemmaUSub}
\begin{split}
	&\int_0^\pi\frac{\sin^{n-2}(\phi)}{(r^2+1-2r\cos(\phi))^{(n-\alpha)/2}}d\phi \\ &=\int_0^\infty \left(\frac{2u}{1+u^2}\right)^{n-2}\frac{1}{\left(r^2+1-2r\left(\frac{1-u^2}{1+u^2}\right)\right)^{(n-\alpha)/2}}\frac{2du}{1+u^2}\\
	&=\frac{2^{n-1}}{(1-r)^{n-\alpha}}\int_0^\infty\frac{u^{n-2}}{(1+u^2)^{(n-2+\alpha)/2}}\frac{du}{\left(\left(\frac{1+r}{1-r}\right)^2u^2+1\right)^{(n-\alpha)/2}}
	\end{split}
\end{equation}
Using u-substitution again, take $v=\frac{1+r}{1-r}u$, we have
\begin{eqnarray*}
	&&\int_0^\pi\frac{\sin^{n-2}(\phi)}{(r^2+1-2r\cos(\phi))^{(n-\alpha)/2}}d\phi \\ &=&\frac{2^{n-1}(1-r)^{\alpha-1}}{(1+r)^{n-1}}\int_0^\infty\frac{v^{n-2}}{\left(1+\left(\frac{1-r}{1+r}\right)^2v^2\right)^{(n-2+\alpha)/2}}\frac{dv}{(v^2+1)^{(n-\alpha)/2}}\\
	&\leq &\frac{2^{n-1}(1-r)^{\alpha-1}}{(1+r)^{n-1}}\int_0^\infty\frac{v^{n-2}dv}{(v^2+1)^{(n-\alpha)/2}}\\
	&=&\frac{2^{n-2}(1-r)^{\alpha-1}\Gamma(\frac{1-\alpha}{2})\Gamma(\frac{n-1}{2})}{(1+r)^{n-1}\Gamma(\frac{n-\alpha}{2})}
\end{eqnarray*}
Overall, we have
\begin{eqnarray*}
	\int_{\Sph^{n-1}_r}\pka (x,\xi)dx
	&\leq &\frac{2^{n-3+\alpha}r^{n-1}c_{n,\alpha}|\Sph^{n-2}|\Gamma(\frac{1-\alpha}{2})\Gamma(\frac{n-1}{2})}{(1+r)^{n-2+\alpha}\Gamma(\frac{n-\alpha}{2})}\\
	&=& \frac{2^{n-2+\alpha}r^{n-1}}{(1+r)^{n-2+\alpha}}\\
	&\leq & 1.
	\end{eqnarray*}
Here we used (\ref{cna}) and the fact that the function $\frac{r^{n-1}}{(1+r)^{n-2+\alpha}}$ is an increasing function of $r$ for $2-n\leq\alpha<1$ and $0<r<1$.
\end{proof}

\begin{remark}\label{IntegrationLemmaLimitCaseRemark}
	From (\ref{IntegrationLemmaUSub}) we see that when $\alpha=2-n$
\begin{equation*}
\begin{split}
	&\int_0^\pi\frac{\sin^{n-2}(\phi)}{(r^2+1-2r\cos(\phi))^{n-1}}d\phi\\ 
	&=\int_0^\infty \left(\frac{2u}{1+u^2}\right)^{n-2}\frac{1}{\left(r^2+1-2r\left(\frac{1-u^2}{1+u^2}\right)\right)^{n-1}}\frac{2du}{1+u^2}\\
	&=\frac{2^{n-1}}{(1-r)^{2n-2}}\int_0^\infty\frac{u^{n-2}du}{\left(\left(\frac{1+r}{1-r}\right)^2u^2+1\right)^{n-1}}.
	\end{split}
\end{equation*}
Now if we take $v=\frac{1+r}{1-r}u$ then we have
\begin{equation*}
\begin{split}
\int_0^\pi\frac{\sin^{n-2}(\phi)}{(r^2+1-2r\cos(\phi))^{n-1}}d\phi &=\frac{2^{n-1}}{(1-r^2)^{n-1}}\int_0^\infty \frac{v^{n-2}dv}{(1+v^2)^{n-1}}\\
&=\frac{2^{n-2}\left(\Gamma\left(\frac{n-1}{2}\right)\right)^2}{(1-r^2)^{n-1}\Gamma(n-1)},
\end{split}
\end{equation*}
where in the last step we used (\ref{cna}). As a result, for any $x\in\B^n$
\[\int_{\Sph^{n-1}}\widetilde{p}_{2-n}(x,\xi)=1.\]
\end{remark}

\begin{remark}\label{lowerbound}
	From the calculation in Lemma $\ref{integration}$ we can also get a lower bound for the integration. Note that
	\begin{eqnarray*}
	&&\int_0^\pi\frac{\sin^{n-2}(\phi)}{(r^2+1-2r\cos(\phi))^{(n-\alpha)/2}}d\phi\\
	&=&\frac{2^{n-1}(1-r)^{\alpha-1}}{(1+r)^{n-1}}\int_0^\infty\frac{v^{n-2}}{\left(1+\left(\frac{1-r}{1+r}\right)^2v^2\right)^{(n-2+\alpha)/2}}\frac{dv}{(v^2+1)^{(n-\alpha)/2}}\\
	&\geq &	\frac{2^{n-1}(1-r)^{\alpha-1}}{(1+r)^{n-1}}\int_0^\infty\frac{v^{n-2}dv}{(v^2+1)^{n-1}}\\
	&=&\frac{2^{n-2}(1-r)^{\alpha-1}(\Gamma(\frac{n-1}{2}))^2}{(1+r)^{n-1}\Gamma(n-1)}.
	\end{eqnarray*}
As a result, we have
	
\begin{equation*}
	\int_{\Sph^{n-1}_r}\pka(x,\xi)dx\geq\frac{2^{n-2+\alpha}r^{n-1}\Gamma(\frac{n-\alpha}{2})\Gamma(\frac{n-1}{2})}{(1+r)^{n-2+\alpha}\Gamma(n-1)\Gamma(\frac{1-\alpha}{2})}
\end{equation*}
\end{remark}

\begin{remark}\label{DCT}
 In the calculation of Lemma \ref{integration}, we have
\[\frac{v^{n-2}}{\left(1+\left(\frac{1-r}{1+r}\right)^2 v^2\right)^{(n-2+\alpha)/2}(v^2+1)^{(n-\alpha)/2}}\leq \frac{v^{n-2}}{(v^2+1)^{(n-\alpha)/2}},\] 
	for all $r\in [0,1]$ and all $v\geq 0$. As a result, by dominated convergence theorem, for any $r_0\in [0,1]$ we have
\begin{eqnarray*}
&&\lim_{r\to r_0} \int_0^\infty\frac{v^{n-2}dv}{\left(1+\left(\frac{1-r}{1+r}\right)^2 v^2\right)^{(n-2+\alpha)/2}(v^2+1)^{(n-\alpha)/2}}\\
& =&\int_0^\infty \frac{v^{n-2}dv}{\left(1+\left(\frac{1-r_0}{1+r_0}\right)^2 v^2\right)^{(n-2+\alpha)/2}(v^2+1)^{(n-\alpha)/2}}, 
\end{eqnarray*}

	In particular, we have
\begin{eqnarray*}
&&\lim_{r\to 1} \int_0^\infty\frac{v^{n-2}dv}{\left(1+\left(\frac{1-r}{1+r}\right)^2 v^2\right)^{(n-2+\alpha)/2}(v^2+1)^{(n-\alpha)/2}}\\
& =&\int_0^\infty \frac{v^{n-2} dv}{(v^2+1)^{(n-\alpha)/2}}.
\end{eqnarray*}
\end{remark}

Before we can prove the estimate in Proposition \ref{WeakEstimate} we need to define the weak norm:
\begin{definition}
Define the weak norm $L^p_W(\B^n)$	, such that
\begin{equation*}
	|u|_{L^p_W (\B^n)}=\sup_{t>0} t\left\lvert|u|>t\right\rvert^{\frac{1}{p}}.
\end{equation*}
Here $\left\lvert |u|>t\right\rvert$ is the measure of the set $\{|u|>t\}$.
\end{definition}

We are now ready to prove the following estimates, the proof uses the same method as in \cite{HWY2}.
\begin{proposition}\label{WeakEstimate}
For any $n\geq 2$ and any $2-n \leq \alpha<1$ the extension operator $\PKA$ satisfies
\begin{equation*}
\left|\PKA f\right|_{L^{\frac{n}{n-1}}_W(\B^n)}	\leq C(n,\alpha)\|f\|_{L^1(\Sph^{n-1})},
\end{equation*}
and
\begin{equation*}
	\left\|\PKA f\right\|_{L^{\frac{np}{n-1}}(\B^n)}\leq C(n,\alpha,p) \|f\|_{L^p(\Sph^{n-1})}
\end{equation*}
\end{proposition}
\begin{proof}
	Note that we only need to prove the weak estimate. The strong estimate follows from Marcinkiewicz interpolation theorem and the fact that for any $x\in \B^n$
	\begin{equation*}
		|\PKA f (x)|\leq |f|_{L^\infty(\Sph^{n-1})}\int_{\Sph^{n-1}}\pka (x,\xi)d\xi\leq C(n,\alpha)|f|_{L^\infty(\Sph^{n-1})},
\end{equation*}
here the last step follows from Lemma \ref{integration}. The constant $C(n,\alpha)$ here only depends on $n$ and $\alpha$. Note that it is different from the notation $c_{n,\alpha}$, and that $C(n,\alpha)$ changes through out the dissertation.
To prove the weak type estimate. Assume that $f\geq 0$ and $|f|_{L^1(\Sph^{n-1})}=1$. Note that
\begin{equation}\label{pkaLInfinityBound}
\begin{split}
	\pka(x,\xi) &= 2^{\alpha -1}c_{n,\alpha}\frac{(1-|x|^2)^{1-\alpha}}{|x-\xi|^{n-\alpha}}\\
	&\leq   C(n,\alpha) \frac{(1-|x|^2)^{1-\alpha}}{(1-|x|)^{n-\alpha}}\\
	&\leq  \frac{C(n,\alpha)}{(1-|x|)^{n-1}}.
\end{split}
\end{equation}
As a result, we have
\begin{equation}\label{pkainfty}
	0\leq\PKA f\leq\frac{C(n,\alpha)}{(1-|x|)^{n-1}}.
\end{equation}
From (\ref{pkainfty}) we conclude that
\begin{equation*}
|\PKA f>\lambda|=|\{x\in \B^n : 1-|x|< C(n,\alpha) \lambda ^{-\frac{1}{n-1}},\PKA f>\lambda\}|
\end{equation*}

If $1\leq C(n,\alpha)\lambda^{-\frac{1}{n-1}}$, then we have
\begin{eqnarray*}
|\PKA f>\lambda|&\leq & \frac{1}{\lambda}\int_{\B^n}\PKA fdx d\xi\\
&\leq &\frac{1}{\lambda}\int_{\Sph^{n-1}}f(\xi)\int_{\B^n}\pka (x,\xi)dx d\xi\\
&\leq &\frac{C(n,\alpha)}{\lambda}\\
&\leq &C(n,\alpha)\lambda^{-\frac{n}{n-1}}\lambda^{\frac{1}{n-1}}\\
&\leq &  C(n,\alpha)\lambda^{-\frac{n}{n-1}}.
\end{eqnarray*}
Note that here we used Lemma \ref{integration} and the fact that $\lambda^{\frac{1}{n-1}}\leq C(n,\alpha)$. Note also that the constant $C(n,\alpha)$ changes along the argument.

If $C(n,\alpha)\lambda^{-\frac{1}{n-1}}<1$ then we can define $r_0=1-C(n,\alpha)\lambda^{-\frac{1}{n-1}}$, then we have
\begin{eqnarray*}
|\PKA f>\lambda|&\leq & \frac{1}{\lambda}\int_{\B^n\backslash\B^{n}_{r_0}}\PKA fdx d\xi\\
&\leq &\frac{1}{\lambda}\int_{\Sph^{n-1}}f(\xi)\int_{\B^n\backslash\B^{n}_{r_0}}\pka (x,\xi)dx d\xi\\
&\leq &\frac{C(n,\alpha)(1-r_0)}{\lambda}\\
&\leq &  C(n,\alpha)\lambda^{-\frac{n}{n-1}}.
\end{eqnarray*}
This finishes the proof.
\end{proof}

With the help of Proposition \ref{WeakEstimate} we can prove the following:

\begin{corollary}\label{compact}
For any $n\geq 2$, $2-n\leq \alpha<1$, $1\leq p<\infty$, and $1\leq q<\frac{np}{n-1}$ the operator $\PKA: L^p(\Sph^{n-1})\to L^q(\B^n)$ is compact.
\end{corollary}
\begin{proof}
First assume $1<p<\infty$. Suppose we have a sequence of function $f_i\in L^p(\Sph^{n-1})$ such that $|f_i|_{L^p(\Sph^{n-1})}\leq 1$. Then from (\ref{pkainfty}) we have for all $i$ and all $x\in\B^n$
\begin{equation*}
	|\PKA f_i(x)|\leq \frac{C(n,\alpha)}{(1-|x|)^{n-1}}.
\end{equation*}
By Schauder estimate, there exists $u\in C^2(\B^n)$ such that $\PKA f_i\to u$ in $C^2_{loc}(\B^n)$. As a result we have: for $r\in (0,1)$
\begin{eqnarray*}
	|\PKA f_i-\PKA f_j|_{L^q(\B^n)} &\leq & |\PKA f_i-\PKA f_j|_{L^q(\B^n_r)}+|\PKA f_i-\PKA f_j|_{L^q(\B^n\backslash\B^n_r)}\\
	&\leq & |\PKA f_i-\PKA f_j|_{L^q(\B^n_r)}
	\\&&+|\PKA f_i-\PKA f_j|_{L^{\frac{np}{n-1}}(\B^n\backslash\B^n_r)}|\B^n\backslash\B^n_r|^{\frac{1}{q}-\frac{n-1}{np}}\\
	&\leq &|\PKA f_i-\PKA f_j|_{L^q(\B^n_r)}+C(n,\alpha,p)|\B^n\backslash\B^n_r|^{\frac{1}{q}-\frac{n-1}{np}},
\end{eqnarray*}
where we used Holder inequality and Proposition \ref{WeakEstimate}. Hence
\begin{equation*}
	\limsup_{i,j\to\infty}|\PKA f_i-\PKA f_j|_{L^q(\B^n)}\leq C(n,\alpha,p)|\B^n\backslash\B^n_r|^{\frac{1}{q}-\frac{n-1}{np}}.
\end{equation*}
Letting $r\to 1$, we see that  $\PKA f_i$ is a Cauchy sequence in $L^q(\B^n)$, hence $\PKA: L^p (\Sph^{n-1})\to L^q (\B^n)$ is compact.
\end{proof}
We can see that $\PKA :L^{\frac{2(n-1)}{n-2+\alpha}}(\Sph^{n-1})\to L^{\frac{2n}{n-2+\alpha}}(\B^n)$	 is not compact in the following example, which is inspired by the example given in \cite[Chapter 1]{Chang}.
\begin{remark}
We consider a sequence of conformal transformation $\Phi_a:\overline{\B^n}\to \overline{\B^n}$ defined by
\[\Phi_a(x)=\frac{a|x-a|^2+(1-|a|^2)(a-x)}{|a|^2|a^\ast-x|^2}.\]
Here $a\in \B^n$ such that $a=(0,...,0,1-\epsilon)$ for some $\epsilon\in (0,1)$, and $a^\ast=\frac{a}{|a|^2}$. From (\ref{PhiPrime}) we see that for any $x\in \B^n$
\[|\Phi_a'(x)|=\frac{1-|\Phi_a(x)|^2}{1-|x|^2}=\frac{1-|a|^2}{|a|^2|a^\ast-x|^2},\]
take limit $x\to \xi$ for some $\xi\in \Sph^{n-1}$ we get
\[|\Phi_a'(\xi)|=\frac{1-|a|^2}{|a|^2|a^\ast-\xi|^2}=\frac{\epsilon(2+\epsilon)}{(1-\epsilon)^2|a^\ast-\xi|^2}.\]
If $\xi\neq (0,...,0,1)$, then it is easy to see that $\lim_{\epsilon\to 0}|\Phi_a'(\xi)|=0$. If $\xi=(0,...,0,1)=\frac{a}{|a|}$, then we have
\[\left|\Phi_a'\left(\frac{a}{|a|}\right)\right|=\frac{\epsilon(2+\epsilon)}{\epsilon^2},\]
hence $\lim_{\epsilon\to 0}\left|\Phi_a'\left(\frac{a}{|a|}\right)\right|=\infty$.

Now consider the function $\widetilde{f} :\Sph^{n-1}\to\R$ such that $\widetilde{f}=1$. Define $\widetilde{f}_{\Phi_a}$ as in (\ref{ConformalTransformationPhiSphere}), then it is easy to see that
\[\|\widetilde{f}_{\Phi_a}\|_{L^{\frac{2(n-1)}{n-2+\alpha}}(\Sph^{n-1})}=\|\widetilde{f}\|_{L^{\frac{2(n-1)}{n-2+\alpha}}(\Sph^{n-1})},\]
and that $\widetilde{f}_{\Phi_a}$ weakly converges to the zero function in $L^{\frac{2(n-1)}{n-2+\alpha}}(\Sph^{n-1})$. For any given $x\in \B^n$, think of $\pka(x,\xi)$ as a function of $\xi$, using the $L^\infty$ bound (\ref{pkaLInfinityBound}) we can show that
\[\lim_{\epsilon\to 0}\PKA \widetilde{f}_{\Phi_a}(x)=0.\]
Now we can show that $\PKA \widetilde{f}_{\Phi_a}$ weakly converges to the zero function in $L^{\frac{2n}{n-2+\alpha}}(\B^n)$. For any function in the dual space $h\in L^{\frac{2n}{n+2-\alpha}}(\B^n)$ and any $r\in (0,1)$, we have
\begin{equation*}
\begin{split}
&\int_{\B^n} \PKA\widetilde{f}_{\Phi_a}(x)h(x)dx \\
&=\int_{\B^n\backslash\B^n_r}	 \PKA\widetilde{f}_{\Phi_a}(x)h(x)dx+\int_{\B^n_r}\PKA\widetilde{f}_{\Phi_a}(x)h(x)dx \\
& \leq \left\|\PKA \widetilde{f}_{\Phi_a}\right\|_{L^{\frac{2n}{n-2+\alpha}}(\B^n)} \|h\|_{L^{\frac{2n}{n+2-\alpha}}(\B^n\backslash \B^n_r)}+\int_{\B^n_r}\PKA\widetilde{f}_{\Phi_a}(x)h(x)dx,
\end{split}
\end{equation*}
where the second step follows from H\"older's inequality. Note that from (\ref{ConformalTransformationPhiBall}) we can see that
\[\left\|\PKA \widetilde{f}_{\Phi_a}\right\|_{L^{\frac{2n}{n-2+\alpha}}(\B^n)}=\left\|\PKA \widetilde{f}\right\|_{L^{\frac{2n}{n-2+\alpha}}(\B^n)}.\]
By dominated convergence theorem we have
\[\lim_{r\to 1}\|h\|_{L^{\frac{2n}{n+2-\alpha}}(\B^n\backslash \B^n_r)}=0.\]
Combine the $L^\infty$ bound (\ref{pkainfty}) with dominated convergence theorem we see that for any $r\in (0,1)$
\[\lim_{\epsilon\to 0}\int_{\B^n_r}\PKA\widetilde{f}_{\Phi_a}(x)h(x)dx=0.\]
Now for any $\delta>0$ small, we can choose $r\in (0,1)$ such that
\[\|h\|_{L^{\frac{2n}{n+2-\alpha}}(\B^n\backslash \B^n_r)}<\delta.\]
For this given $r$, we can choose $\epsilon>0$ small such that
\[\int_{\B^n_r}\PKA\widetilde{f}_{\Phi_a}(x)h(x)dx<\delta.\]
Combine these results, we see that
\[\lim_{\epsilon \to 0} \int_{\B^n} \PKA\widetilde{f}_{\Phi_a}(x)h(x)dx =0\]
for any $h\in L^{\frac{2n}{n+2-\alpha}}(\B^n)$.

But since
\[\left\|\PKA \widetilde{f}_{\Phi_a}\right\|_{L^{\frac{2n}{n-2+\alpha}}(\B^n)}=\left\|\PKA \widetilde{f}\right\|_{L^{\frac{2n}{n-2+\alpha}}(\B^n)}\neq 0,\]
we conclude that $\PKA:L^{\frac{2(n-1)}{n-2+\alpha}}(\Sph^{n-1})\to L^{\frac{2n}{n-2+\alpha}}(\B^n)$ is not compact.
\end{remark}

\subsection{Extremal Function}
Using Corollary \ref{compact}, we can identify the extremal function in the same way as in \cite{HWY2}. In order to do so we need the help of the Kazdan-Warner type condition which is proved in the following lemma
\begin{lemma}\label{Kazdan-Warner}
Suppose $\alpha\in (2-n,1)$, and $K,\ f\in C^1(\Sph^{n-1})$ such that for any $\xi\in\Sph^{n-1}$
\begin{equation*}
	K(\xi)f(\xi)^{\frac{n-\alpha}{n-2+\alpha}}=\int_{\B^n}\pka(x,\xi)\left(\PKA f(x)\right)^{\frac{n+2-\alpha}{n-2+\alpha}}dx.
\end{equation*}
Let $X$ be a conformal vector field in $\overline{\B}^n$, then we have
\begin{equation*}
	\int_{\Sph^{n-1}} XK\cdot f^{\frac{2(n-1)}{n-2+\alpha}}d\xi=0. 
\end{equation*}

\end{lemma}

\begin{proof}
Consider the functional
\begin{equation*}
I(K,f)=\frac{|\PKA f|_{L^{\frac{2n}{n-2+\alpha}}(\B^n)}}{\left(\int_{\Sph^{n-1}}K\cdot f^{\frac{2(n-1)}{n-2+\alpha}}d\xi\right)^{\frac{n-2+\alpha}{2(n-1)}}}
\end{equation*}
with Euler-Lagrange equation
	\begin{equation*}
	K(\xi)f(\xi)^{\frac{n-\alpha}{n-2+\alpha}}=\int_{\B^n}\pka(x,\xi)\left(\PKA f(x)\right)^{\frac{n+2-\alpha}{n-2+\alpha}}dx.
\end{equation*}
Consider $\Phi_t$ as the 1-parameter family of conformal group generated by $X$. Define 
\[f_{\Phi_t}=f\circ\Phi_t(\xi)|\Phi'_t(\xi)|^{\frac{n-2+\alpha}{2}}.\]
Since $f$ is a critical function for the functional $I(K,f)$, we have
\begin{equation*}
\frac{d}{dt}\biggr|_{t=0} I(K,f_{\Phi_t})=0.
\end{equation*}
Where according to the calculation of conformal invariance in the beginning, we have
\begin{eqnarray*}
	I(K,f_{\Phi_t})&=&\frac{|\PKA f_{\Phi_t}|_{L^{\frac{2n}{n-2+\alpha}}(\B^n)}}{\left(\int_{\Sph^{n-1}}K\cdot f_{\Phi_t}^{\frac{2(n-1)}{n-2+\alpha}}d\xi\right)^{\frac{n-2+\alpha}{2(n-1)}}}\\
	&=&\frac{|\PKA f|_{L^{\frac{2n}{n-2+\alpha}}(\B^n)}}{\left(\int_{\Sph^{n-1}}K\circ\Phi_{-t}\circ\Phi_t(\xi)\cdot \left(f\circ\Phi_t(\xi)\right)^{\frac{2(n-1)}{n-2+\alpha}}|\Phi_t '(\xi)|^{n-1}d\xi\right)^{\frac{n-2+\alpha}{2(n-1)}}}\\
	&=& \frac{|\PKA f|_{L^{\frac{2n}{n-2+\alpha}}(\B^n)}}{\left(\int_{\Sph^{n-1}}K\circ\Phi_{-t}(\xi)\cdot f^{\frac{2(n-1)}{n-2+\alpha}}d\xi\right)^{\frac{n-2+\alpha}{2(n-1)}}}\\
	&=& I(K\circ \Phi_{-t},f).
\end{eqnarray*}
As a result, we have
\begin{equation*}
	\frac{d}{dt}\biggr|_{t=0} I(K,f_{\Phi_t})=\frac{d}{dt}\biggr|_{t=0} I(K\circ \Phi_{-t},f)=0.
\end{equation*}
From this we can conclude that
\begin{equation*}
	\int_{\Sph^{n-1}}XK\cdot f^{\frac{2(n-1)}{n-2+\alpha}}d\xi=0.
\end{equation*}
\end{proof}

Now we can find the extremal function and the sharp constant using subcritical approximation as in \cite{HWY2}.

\begin{theorem}\label{maininequality}
	Assume $n\geq 3$ and $\alpha\in (2-n,1)$. For every $f\in L^{\frac{2(n-1)}{n-2+\alpha}}(\Sph^{n-1})$, we have
	\begin{equation*}
		\left\|\PKA f\right\|_{L^{\frac{2n}{n-2+\alpha}}(\B^n)}\leq S_{n,\alpha}\left\|f\right\|_{L^{\frac{2(n-1)}{n-2+\alpha}}(\Sph^{n-1})}.
	\end{equation*}
	Where $S_{n,\alpha}$ is a constant that only depends on $n$ and $\alpha$. Up to conformal transformation any constant is an optimizer.
\end{theorem}

\begin{proof}
	For $p>\frac{2(n-1)}{n-2+\alpha}$, by corollary \ref{compact}, the operator
	\begin{equation*}
		\PKA: L^p(\Sph^{n-1})\to L^{\frac{2n}{n-2+\alpha}}(\B^n)
	\end{equation*}
	is compact. 
	Consider the variational problem
	\begin{equation*}
		S_{n,\alpha}=\sup\left\{\left\|\PKA f\right\|_{L^{\frac{2n}{n-2+\alpha}}(\B^n)}: f\in L^p(\Sph^{n-1}) \text{ such that }  \left\|f\right\|_{L^{\frac{2(n-1)}{n-2+\alpha}}(\Sph^{n-1})}=1\right\}.
	\end{equation*}
	We show that the supremum is achieved as follows:

Consider a maximizing sequence $f_i\in L^p(\Sph^{n-1})$, with $f_i\geq 0$,
\[\|f_i\|_{L^p(\Sph^{n-1})}=1\]
and
\[\lim_{i\to\infty}\|f_i\|_{L^p(\Sph^{n-1})}=S_{n,\alpha}.\] 
By uniform boundedness of $L^p$ norm, we know that there exists a subsequence $f_i$ weakly converges to some function $f_p\in L^p(\Sph^{n-1})$. By compactness of $\PKA$ we also know that there exists a subsequence $f_i$ such that $\PKA f_i$ converges to $v$ in $ L^{\frac{n}{n-2+\alpha}}(\B^n)$ norm. By the weak $L^p$ convergence of $f_i$ to $f_p$, we also have $\PKA f_i$ converges to $\PKA f_p$ pointwise. As a result we have $\PKA f_p=v$ and the supremum $S_{n,\alpha}$ is achieved at $f_p$.

Replacing $f_p$ by $f_p^*$ if necessary, we may assume that $f_p$ is radial symmetric and decreasing. Meaning that for $\xi=(\xi_1,...,\xi_n)\in \R^n$ such that $|\xi|=1$, the function $f(\xi)$ only depends on $\xi_n$ and that $\frac{\partial f}{\partial \xi_n}(\xi)\leq 0$.

After rescaling, we may assume $f_p$ satisfy the Euler-Lagrange equation
\begin{equation*}
	\int_{\B^n} \pka(x,\xi) (\PKA f_p)(x)^{\frac{n+2-\alpha}{n-2+\alpha}}dx=f_p(\xi)^{p-1}=f_p(\xi)^{\frac{n-\alpha}{n-2+\alpha}}f_p(\xi)^{p-\frac{2(n-1)}{n-2+\alpha}}
\end{equation*}
Apply Proposition \ref{AppendixPropositionRegularity1} from the appendix, we know that $f_p\in C^1(\Sph^{n-1})$. By Lemma \ref{Kazdan-Warner}, we have
\begin{equation*}
\int_{\Sph^{n-1}}\langle\nabla f_p (\xi)^{p-\frac{2(n-1)}{n-2+\alpha}},\nabla\xi_n\rangle f_p(\xi)^{\frac{2(n-1)}{n-2}}d\xi=0.
\end{equation*}
Consider the function $g_p(r)=f_p(0,...,0,\sin r,\cos r)$ for $r\in [0,\pi]$. The equality becomes
\begin{equation*}
\int_0^\pi g_p '(r)g_p(r)^{p-1}\sin ^{n-2} (r) dr=0.	
\end{equation*}
Note that $g_p'=-\partial_n f\sin (r)\geq 0$. Hence we know that $f_p$ is actually a constant.
\end{proof}

\section{Limit Case Inequality}\label{LimitCaseInequality}
In this chapter we want to take limit $\alpha\to 2-n$ and study the limit case inequality. In the process of taking the limit, a very special function $\widetilde{I}_n$ shows up. The property of the function $\widetilde{I}_n$ is crucial in the study of the limit case inequality. We prove several important properties of the function $\widetilde{I}_n$ in Section \ref{SectionTheFunctionIn}.

Through out this Section \ref{LimitCaseInequality}, Section \ref{SectionTheFunctionIn} and Section \ref{UniquenessThroughMovingSpheres}, we still use notations $\widetilde{f}$ and $f$ to denote functions on $\Sph^{n-1}$ and $\R^{n-1}$ respectively, but the relation between them is different from the relation discussed in previous sections. We will specify their relation in (\ref{flimitdef}) below.

We consider the limit case $\alpha\to 2-n$ in the same way as \cite{Ch}, our statement and proof are slightly different.

For any $\widetilde{F}\in L^\infty (\Sph^{n-1})$, define $\widetilde{f}=1+\frac{n-2+\alpha}{2}\widetilde{F}$. We have $\widetilde{f}\in L^{\frac{2(n-1)}{n-2+\alpha}}(\Sph^{n-1})$ for all $\alpha\in (2-n,1)$. We prove the following theorem for $\widetilde{F}$.

\begin{theorem}\label{TheoremLimitCaseInequalityExplicitIn}
For dimension $n\geq 2$, and any function $\widetilde{F}\in L^\infty(\Sph^{n-1})$ we have
\begin{equation}\label{InequalityLimitCase}
\|e^{\widetilde{I}_n+\widetilde{P}_{2-n}\widetilde{F}}\|_{L^n(\B^n)}\leq S_n \|e^{\widetilde{F}}\|_{L^{n-1}(\Sph^{n-1})}.	
\end{equation}

Where $\widetilde{I}_n(x)=2\frac{d\widetilde{P}_\alpha 1}{d\alpha}\bigr|_{\alpha=2-n}$. When $n$ is even we have
\[\widetilde{I}_n(x)=\sum_{k=1}^{n/2-1} \frac{1}{2k}\cdot \frac{\Gamma\left(\frac{n-2}{2}\right)\Gamma\left(n-k-1\right)}{\Gamma(n-2)\Gamma\left(\frac{n}{2}-k\right)} (1-|x|^2)^k.\]
The sharp constant
$S_n=\frac{\|e^{\widetilde{I}_n}\|_{L^n(\B^n)}}{|\Sph^{n-1}|^{\frac{1}{n-1}}}$
\end{theorem}

\begin{proof}
For any $\widetilde{F}\in L^\infty(\Sph^{n-1})$, define $\widetilde{f}=1+\frac{n-2+\alpha}{2}\widetilde{F}$. Define $\epsilon=n-2+\alpha$, from theorem \ref{maininequality}, we have
\[\|\widetilde{P}_\alpha (1+\epsilon \widetilde{F})\|_{L^{\frac{2n}{n-2+\alpha}}(\B^n)}\leq S_{n,\alpha}\|1+\epsilon\widetilde{F}\|_{L^{\frac{2(n-1)}{n-2+\alpha}}(\Sph^{n-1})},\]
which is equivalent to
\[\left(\int_{\B^n}(\PKA 1)^{\frac{n}{\epsilon}}\left(1+\frac{\epsilon \PKA \widetilde{F}}{\PKA 1}\right)^{\frac{n}{\epsilon}}\right)^{\frac{1}{n}}\leq (S_{n,\alpha})^{\frac{1}{\epsilon}}\left(\int_{\Sph^{n-1}}(1+\epsilon \widetilde{F})^{\frac{n-1}{\epsilon}}\right)^{\frac{1}{n-1}}.\]
As in \cite{Ch}, we need to find a lower bound for $\PKA 1$ and an upper bound for $(\PKA 1)^{\frac{n}{\epsilon}}$.

We handle the lower bound for $\PKA 1$ firstly. From remark \ref{lowerbound}, we know that
\[\PKA 1\geq \frac{\Gamma\left(\frac{n-\alpha}{2}\right)\Gamma\left(\frac{n-1}{2}\right)}{\Gamma(n-1)\Gamma\left(\frac{1-\alpha}{2}\right)}.\]
Since $\frac{\Gamma\left(\frac{n-\alpha}{2}\right)\Gamma\left(\frac{n-1}{2}\right)}{\Gamma(n-1)\Gamma\left(\frac{1-\alpha}{2}\right)}$ is a continuous function of $\alpha$ for all $\alpha\in [2-n,1)$, and
\[\frac{\Gamma\left(\frac{n-\alpha}{2}\right)\Gamma\left(\frac{n-1}{2}\right)}{\Gamma(n-1)\Gamma\left(\frac{1-\alpha}{2}\right)}>0\]
for $\alpha<1$. As a result for some $0<\alpha_0<1$, there exists $m>0$ such that
\[\frac{\Gamma\left(\frac{n-\alpha}{2}\right)\Gamma\left(\frac{n-1}{2}\right)}{\Gamma(n-1)\Gamma\left(\frac{1-\alpha}{2}\right)}\geq m>0\]
for all $\alpha$ such that $2-n\leq \alpha <\alpha_0<1$. Here $m$ will be the lower bound for the function $\PKA 1(x)$. Note that it does not depend on $\alpha$ or $x$.

Now we consider the upper bound for $(\PKA 1)^{\frac{n}{\epsilon}}$. From Lemma \ref{integration}, we have
\[\PKA 1 =\frac{\int_{\Sph^{n-1}_r} \pka(x,\xi)dx}{r^{n-1}}\leq \left(\frac{2}{1+r}\right)^{n-2+\alpha},\]
for all $\alpha\in [2-n,1)$.
As a result we have an upper bound
\[(\PKA 1)^{\frac{n}{\epsilon}}\leq \left(\frac{2}{1+r}\right)^{n}\leq 2^n,\]
for all $\alpha\in [2-n,1)$. If we define $\widetilde{I}_n=2\frac{d\PKA 1}{d\alpha}|_{\alpha=2-n} $, then with the help of the lower bound for $\PKA 1(x)$ and the upper bound for $(\PKA 1)^{\frac{n}{\epsilon}}$, we can apply dominated convergence theorem to get
\[\lim_{\epsilon\to 0}\int_{\B^n} (\PKA 1)^{\frac{n}{\epsilon}} \left(1+\frac{\epsilon \PKA \widetilde{F}}{\PKA 1}\right)^{\frac{n}{\epsilon}}=\int_{\B^n} e^{n\widetilde{I}_n +n\PK_{2-n} \widetilde{F}}.\]

For the right hand side of the inequality we can get
\[\lim_{\epsilon\to 0} \int_{\Sph^{n-1}}(1+\epsilon \widetilde{F})^{\frac{n-1}{\epsilon}}=\int_{\Sph^{n-1}} e^{(n-1)\widetilde{F}}.\]
In order to find the limit $\lim_{\epsilon\to 0} (S_{n,\alpha})^{\frac{1}{n}}$, first note that since constant is an optimizer in theorem \ref{maininequality}. As a result, if we take $\widetilde{F}=0$, then  we can have
\[\lim_{\epsilon\to 0}(S_{n,\alpha})^{\frac{1}{\epsilon}}=\lim_{\epsilon\to 0}\frac{\left(\int_{\B^n} (\PKA 1)^{\frac{n}{\epsilon}}
\right)^{\frac{1}{n}}}{|\Sph^{n-1}|^{\frac{1}{n-1}}}=\frac{\|e^{\widetilde{I}_n}\|_{L^n(\B^n)}}{|\Sph^{n-1}|^{\frac{1}{n-1}}}.\]
When $n$ is an even integer, using mathematical induction and (\ref{InInductionRelation})	it is easy to prove that
\[\widetilde{I}_n(x)=\sum_{k=1}^{n/2-1} \frac{1}{2k}\cdot \frac{\Gamma\left(\frac{n-2}{2}\right)\Gamma\left(n-k-1\right)}{\Gamma(n-2)\Gamma\left(\frac{n}{2}-k\right)} (1-|x|^2)^k.\]
This finishes the proof.
\end{proof}

\section{The Function $\widetilde{I}_n$}\label{SectionTheFunctionIn}
The function $\widetilde{I}_n$ naturally appears in the process of taking limit; its properties are very important for subsequent analysis. Yang \cite{Yang} found an explicit formula for the function $\widetilde{I}_n$ when $n$ is an even integer. In this section we prove an induction relation concerning $\widetilde{I}_n$, in particular the relation (\ref{InInductionRelation}). When $n$ is an even integer, this induction relation can be used to find explicit formula for $\widetilde{I}_n$ as discussed in the proof of Theorem \ref{TheoremLimitCaseInequalityExplicitIn}. When $n$ is odd we do not have an explicit formula for $\widetilde{I}_n$, but the same induction relation still applies.

In Subsection \ref{SubsectionConformalTransformationIn} we consider how $\widetilde{I}_n$ transforms under conformal transformation. In Subsection \ref{SubsectionInductionRelation} we prove the induction relation (\ref{InInductionRelation}). The induction relation (\ref{InInductionRelation}) is the basis for the proofs in Subsection \ref{SubsectionHyperbolicHarmonicThroughInduction} and Subsection \ref{SubsectionBoundaryValueThroughInduction}

\subsection{Conformal Transformation of $\widetilde{I}_n$}\label{SubsectionConformalTransformationIn}

In this subsection we want to take a closer look at how the function $\widetilde{I}_n$ transforms under conformal transformation $\Phi: \B^n\to \B^n$, as well as how the it transforms under the projection map $\Psi:\R^n_+\to \B^n$. These are given in (\ref{InBall}) and (\ref{Inconformal}) respectively.

To begin with, we have

\begin{equation*}
\PKA 1=2^{\alpha -1}c_{n,\alpha}\int_{\Sph^{n-1}}\frac{(1-|x|^2)^{1-\alpha}}{|x-\xi|^{n-\alpha}}d\xi,
\end{equation*}
taking derivative with respect to $\alpha$ at $\alpha=2-n$, we get

\begin{equation}\label{InFirstForm}
\begin{split}
\frac{d\PKA 1}{d\alpha} \biggr|_{\alpha=2-n} &= 2^{1-n}c_{n,2-n}\int_{\Sph^{n-1}}\frac{(1-|x|^2)^{n-1}\ln|x-\xi|}{|x-\xi|^{2n-2}}d\xi 
\\ &-\ln(1-|x|^2)+\ln(2)-\frac{\psi^0(n-1)}{2}+\frac{\psi^0(\frac{n-1}{2})}{2}
\end{split}
\end{equation}
where $\psi^0(x)=\frac{d}{dx}\ln(\Gamma(x))$ is the polygamma function.

Under conformal transformation $\Phi$, we can see that
\begin{eqnarray*}
 &&\int_{\Sph^{n-1}}\pka(\Phi(x),\xi)\ln|\Phi(x)-\xi|d\xi \\
&= &\int_{\Sph^{n-1}}(\pka(x,\xi)\ln|x-\xi|)|\Phi'(x)|^{\frac{2-n-\alpha}{2}}|\Phi'(\xi)|^{\frac{n-2+\alpha}{2}}d\xi\\
&&+\int_{\Sph^{n-1}}(\pka(x,\xi)\ln|\Phi'(\xi)|^{\frac{1}{2}})|\Phi'(x)|^{\frac{2-n-\alpha}{2}}|\Phi'(\xi)|^{\frac{n-2+\alpha}{2}}d\xi\\
&&+\int_{\Sph^{n-1}}(\pka(x,\xi)\ln|\Phi'(x)|^{\frac{1}{2}})|\Phi'(x)|^{\frac{2-n-\alpha}{2}}|\Phi'(\xi)|^{\frac{n-2+\alpha}{2}}d\xi.
\end{eqnarray*}
When we take limit $\alpha\to 2-n$, we get
\begin{eqnarray*}
\int_{\Sph^{n-1}}\pk_{2-n}(\Phi(x),\xi)\ln|\Phi(x)-\xi|d\xi &=&\int_{\Sph^{n-1}}\pk_{2-n}(x,\xi)\ln|x-\xi|d\xi\\
&&+\frac{1}{2}\PK_{2-n}(\ln|\Phi'(\xi)|)+\frac{1}{2}\ln|\Phi'(x)|.
\end{eqnarray*}
From which we can get the conformal transformation for $\frac{d\PKA 1}{d\alpha} \big|_{\alpha=2-n}$
\[\frac{d\PKA 1}{d\alpha}\bigg|_{\alpha=2-n}\circ\Phi(x)+\frac{1}{2}\ln|\Phi'(x)|=\frac{d\PKA 1}{d\alpha} \bigg|_{\alpha=2-n}(x)+\frac{1}{2}\PK_{2-n}(\ln|\Phi'(\xi)|).\]
Recall that for any $x\in \B^n$, we define
\[\widetilde{I}_n(x)=2\frac{d\PKA 1}{d\alpha} \bigg|_{\alpha=2-n}(x),\]
then we have
\begin{equation}\label{InBall}
\widetilde{I}_n	\circ \Phi(x)+\ln |\Phi'(x)|=\widetilde{I}_n(x)+\widetilde{P}_{2-n}(\ln|\Phi'(\xi)|).
\end{equation}
From this we see that $\widetilde{I}_n$ is a radial function, as a result, we sometimes think of $\widetilde{I}_{n}(x)$ as $\widetilde{I}_n(r)$ for $r=|x|$. 

We also want to consider how $\widetilde{I}_n$ changes under the transformation $\Psi:\R^n_+\to \B^n$. Since through change of variable we have
\[P_\alpha1=c_{n,\alpha}\int_{\R^{n-1}}\frac{1}{(|u|^2+1)^{\frac{n-\alpha}{2}}}du=c_{n,\alpha}|\Sph^{n-2}|\int_0^\infty \frac{r^{n-2}dr}{(r^2+1)^{\frac{n-\alpha}{2}}}=1,\]
we have
\[\frac{dP_{\alpha}1}{d\alpha}\bigg|_{\alpha=2-n}=0\]
for all $y\in \R^n_+$.

Define 
\[I_n(y)=2\frac{dP_{\alpha}1}{d\alpha}\bigg|_{\alpha=2-n},\] 
then using (\ref{pkPsi}) we can show that
\begin{equation}\label{Inconformal}
\widetilde{I}_n	\circ \Psi (y)+\ln|\Psi'(y)|=I_n(y)+P_{2-n}\ln|\Psi'(w)|=P_{2-n}\ln|\Psi'(w)|.
\end{equation}
Note that here $w\in \R^{n-1}$, whereas $|\Psi'(y)|$ and $|\Psi'(w)|$ are as in (\ref{PsiConformalFactorExampleUpperHalf}) and (\ref{PsiConformalFactorExampleBoundary}) respectively.

\subsection{Simplify the Function $\widetilde{I}_n$}
In this subsection we want to further simplify the function $\widetilde{I}_n$. We write the integration in the polar coordinate in the Euclidean ball $\B^n$, then (\ref{InFirstForm}) becomes

\begin{eqnarray*}
\frac{d\PKA 1}{d\alpha} \biggr|_{\alpha=2-n} &=& \frac{2^{1-n}\Gamma(n-1)}{\Gamma\left(\frac{n-1}{2}\right)^2}(1-r^2)^{n-1}\int_0^\pi\frac{\sin^{n-2}\phi\ln(1-2r\cos\phi+r^2) d\phi}{(1-2r\cos\phi+r^2)^{n-1}}\\
&&-\ln(1-r^2)+\ln(2)-\frac{\psi^0(n-1)}{2}+\frac{\psi^0(\frac{n-1}{2})}{2}.
\end{eqnarray*}
Note that here we also used the explicit formula for the constant $c_{n,2-n}$ from (\ref{cna}).

Polygamma functions have two special properties that are useful to us. The first property is
\begin{equation}\label{polygamma1}
\psi^0(n)=-\gamma+\sum_{k=1}^{n-1}\frac{1}{k},\text{ where }n\in\mathbb{N}^+ .
\end{equation}
Here $\gamma$ is the Euler Mascheroni constant. The second property is
\begin{equation}\label{polygamma2}
\psi^0(2z)=\frac{1}{2}\left(\psi^0(z)+\psi^0(z+\frac{1}{2})\right)+\ln(2),\text{ where }z\in\mathbb{C}^*.	
\end{equation}

For any $n\in \mathbb{N}^+$ such that $n\geq 2$, plug $z=\frac{n-1}{2}$ into (\ref{polygamma2}) we get
\begin{equation}\label{polygamma3}
\psi^0\left(\frac{n-1}{2}\right)	-\psi^0(n-1) =\psi^0(n-1)-\psi^0\left(\frac{n}{2}\right)-\ln (4).
\end{equation}
Combine (\ref{polygamma3}) and (\ref{polygamma1}) we get that when $n\in\mathbb{N}^+$ is an even integer then
\begin{equation}\label{polygammahalf}
\psi^0\left(\frac{n-1}{2}\right)=-\ln(4)-\gamma+\sum_{k=n/2}^{n-2}\frac{1}{k}+\sum_{k=1}^{n-2}\frac{1}{k}	
\end{equation}
Now combine (\ref{polygamma1}), (\ref{polygamma3}) and (\ref{polygammahalf}) we see that for any $n\geq 2$ such that $n$ is an even integer:
\begin{equation*}
	\psi^0\left(\frac{n-1}{2}\right)	-\psi^0(n-1)=-\log(4)+\sum_{k=n/2}^{n-2}\frac{1}{k}.
\end{equation*}
Using (\ref{polygamma1}) we can see that for any $n\geq 2$ such that $n$ is an odd integer:
\begin{equation*}
	\psi^0\left(\frac{n-1}{2}\right)	-\psi^0(n-1)=-\sum_{k=(n-1)/2}^{n-2}\frac{1}{k}.
\end{equation*}
As a result, when $n\geq 2$ is an even integer, equation (\ref{InFirstForm}) simplifies to

\begin{equation}\label{InSimplifiedFormEven}
\begin{split}
	&\frac{d\PKA 1}{d\alpha} \biggr|_{\alpha=2-n} \\&=\frac{2^{1-n}\Gamma(n-1)}{\Gamma\left(\frac{n-1}{2}\right)^2}\int_0^\pi \frac{(1-r^2)^{n-1}\sin^{n-2}\phi\ln\ipt}{\ipt^{n-1}}d\phi \\
	&-\ln(1-r^2)+\frac{1}{2}\sum_{k=n/2}^{n-2} \frac{1}{k}.
	\end{split}
\end{equation}
When $n\geq 2$ is an odd integer, equation (\ref{InFirstForm}) simplifies to
\begin{equation}\label{InSimplifiedFormOdd}
\begin{split}
	&\frac{d\PKA 1}{d\alpha} \biggr|_{\alpha=2-n} \\&=\frac{2^{1-n}\Gamma(n-1)}{\Gamma\left(\frac{n-1}{2}\right)^2}\int_0^\pi \frac{(1-r^2)^{n-1}\sin^{n-2}\phi\ln\ipt}{\ipt^{n-1}}d\phi \\
	&-\ln(1-r^2)+\ln2-\frac{1}{2}\sum_{k=(n-1)/2}^{n-2} \frac{1}{k}.
	\end{split}
\end{equation}

\subsection{Induction Relation}\label{SubsectionInductionRelation}
We have the following induction relation.

\begin{lemma}\label{InInductionLemma}
For $n\in\mathbb{N}^+$ such that $n>3$, $\widetilde{I}_n$ satisfies the induction relation
\begin{equation}\label{InInductionRelation}
\widetilde{I}_n=\frac{1-r^4}{4r(n-3)}\frac{d}{dr}(\widetilde{I}_{n-2})+\widetilde{I}_{n-2}+\frac{1-r^2}{2(n-3)}.
\end{equation}
\end{lemma}

\begin{proof}
The main calculation here is to use integration by parts to evaluate the integral
\[\int_0^\pi \frac{\sin^{n-2}\phi\ln\ipt}{\ipt^{n-1}}d\phi.\]
Take 
\[v=\sin^{n-3}\phi\] 
then we have \[dv=(n-3)\sin^{n-4}\phi\cos\phi d\phi.\] Take \[dw=\frac{\sin\phi \ln\ipt}{\ipt^{n-1}}d\phi,\] then we have \[w=-\frac{1}{2r}\left(\frac{\ln\ipt}{(n-2)\ipt^{n-2}}+\frac{1}{(n-2)^2\ipt^{n-2}}\right).\]
As a result, we have
\begin{equation*}
\begin{split}
	&\int_0^\pi\frac{\sin^{n-2}\phi\ln\ipt d\phi}{\ipt^{n-1}} \\
	&= \frac{(n-3)(1+r^2)}{(n-2)4r^2}\int\frac{\sin^{n-4}\phi\ln\ipt d\phi}{\ipt^{n-2}}
	\\ &-\frac{(n-3)}{(n-2)4r^2}\int\frac{\sin^{n-4}\phi\ln\ipt d\phi}{\ipt^{n-3}}
	\\ &+\frac{(n-3)(1+r^2)}{(n-2)^2 4r^2}\int\frac{\sin^{n-4}\phi d\phi}{\ipt^{n-2}}
	\\&-\frac{n-3}{(n-2)^2 4r^2}\int \frac{\sin^{n-4}\phi d\phi}{\ipt^{n-3}}.
\end{split}
\end{equation*}
We evaluate each one of the integrals separately in the next subsection. Using the results from next subsection, namely by (\ref{usk}), (\ref{uskp1}) and (\ref{uslkp1}), we have

\begin{equation*}
\begin{split}
	&\int_0^\pi\frac{\sin^{n-2}\phi\ln\ipt d\phi}{\ipt^{n-1}}\\
	 &=\frac{(1+r^2)}{(n-2)4r(1-r^2)}\frac{d}{dr}\left(\int\frac{\sin^{n-4}\phi\ln\ipt d\phi}{\ipt^{n-3}}\right)
	\\&+\frac{(n-3)}{2(n-2)(1-r^2)}\int\frac{\sin^{n-4}\phi\ln\ipt d\phi}{\ipt^{n-3}}
	\\&+\frac{1}{n-2}\frac{2^{n-4}\Gamma(\frac{n-3}{2})^2}{\Gamma(n-3)}\left(\frac{n-3}{n-2}+\frac{1+r^2}{2}\right)\frac{1}{(1-r^2)^{n-1}}	
\end{split}
\end{equation*}
If we define
\[a_n=\frac{2^{1-n}\Gamma(n-1)}{\Gamma\left(\frac{n-1}{2}\right)^2}(1-r^2)^{n-1}\int\frac{\sin^{n-2}\phi\ln\ipt d\phi}{\ipt^{n-1}},\]
then we have the following induction relation
\begin{equation}\label{InductionAn}
a_n=\frac{1-r^4}{4r(n-3)}\frac{d}{dr}(a_{n-2})+a_{n-2}+\frac{1}{2(n-2)}+\frac{1+r^2}{4(n-3)}	
\end{equation}

When $n$ is even, from (\ref{InSimplifiedFormEven}) we see that
\begin{equation}\label{InductionRelationBetweenInAn}
\widetilde{I}_n=2\frac{d\PKA 1}{d\epsilon} |_{\epsilon=0}=2a_n-2\ln(1-r^2)+\sum_{k=\frac{n}{2}}^{n-2}\frac{1}{k}.
\end{equation}
Taking derivative with respect to $r$ we see that
\begin{equation}\label{InductionDerivative}
\frac{d}{dr}(\widetilde{I}_{n-2})=2\frac{d}{dr}(a_{n-2})+\frac{4r}{1-r^2}.	
\end{equation}
Combine (\ref{InductionAn}), (\ref{InductionRelationBetweenInAn}) and (\ref{InductionDerivative}) we get that
\begin{equation*}
\begin{split}
	\widetilde{I}_n &=\frac{1-r^4}{4r(n-3)}\frac{d}{dr}(2a_{n-2})+2a_{n-2}+\frac{1}{n-2}+\frac{1+r^2}{2(n-3)}\\
	&-2\ln(1-r^2)+\sum_{k=n/2}^{n-2}\frac{1}{k}\\
	&=\frac{1-r^4}{4r(n-3)}\left(\frac{d}{dr}(\widetilde{I}_{n-2})-\frac{4r}{1-r^2}\right)\\
	&+2a_{n-2}-2\ln (1-r^2)+\sum_{k=(n-2)/2}^{n-4}\frac{1}{k}\\
	&+\frac{1}{n-3}+\frac{1+r^2}{2(n-3)}\\
	&=\frac{1-r^4}{4r(n-3)}\frac{d}{dr}(\widetilde{I}_{n-2})+\widetilde{I}_{n-2}+\frac{1-r^2}{2(n-3)}.
\end{split}	
\end{equation*}
This is the end of the calculation for the case when $n$ is even.

In the case when $n$ is odd, from (\ref{InSimplifiedFormOdd}) we have
\begin{equation*}
\widetilde{I}_n=2\frac{d\PKA 1}{d\epsilon} |_{\epsilon=0}=2a_n-2\ln(1-r^2)+\ln4-\sum_{k=(n-1)/2}^{n-2} \frac{1}{k}.
\end{equation*}
Going through similar calculations as in the case when $n$ is even, we can see that for the case $n$ is odd we have exactly the same induction relation.
\end{proof}

Using the induction relation (\ref{InInductionRelation}) and the fact that $\widetilde{I}_2=0$, it is easy to find an explicit formula for $\widetilde{I}_{n}$ when $n$ is even.

\subsection{Supplementary Calculation}
In this subsection we continue several calculations from previous subsection. We use $k$ to denote any positive integer.
\begin{lemma}
	For any $k\in \mathbb{N}^+$ such that $k\geq 2$
	\begin{equation}\label{usk}
	\int_0^\pi\frac{\sin^{k-1}\phi d\phi}{\ipt^k}=\frac{2^k}{|\Sph^{k-1}|c_{k+1,1-k}}\frac{1}{(1-r^2)^k}.
\end{equation}
\end{lemma}
\begin{proof}
This follows directly from Remark \ref{IntegrationLemmaLimitCaseRemark}.
\end{proof}

\begin{lemma}
	For any $k\in \mathbb{N}^+$ such that $k\geq 2$
	\begin{equation}\label{uskp1}
	\int_0^\pi\frac{\sin^{k-1}\phi d\phi}{\ipt^{k+1}}=\frac{2^{k-1}(\Gamma(k/2))^2}{\Gamma(k)}\frac{1+r^2}{(1-r^2)^{k+2}}	
	\end{equation}
\end{lemma}
\begin{proof}
By taking derivative with respect to $r$, we get
\begin{equation*}
\begin{split}
&\frac{r}{k(1-r^2)}\frac{d}{dr}\left(\int_0^\pi\frac{\sin^{k-1}\phi d\phi}{\ipt^k}\right)	\\
&=\int_0^\pi\frac{\sin^{k-1}\phi d\phi}{\ipt^{k+1}}-\frac{1}{1-r^2}\int_0^\pi\frac{\sin^{k-1}\phi d\phi}{\ipt^k}.
\end{split}
\end{equation*}
Combine this with (\ref{usk}) then we are done.
\end{proof}

\begin{lemma}
	For any $k\in \mathbb{N}^+$ such that $k\geq 2$
\begin{equation}\label{uslkp1}
\begin{split}
	&\int_0^\pi\frac{\sin^{k-1}\phi\ln\ipt d\phi}{\ipt^{k+1}} 
	\\&=\frac{r}{k(1-r^2)}\frac{d}{dr}\left(\int_0^\pi\frac{\sin^{k-1}\phi\ln\ipt d\phi}{\ipt^{k}}\right)\nonumber
	\\ &+\frac{1}{1-r^2}\int_0^\pi\frac{\sin^{k-1}\phi\ln\ipt d\phi}{\ipt^k}\nonumber
	\\ &+\frac{2^{k-1}(\Gamma(k/2))^2}{\Gamma(k)}\frac{2r^2}{k(1-r^2)^{k+2}}	.
\end{split}	
\end{equation}
\end{lemma}

\begin{proof}
Take derivative of $\int_0^\pi\frac{\sin^{k-1}\phi\ln\ipt d\phi}{\ipt^k}$ with respect to $r$, we can get
\begin{equation*}
\begin{split}
	 &\int_0^\pi\frac{\sin^{k-1}\phi\ln\ipt d\phi}{\ipt^{k+1}}
	\\&=\frac{r}{k(1-r^2)}\frac{d}{dr}\left(\int_0^\pi\frac{\sin^{k-1}\phi\ln\ipt d\phi}{\ipt^{k}}\right)
	\\ &+\frac{1}{1-r^2}\int_0^\pi\frac{\sin^{k-1}\phi\ln\ipt d\phi}{\ipt^k}
	\\ &-\frac{1}{k(1-r^2)}\int_0^\pi\frac{\sin^{k-1}\phi d\phi}{\ipt^k}
	\\ &+\frac{1}{k}\int_0^\pi\frac{\sin^{k-1}\phi d\phi}{\ipt^{k+1}}.
\end{split}
\end{equation*}

By (\ref{usk}) and (\ref{uskp1}), we have
\begin{equation*}
\begin{split}
&\frac{1}{k}\int\frac{\sin^{k-1}\phi d\phi}{\ipt^{k+1}}-\frac{1}{k(1-r^2)}\int\frac{\sin^{k-1}\phi d\phi}{\ipt^k}\\
&=\frac{2^{k-1}(\Gamma(k/2))^2}{\Gamma(k)}\frac{2r^2}{k(1-r^2)^{k+2}}	
\end{split}	
\end{equation*}
Combine these two equations we can get (\ref{uslkp1}).
\end{proof}

\subsection{Hyperbolic Harmonic Through Induction}\label{SubsectionHyperbolicHarmonicThroughInduction}
Using the induction relation (\ref{InInductionRelation}), we can prove that $\widetilde{I}_n	\circ \Psi (y)+\ln|\Psi'(y)|$ is harmonic with respect to the standard hyperbolic metric. We prove it in the unit ball model of hyperbolic space.
\begin{lemma}\label{InLnHyperbolicHarmonic}
For $n\geq 2$, in $\B^n$ we have
\[\Delta_{\mathbb{H}} \left(\widetilde{I}_n+\ln \frac{1-2x_n+|x|^2}{2}\right)	=0.\]
Here $\Delta_{\mathbb{H}}$ is the Laplacian in hyperbolic space. For any function $u\in C^\infty(\B^n)$ we have
\[\Delta_{\mathbb{H}}u=\left(\frac{1-|x|^2}{2}\right)^2\Delta u+(n-2)\frac{1-|x|^2}{2}\langle x, \nabla u \rangle.\]
Where $\Delta$ is the Laplacian in Euclidean space and $\langle x, \nabla u \rangle$ is the inner product in Euclidean space.
\end{lemma}

\begin{proof}
	Through direct calculation we have
\[\nabla\ln\frac{1-2x_n+|x|^2}{2}=\frac{2x-2e_n}{1-2x_n+|x|^2},\]
here $e_n$ means the unit vector in the direction $x_n$.

As a result, we have
\[\Delta \ln\frac{1-2x_n+|x|^2}{2}=\frac{2n-4}{1-2x_n +|x|^2},\]
and
\begin{eqnarray*}
\Delta_{\mathbb{H}} \ln\frac{1-2x_n+|x|^2}{2} &=&\left(\frac{1-|x|^2}{2}\right)^2 \frac{2n-4}{1-2x_n +|x|^2}
\\&&+(n-2)\frac{1-|x|^2}{2}\left\langle x,\frac{2x-2e_n}{1-2x_n+|x|^2}\right\rangle
\\	&=& \frac{(1-|x|^2)^2(n-2)+(n-2)(1-|x|^2)(2|x|^2-2x_n)}{2(1-2x_n+|x|^2)}
\\ &=& \frac{(n-2)(1-|x|^2)}{2}
\end{eqnarray*}

Next we want to show that
\begin{equation}\label{InInductionResult}
\Delta_{\mathbb{H}} \widetilde{I}_n=-\frac{(n-2)(1-|x|^2)}{2}.
\end{equation}
Since $\widetilde{I}_n$ is a radial function, we can verify this in polar coordinates in the Euclidean unit ball. Where we have
\[\Delta_{\mathbb{H}} \widetilde{I}_n=\left(\frac{1-r^2}{2}\right)^2\partial_r^2 \widetilde{I}_n+\left(\frac{1-r^2}{2}\right)^2\frac{n-1}{r}\partial_r \widetilde{I}_n+\frac{(n-2)r(1-r^2)}{2}\partial_r \widetilde{I}_n\]
For $n=2$ it is easy to see that $\widetilde{I}_2=0$, and for $n=3$, we can integrate by part to get $\widetilde{I}_3=\ln(4)+\frac{(1-r)^2\ln(1-r)-(1+r)^2\ln(1+r)}{2r}$. So it is easy to verify by direct calculation that (\ref{InInductionResult}) is true for $n=2$ and $n=3$.

When $n>3$, suppose (\ref{InInductionResult}) is true for $\widetilde{I}_{n-2}$, from which we have
\begin{eqnarray*}
\Delta_{\mathbb{H}} \widetilde{I}_{n-2} &=&\left(\frac{1-r^2}{2}\right)^2\partial_r^2 \widetilde{I}_{n-2}+\left(\frac{1-r^2}{2}\right)^2\frac{n-3}{r}\partial_r \widetilde{I}_{n-2}
\\&&+\frac{(n-4)r(1-r^2)}{2}\partial_r \widetilde{I}_{n-2}\\
&=& \left(\frac{1-r^2}{2}\right)^2\partial_r^2 \widetilde{I}_{n-2}+\frac{1-r^2}{4r}((n-3)+(n-5)r^2)\partial_r \widetilde{I}_{n-2} \\
&=&-\frac{(n-4)(1-r^2)}{2}.
\end{eqnarray*}
After rearranging we get
\begin{equation}\label{inductionstep}
	\partial_r^2 \widetilde{I}_{n-2}=-\frac{2(n-4)}{1-r^2}-\frac{2(n-4)r}{1-r^2}\partial_r\widetilde{I}_{n-2}-\frac{n-3}{r}\partial_r\widetilde{I}_{n-2}.
\end{equation}
Consider (\ref{InInductionRelation}), take derivative with respect to $r$, we get
\[\partial_r \widetilde{I}_n=\frac{1-r^4}{4r(n-3)}\partial_r^2\widetilde{I}_{n-2}+\left(1-\frac{1+3r^4}{4r^2(n-3)}\right)\partial_r \widetilde{I}_{n-2}-\frac{r}{n-3},\]
and
\begin{eqnarray*}
\partial_r^2 \widetilde{I}_n &=& 	\frac{1-r^4}{4r(n-3)}\partial_r^3\widetilde{I}_{n-2} +\left(1-\frac{1+3r^4}{2r^2(n-3)}\right)\partial_r^2 \widetilde{I}_{n-2} 
\\ &&+\frac{1-3r^4}{2r^3(n-3)}\partial_r \widetilde{I}_{n-2}-\frac{1}{n-3}.
\end{eqnarray*}
Using (\ref{inductionstep}), we have
\begin{eqnarray*}\partial_r^3\widetilde{I}_{n-2} &=&-\frac{4(n-4)r}{(1-r^2)^2}+\left(\frac{n-3}{r^2}-\frac{2(n-4)(1+r^2)}{(1-r^2)^2}\right)\partial_r\widetilde{I}_{n-2}
\\&&-\left(\frac{n-3}{r}+\frac{2(n-4)r}{1-r^2}\right)\partial_r^2 \widetilde{I}_{n-2},	
\end{eqnarray*}
and
\begin{eqnarray*}
 &&\frac{1-r^4}{4r(n-3)}\partial_r^3\widetilde{I}_{n-2}
\\ &=&-\frac{(n-4)(1+r^2)}{(n-3)(1-r^2)}+\left(\frac{1-r^4}{4r^3}-\frac{(n-4)(1+r^2)^2}{2r(n-3)(1-r^2)}\right)\partial_r\widetilde{I}_{n-2}\\
&&-\left(\frac{1-r^4}{4r^2}+\frac{(n-4)(1+r^2)}{2(n-3)}\right)\partial_r^2\widetilde{I}_{n-2}	
\end{eqnarray*}

As a result, we have
\begin{eqnarray*}
\partial_r^2\widetilde{I}_{n} &=& \left(\frac{-(n+1)r^4+2(n-2)r^2-(n-1)}{4r^2(n-3)}\right)	\partial_r^2\widetilde{I}_{n-2}\\
&&+\left(\frac{(n-1)-(3n-9)r^2-(5n-13)r^4-(n-11)r^6}{4r^3(n-3)(1-r^2)}\right)\partial_r\widetilde{I}_{n-2}\\
&&- \frac{(n-3)+(n-5)r^2}{(n-3)(1-r^2)}
\end{eqnarray*}

Hence
\begin{eqnarray*}
&&\left(\frac{1-r^2}{2}\right)^2\partial_r^2\widetilde{I}_{n}
\\ &=& \left(\frac{-(n+1)r^4+2(n-2)r^2-(n-1)}{4r^2(n-3)}\right)\left(\frac{1-r^2}{2}\right)^2	\partial_r^2\widetilde{I}_{n-2}\\
&&+\left(\frac{(n-1)-(3n-9)r^2-(5n-13)r^4-(n-11)r^6}{8r^3(n-3)}\right)\frac{1-r^2}{2}\partial_r\widetilde{I}_{n-2}\\
&&- \frac{(n-3)+(n-5)r^2}{2(n-3)}\left(\frac{1-r^2}{2}\right)
\end{eqnarray*}

\begin{eqnarray*}
&&\left(\frac{1-r^2}{2}	\right)^2\frac{n-1}{r}\partial_r\widetilde{I}_n 
\\&=& 
\frac{(n-1)(1-r^4)}{4r^2(n-3)}\left(\frac{1-r^2}{2}\right)^2\partial_r^2\widetilde{I}_{n-2}\\
&&+\left(\frac{(n-1)(1-r^2)^2}{4r}-\frac{(n-1)(1-r^2)^2(1+3r^4)}{16r^3(n-3)}\right)\partial_r\widetilde{I}_{n-2}\\
&&-\frac{(n-1)(1-r^2)^2}{4(n-3)}
\end{eqnarray*}

\begin{eqnarray*}
&&\frac{(n-2)r(1-r^2)}{2}\partial_r \widetilde{I}_n
\\ & =& \frac{(n-2)(1+r^2)}{2(n-3)}\left(\frac{1-r^2}{2}\right)^2\partial_r^2\widetilde{I}_{n-2}\\
&& +\left(1-\frac{1+3r^4}{4(n-3)r^2}\right)\frac{(n-2)r(1-r^2)}{2}\partial_r\widetilde{I}_{n-2}\\
&& -\frac{(n-2)r^2(1-r^2)}{2(n-3)}
\end{eqnarray*}

Adding them up, we get
\begin{eqnarray*}
\Delta_{\mathbb{H}} \widetilde{I}_n &=& \frac{(n-2)-r^2}{n-3}\left(\frac{1-r^2}{2}\right)	^2\partial_r^2 \widetilde{I}_{n-2}\\
&&+\frac{-(n-5)r^4+(n^2-8n+13)r^2+(n-2)(n-3)}{2r(n-3)}\left(\frac{1-r^2}{2}\right)\partial_r\widetilde{I}_{n-2}\\
&&- \frac{(1-r^2)((n-2)+(n-4)r^2)}{2(n-3)}
\end{eqnarray*}

Note that we have
\begin{eqnarray*}
&&\frac{(n-2)-r^2}{(n-3)}\cdot\frac{(n-3)+(n-5)r^2}{2r}
\\&=&\frac{-(n-5)r^4+(n^2-8n+13)r^2+(n-2)(n-3)}{2r(n-3)},	
\end{eqnarray*}
as a result, we have

\begin{eqnarray*}
&&\Delta_{\mathbb{H}} \widetilde{I}_n
\\ &=& \frac{(n-2)-r^2}{n-3} \left(\left(\frac{1-r^2}{2}\right)	^2\partial_r^2 \widetilde{I}_{n-2} +\frac{(n-3)+(n-5)r^2}{2r} \left(\frac{1-r^2}{2}\partial_r \widetilde{I}_{n-2}\right) \right)\\
&& -\frac{(1-r^2)((n-2)+(n-4)r^2)}{2(n-3)}
\end{eqnarray*}

Use induction assumption in the form
\[\left(\frac{1-r^2}{2}\right)^2\partial_r^2 \widetilde{I}_{n-2}+\frac{1-r^2}{4r}((n-3)+(n-5)r^2)\partial_r \widetilde{I}_{n-2}
=-\frac{(n-4)(1-r^2)}{2},\]

Then we get
\[\Delta_{\mathbb{H}} \widetilde{I}_n = -\frac{(n-2)(1-r^2)}{2}.\]
Which finishes the proof.
\end{proof}

\subsection{Boundary Value Through Induction}\label{SubsectionBoundaryValueThroughInduction}
Using the induction relation (\ref{InInductionRelation}), we can find boundary value for $\widetilde{I}_n$. We have the following lemma.

\begin{lemma}\label{LemmaBoundaryValueThroughInduction}
For $n>2$ we have
\[\lim_{r\to 1}\widetilde{I}_n=0,\]	
\[\lim_{r\to 1}\partial_r\widetilde{I}_n=-1,\]
and
\[\lim_{r\to 1}\frac{\widetilde{I}_n -0}{r-1}=-1.\]
\end{lemma}

\begin{proof}
For $n=2$, we have $\widetilde{I}_n =0$. For $n=3$ we have 
\[\widetilde{I}_3= \ln 4+\frac{(1-r)^2\ln (1-r)-(1+r)^2\ln(1+r)}{2r}.\]
It is easy to see that
\[\lim_{r\to 1}\widetilde{I}_3=0,\]
\[\lim_{r\to 1}\partial_r\widetilde{I}_3=-1,\]
and
\[\lim_{r\to 1}\frac{\widetilde{I}_3 -0}{r-1}=-1.\]

In general we use induction to prove that 
\[\lim_{r\to 1}\widetilde{I}_n=0\]
and
\[\lim_{r\to 1}\partial_r\widetilde{I}_n=C.\]
Here $C=0$ for $n=2$ and $C=-1$ for $n\geq 3$. Suppose the induction assumption holds for $n-2$, then we have
\[\lim_{r\to 1}\widetilde{I}_{n-2}=0\]
and
\[\lim_{r\to 1}\partial_r\widetilde{I}_{n-2}=C.\]

Now consider $\widetilde{I}_n$. By (\ref{InInductionRelation}), we have
\[\lim_{r\to 1}\widetilde{I}_{n}= \lim_{r\to 1}\frac{1-r^4}{4r(n-3)}\lim_{r\to 1}\partial_r\widetilde{I}_{n-2} +\lim_{r\to 1} \widetilde{I}_{n-2}+\lim_{r\to 1}\frac{1-r^2}{2(n-3)}. \]
Using the induction assumption for $\widetilde{I}_{n-2}$, it is easy to see that
\[\lim_{r\to 1}\widetilde{I}_{n}=0.\]

Start with (\ref{InInductionRelation}), take derivative with respect to $r$, we get
\[\partial_r \widetilde{I}_n=\frac{1-r^4}{4r(n-3)}\partial_r^2\widetilde{I}_{n-2}+\left(1-\frac{1+3r^4}{4r^2(n-3)}\right)\partial_r \widetilde{I}_{n-2}-\frac{r}{n-3}.\]
Use (\ref{inductionstep}) to substitute $\partial_r^2 \widetilde{I}_{n-2}$, then we get
\begin{eqnarray*}
\partial_r \widetilde{I}_n &=& \frac{1-r^4}{4r(n-3)}\left(-\frac{2(n-4)}{1-r^2}-\frac{2(n-4)r}{1-r^2}\partial_r\widetilde{I}_{n-2}-\frac{n-3}{r}\partial_r\widetilde{I}_{n-2}\right)\\
&&+\left(1-\frac{1+3r^4}{4r^2(n-3)}\right)\partial_r \widetilde{I}_{n-2}-\frac{r}{n-3}.	
\end{eqnarray*}
Use the induction assumption and take limit we have
\[\lim_{r\to 1} \partial_r\widetilde{I}_n=-1.\]

Now we use induction to show that
\[\lim_{r\to 1}\frac{\widetilde{I}_n(r)-0}{r-1}=C,\]
again, we have $C=0$ for $n=2$ and $C=-1$ for $n\geq 3$.
Suppose it is true for $n-2$ then we have the induction assumption
\[\lim_{r\to 1}\frac{\widetilde{I}_{n-2}(r)-0}{r-1}=C.\]
Use (\ref{InInductionRelation}), then we have
\[\frac{\widetilde{I}_n-0}{r-1}=-\frac{(1+r)(1+r^2)}{4r(n-3)}\partial_r \widetilde{I}_{n-2}+\frac{\widetilde{I}_{n-2}-0}{r-1}-\frac{1+r}{2(n-3)}.\]
Now take limit and use the induction assumption, then we have
\[\lim_{r\to 1}\frac{\widetilde{I}_{n}(r)-0}{r-1}=-\frac{C}{n-3}+C-\frac{1}{n-3}.\]
When $n\geq 5$, we have $\lim_{r\to 1} \partial_r\widetilde{I}_{n-2}=-1$ and $\lim_{r\to 1} \frac{\widetilde{I}_{n-2}-0}{r-1}=-1$, and as a result we have
\[\lim_{r\to 1}\frac{\widetilde{I}_{n}(r)-0}{r-1}=-1.\]
When $n=4$, we have $\lim_{r\to 1} \partial_r\widetilde{I}_{n-2}=0$ and $\lim_{r\to 1} \frac{\widetilde{I}_{n-2}-0}{r-1}=0$, and
\[\lim_{r\to 1}\frac{\widetilde{I}_{n}(r)-0}{r-1}=-\frac{1}{n-3}=-1.\]
This concludes the proof.
\end{proof}

\section{Uniqueness in the Limit Case Through the Method of Moving Spheres}\label{UniquenessThroughMovingSpheres}
In this section, we prove Theorem \ref{MainTheoremLimitCaseUniqueness}. For the convenience of calculation, it will be easier for us transform the inequality (\ref{InequalityLimitCase}) and its corresponding Euler-Lagrange equation to the upper half space using the transformation $\Psi:\R^n_+\to \B^n$ as given in (\ref{PsiDefinition}).

Note that in (\ref{InequalityLimitCase}) we have the inequality 
\[
	\|e^{\widetilde{I}_n+\widetilde{P}_{2-n}\widetilde{f}}\|_{L^n(\B^n)}\leq S_n\|e^{\widetilde{f}}\|_{L^{n-1}(\Sph^{n-1})},
\]
with Euler-Lagrange euquation
\begin{equation}\label{ELball}
	e^{(n-1)\widetilde{f}(\xi)}=\int_{\B^n}e^{n\widetilde{I}_n+n\widetilde{P}_{2-n}\widetilde{f}}\widetilde{p}_{2-n}(x,\xi)dx
\end{equation}

For any $\widetilde{f}\in L^{\infty}(\B^n)$, if we define a corresponding $f$ such that
\begin{equation}\label{flimitdef}
f(w)=\widetilde{f}\circ\Psi(w)+\ln|\Psi'(w)|	,
\end{equation}
then by change of variable we have
\begin{eqnarray*}
\left\|e^{\widetilde f}\right\|_{L^{n-1}(\Sph^{n-1})}^{n-1} &=&\int_{\Sph^{n-1}}e^{(n-1)\widetilde f}d\xi \\
&=&\int_{\R^{n-1}}e^{(n-1)(\widetilde{f}\circ\Psi+\ln|\Psi'(w)|)}dw\\
&=& \left\|e^{f}\right\|_{L^{n-1}(\R^{n-1})}^{n-1}.
\end{eqnarray*}
On the other hand, using change of variable, (\ref{pkPsi}) and (\ref{Inconformal}) we have

\begin{eqnarray*}
\left\|e^{\widetilde{I}_n+\widetilde{P}_{2-n} \widetilde{f}}\right\|_{L^{n}(B^n)}^n&=& \int_{\B^n}e^{n\widetilde{I}_n+n\widetilde{P}_{2-n}\widetilde{f}}dx\\ 
&=& \int_{\R^n_+}e^{n\widetilde{I}_n\circ\Psi+n(\widetilde{P}_{2-n}\widetilde{f})\circ \Psi+n\ln|\Psi'(y)|}dy
\\ &=&\int_{\R^n_+} e^{nP_{2-n}(\widetilde{f}\circ\Psi+\ln|\Psi'(w)|)}dy\\
&=& \left\|e^{P_{2-n}f}\right\|_{L^n(\R^{n}_+)}^n.
\end{eqnarray*}
Hence we have the inequality
\begin{equation}\label{limitupperhalf}
	\|e^{P_{2-n}f}\|_{L^n(\R^n_+)}\leq S_n\|e^{f}\|_{L^{n-1}(\R^{n-1})}.
\end{equation}
Moreover, if we assume that $\widetilde{f}\in L^\infty(\Sph^{n-1})$ is a solution to (\ref{ELball}), then the corresponding $f$ as defined in \ref{flimitdef} satisfies the Euler-Lagrange equation
\begin{equation}\label{ELup}
e^{(n-1)f(w)}=\int_{\R^n_+}e^{nP_{2-n}f}p_{2-n}(y,w)dy.	
\end{equation}

This section is organized as follows: we define several notation related to the inversion with respect to a sphere in Subsection \ref{SubsectionInversionSphereNotation}. We prove Theorem \ref{MainTheoremLimitCaseUniqueness} in Subsection \ref{SubsectionInversionWithRespectSphere} to Subsection \ref{SubsectionEndofMovingSphere}.  We start by assuming $\widetilde{f}\in C^1(\Sph^{n-1})$ is a solution  to (\ref{ELball}), then define the corresponding $f$ as in (\ref{flimitdef}). We will prove that the function $f$ as a solution to the equation (\ref{ELup}) is unique up to conformal transformation.

Note that this is different from directly proving uniqueness of smooth solutions to (\ref{ELup}). Since by starting with smooth solutions of (\ref{ELball}) we gain the asymptotic behavior $\ln|\Phi'(u)|$ as in (\ref{flimitdef}). This asymptotic behavior helps us to start the sphere in Subsection \ref{SubsectionStarttheSphere}.

\subsection{Notation}\label{SubsectionInversionSphereNotation}
For any $w\in\R^{n-1}$ and $\lambda\in \R$, define $v=(w,0)\in\R^n$. For all $y\in\R^n$ such that $y\neq v$ define the inversion with respect to a sphere centered at $v$ with radius $\lambda$ as
\begin{equation}\label{plvd}
\phi_{\lambda,v}(y)=v+\frac{\lambda^2}{|y-v|^2}(y-v).	
\end{equation}

Note that $\phi_{\lambda,v}(y)$ maps the upper half space into the itself, and $\phi_{\lambda,v}(y):\R^n_+\to \R^n_+$ is a conformal transformation. We use  $|\phi_{\lambda,v}'(y)|$ to denote the conformal factor such that
\[\phi_{\lambda,v}^\ast dy^2=|\phi_{\lambda,v}'(y)|^2 dy^2\]
Note that $|\phi_{\lambda,v}'(y)|^n$ is the Jacobian in $\R^n_+$. Through direct calculation using (\ref{plvd}), we can see that
\[|\phi_{\lambda,v}'(y)|^n=\frac{\lambda^{2n}}{|y-v|^{2n}},\]
and
\begin{equation}\label{plvConformalFactor}
	|\phi_{\lambda,v}'(y)|=\frac{\lambda^{2}}{|y-v|^{2}}
\end{equation}

 Note that since $\Psi\circ \phi_{\lambda,v}\circ\Psi^{-1}$ is a conformal transformation that maps $\B^n$ to itself, we have $\Psi\circ \phi_{\lambda,v}\circ\Psi^{-1}\in SO(n,1)$. For any $f:\R^{n-1}\to\R$, we can define $f_{\lambda,v}$ as the conformal transformation of $f$ under $\phi_{\lambda,v}$ such that
\begin{equation}\label{flvlimitdef}
f_{\lambda,v}=f\circ\phi_{\lambda,v}+\ln|\phi_{\lambda,v}'|.	
\end{equation}

In addition we define
\[\B_{\lambda,v}=\{y\in\R^n: |y-v|<\lambda\},\]
with $\B^+_{\lambda,v}=\B_{\lambda,v}\cap\R^n_+$ and $\overline{\B}^+_{\lambda,v}$ denotes the closure of $\B^+_{\lambda,v}$ in $\R^n$. We also use $\overline{\R}^n_+$ to denote the closure of $\R^n_+$ in $\R^n$.

\subsection{Inversion with Respect to Spheres}\label{SubsectionInversionWithRespectSphere}
Suppose $\widetilde{f}\in C^\infty(\Sph^{n-1})$ is a solution to (\ref{ELball}), for $w\in \partial \R^n_+=\R^{n-1}$ define
\[
f(w)=\widetilde{f}\circ\Psi (w)+\ln |\Psi'(w)|.	
\]
Through the discussion before (\ref{ELup}) we know that $f$ is a solution to (\ref{ELup}). Under transformation $\plv$ as defined in (\ref{plvd}) we can define $\flv$ as in (\ref{flvlimitdef})
\begin{eqnarray*}
\flv (w)&=& f\circ \plv (w)+\ln|\plv'(w)|	
\\&=& \widetilde{f}\circ \Psi\circ\plv(w)+\ln|\Psi'(\plv(w))|+\ln|\plv'(w)|
\\ &=& \widetilde{f}\circ \Psi\circ\plv(w)+\ln\frac{2\lambda^2}{(1+|v_0|^2)|w-v_0|^2+\lambda^4+2\lambda^2 \langle v_0,w-v_0\rangle}.
\end{eqnarray*}
Note that here $|\Psi'(\plv(w))|$ is as in (\ref{PsiConformalFactorExampleBoundary}) and $|\plv'(w)|$ is as in (\ref{plvConformalFactor}). The notation $\langle v_0,u-v_0\rangle$ denotes the Euclidean inner product in $\R^{n}$. 

In the special case $v_0=0$, we have
\begin{equation}\label{asympf}
f_{\lambda,0}(v)=\widetilde{f}\circ\Psi\circ \phi_{\lambda,0} (v)+\ln\frac{2\lambda^2}{\lambda^4+|v|^2}.	
\end{equation}
\begin{remark}\label{RemarkflvSmooth}
Note that for any $\widetilde{f}\in C^{\infty}(\Sph^{n-1})$ and any $\lambda>0$ the function $f_{\lambda,0}$ is smooth in $\R^{n-1}$. We can see this using the definition of $\plv$ and $\Psi$, for any $w\in \R^{n-1}$
\begin{equation*}
\begin{split}
\Psi\circ \phi_{\lambda,0}(w) &=\Psi \left(\frac{\lambda^2}{|w|^2}w\right)\\
 &=	\left(\frac{\frac{2\lambda^2}{|w|^2}w}{1+\frac{\lambda^4}{|w|^2}},\frac{-1+\frac{\lambda^4}{|w|^2}}{1+\frac{\lambda^4}{|w|^2}}\right)\\
 &=\left(\frac{2\lambda^2 w}{\lambda^4+|w|^2},\frac{\lambda^4-|w|^2}{\lambda^4+|w|^2}\right).
\end{split}	
\end{equation*}
If we define $\Psi\circ\phi_{\lambda,0}(0)=(0,...,0,1)\in\R^n$ then it is easy to see that the map 
\[\Psi\circ\phi_{\lambda,0}:\R^{n-1}\to \Sph^{n-1}\]
is smooth. The other parts of the function $f_{\lambda,0}$ is also smooth.

Through similar calculation we can also conclude that for any $\widetilde{f}\in C^{\infty}(\Sph^{n-1})$, any $\lambda>0$ and any $v_0\in \R^{n-1}$, the function $\flv$ is smooth in $\R^{n-1}$.
\end{remark}

On the other hand, for any $\widetilde{f}\in C^\infty(\Sph^{n-1})$ define $f$ as in (\ref{flimitdef}) then we have
\[P_{2-n}(f)=P_{2-n}(\widetilde{f}\circ\Psi)+P_{2-n}(\ln|\Psi' (w)|).\]
Using (\ref{Inconformal}) and Remark \ref{RemarkExtensionSpecialCase}, we see that
\begin{equation}\label{pfharmonic}
\begin{split}
P_{2-n}(f) &=P_{2-n}(\widetilde{f}\circ\Psi)+\widetilde{I}_n\circ\Psi(y)+\ln|\Psi'(y)|\\
&=(\widetilde{P}_{2-n}\widetilde{f})\circ\Psi+\widetilde{I}_n\circ\Psi(y)+\ln|\Psi'(y)|.	
\end{split}
\end{equation}

For $\flv$, using (\ref{PKphi}) and (\ref{Extension}), we have for any $y\in \R^n_+$
\begin{eqnarray*}
P_{2-n}\flv (y) &=&P_{2-n}(f)\circ \plv (y)+\ln|\plv'(y)|	
\\ &=& P_{2-n}(\widetilde{f}\circ\Psi)\circ\plv (y)+\widetilde{I}_n\circ\Psi\circ\plv(y)
\\&&+\ln|\Psi'(\plv(y))|+\ln|\plv'(y)|
\\ &=& P_{2-n}(\widetilde{f}\circ\Psi)\circ\plv (y)+\widetilde{I}_n\circ\Psi\circ\plv(y)\\
& &+\ln\frac{2\lambda^2}{(1+|v_0|^2)|y-v_0|^2+\lambda^4 +2\lambda^2\langle v_0, y-v_0\rangle +\lambda^2 y_n}.
\end{eqnarray*}
Using result from Remark \ref{RemarkExtensionSpecialCase} we have
\begin{equation}\label{pflvharmonic}
\begin{split}
	P_{2-n}\flv (y) &= \left(\widetilde{P}_{2-n}\widetilde{f}\right)\circ\Psi\circ\plv (y)+\widetilde{I}_n\circ\Psi\circ\plv(y)\\
&+\ln\frac{2\lambda^2}{(1+|v_0|^2)|y-v_0|^2+\lambda^4 +2\lambda^2\langle v_0, y-v_0\rangle +\lambda^2 y_n}.
\end{split}
\end{equation}

\begin{remark}
Note that when $\widetilde{f}\in C(\Sph^{n-1})$, both functions $P_{2-n}f$ and $P_{2-n}\flv$ are continuous in $\overline{\R}^n_+$. We can see this from (\ref{pfharmonic}) and (\ref{pflvharmonic}) using similar calculations as in Remark \ref{RemarkflvSmooth}.
\end{remark}

Again, note that . When $v_0=0$, we have
\begin{equation}\label{asymppf}
P_{2-n} f_{\lambda,0} (y)=\left(\widetilde{P}_{2-n}\widetilde{f}\right)\circ\Psi\circ \phi_{\lambda,0}(y)+\widetilde{I}_n\circ\Psi\circ\phi_{\lambda,0}(y)+\ln\frac{2\lambda^2}{|y|^2+\lambda^4+\lambda^2 y_n}	
\end{equation}

 We can show the following result, which will be used in subsequent parts.

\begin{lemma}\label{hyperbolicharmonic}
Suppose $\widetilde{f}\in C(\Sph^{n-1})$. Define $f$, $\flv$ as in (\ref{flimitdef}) and (\ref{flvlimitdef}) repectively. Define the coresponding extensions $P_{2-n}f$ and $P_{2-n}\flv$ as in (\ref{pfharmonic}), then $P_{2-n}f$ and $P_{2-n}\flv$ are harmonic in $\R^n_+$ with the standard Hyperbolic metric and with boundary values $f$ and $\flv$ respectively.
\end{lemma}

\begin{proof}
	We consider $P_{2-n}f$ firstly. Using (\ref{pfharmonic}) and (\ref{pkPsi}) we can see that
\begin{eqnarray*}
	P_{2-n}(f) &=&P_{2-n}(\widetilde{f}\circ\Psi)+\widetilde{I}_n\circ\Psi(y)+\ln|\Psi'(y)|
	\\ &=& (\widetilde{P}_{2-n}\widetilde{f})\circ\Psi+\widetilde{I}_n\circ\Psi(y)+\ln|\Psi'(y)|.
\end{eqnarray*}
From Proposition \ref{alphahamonic} we can see that $(\widetilde{P}_{2-n}\widetilde{f})\circ\Psi$ is harmonic in $\R^{n}_+$ with the Hyperbolic metric and that $\widetilde{P}_{2-n}\widetilde{f}\circ\Psi(u)=\widetilde{f}\circ\Psi(u)$ for all $u\in \partial\R^n_+=\R^{n-1}$.

From Lemma \ref{InLnHyperbolicHarmonic} we see that in $\B^n$ the fucntion $\widetilde{I}_n(x)+\ln\frac{1-2x_n+|x|^2}{2}$ is hyperbolic harmonic. Since $\Psi :\R^n_+\to \B^n$ is an isometry between two models of hyperbolic spaces, we have $\widetilde{I}_n\circ\Psi(y)+\ln|\Psi'(y)|$ is also harmonic in $\R^{n}_+$ with the Hyperbolic metric. Note that here we used the relation (\ref{EquationConformalFactorTransformation}).

From Lemma \ref{LemmaBoundaryValueThroughInduction} we see that $\widetilde{I}_n$ is continuous up to the boundary and that $\widetilde{I}_n\circ \Psi(w)=0$ for all $w\in \partial\R^n_+=\R^{n-1}$. As a result, we have $\widetilde{I}_n\circ\Psi(w)+\ln|\Psi'(w)|=\ln|\Psi'(w)|$ for all $w\in \partial \R^n_+=\R^{n-1}$. This finishes the proof for $P_{2-n}f$.

Now we consider $P_{2-n}\flv$. By Remark \ref{RemarkflvSmooth}, $\flv$ is continuous on $\partial \R^{n}_+$. Using the definition of $\flv$, (\ref{PKphi}) and (\ref{Extension}), we have
\[P_{2-n}\flv=P_{2-n}f\circ\plv+\ln|\plv'|.\]
Since $\plv$ is an isometry of $\R^n_{+}$ with hyperbolic metric, we have $P_{2-n}f\circ\plv$ is harmonic in $\R^{n}_+$ with the hyperbolic metric. On the other hand, using (\ref{plvConformalFactor}), by direct calculation we can see that $\ln|\plv'|$ is hyperbolic harmonic in $\R^{n}_+$.

From (\ref{pflvharmonic}) and similar calculations as in Remark \ref{RemarkflvSmooth} we see that $P_{2-n}\flv$ is continuous in $\overline{R}^n_+$. As a result, $P_{2-n}\flv$ is hyperbolic harmonic in $\R^n_+$ with boundary value $\flv$.

\end{proof}

Note that in this proof we only need the asymptotic behavior (\ref{asympf}) and (\ref{asymppf}) and Proposition \ref{alphahamonic}. We also need the regularity $\widetilde{f}\in C(\Sph^{n-1})$ in order to use Proposition \ref{alphahamonic}, but we do not need to assume  $\widetilde{f}$ to be a solution of (\ref{ELball}).

\subsection{Start the Sphere}\label{SubsectionStarttheSphere}
Now we start the moving sphere argument. The first step is called start the sphere; and we prove it in the following two propositions. The proof of Proposition \ref{startf} relies on the asymptotic behavior (\ref{asympf}) while the proof of Proposition \ref{startu} relies on the asymptotic behavior (\ref{asymppf}) and the maximum principle.

\begin{proposition}\label{startf} For any $\widetilde{f}\in C^\infty(\Sph^{n-1})$ any $\lambda>0$ and any $v_0\in \partial \R^n_{+}=\R^{n-1}$, define $f$ and $\flv$ as in (\ref{flimitdef}) and (\ref{flvlimitdef}) respectively.
Then there exists $\lambda_0>0$ depending only on $v_0$, such that for all $0<\lambda<\lambda_0$
\[e^{(n-1)\flv}(w)< e^{(n-1)f}(w),\]
for all $|w-v_0|>\lambda$.
\end{proposition}

\begin{proof}
Our proof uses similar ideas as in Lemma 3.1 of \cite{L} and Lemma 2.1 of \cite{LZ}. Our proof is even simpler since we have the asymptotic behavior (\ref{asympf}).
	
Consider in polar coordinate of $\R^{n-1}$, $w=(r,\theta)\in\R^{n-1}$ where $r\geq 0$ and $\theta\in\Sph^{n-2}$. Since $f\in C^{\infty}(\R^{n-1})$ (actually we only need $f$ to be $C^1$) and $e^{(n-1)f}>0$, there exists $r_0>0$ such that
\[\frac{d}{dr}\left(r^{n-1}e^{(n-1)f}(r,\theta)\right)>0\]
for all $0<r<r_0$ and all $\theta\in\Sph^{n-2}$.

As a result, when $0<\lambda<|w|<r_0$ we have
\[\left(\frac{\lambda^2}{|w|}\right)^{n-1}e^{(n-1)f}\left(\frac{\lambda^2 w}{|w|^2}\right)<|w|^{n-1}e^{(n-1)f}(w),\]
and hence using (\ref{flvlimitdef}) we have
\[e^{(n-1)f_{\lambda,0}}(w)<e^{(n-1)f}(w)\]
for all $0<\lambda<|w|<r_0$. (Note that here we used the fact that when $0<\lambda<|w|$, we have $\big|\frac{\lambda^2 w}{|w|}\big|^2=\frac{\lambda^2}{|w|}<|w|$.)

For $|u|> r_0$, from (\ref{flimitdef}) and (\ref{asympf}) we have
\[e^{(n-1)f}(w)=\left(\frac{2}{1+|w|^2}\right)^{n-1}e^{(n-1)\widetilde{f}\circ \Psi}(w)\]
and
\[e^{(n-1)f_{\lambda,0}}(w)=\left(\frac{2\lambda^2}{\lambda^4+|w|^2}\right)^{n-1}e^{(n-1)\widetilde{f}\circ\Psi}\left(\frac{\lambda^2 w}{|w|^2}\right).\]
Define $m=\inf_{\xi\in\Sph^{n-1}}e^{\widetilde{f}}(\xi)$ and $M=\sup_{\xi\in\Sph^{n-1}} e^{\widetilde{f}}(\xi)$. Note that we have $0<m\leq M$. Choose $\lambda_0$ small enough such that $0<\lambda_0<r_0$ and
\[m-\lambda_0^2 M>m/2,\]
and
\[r_0^2 m/2>\lambda_0^2 M.\]

Then we have for all $|w|>r_0$ and $0<\lambda<\lambda_0$
\begin{eqnarray*}
(m-\lambda^2 M)|w|^2 &>& (m-\lambda_0^2 M)r_0^2
\\ &>& mr_0^2/2
\\ &>& \lambda_0^2 M
\\ &>& \lambda^2 M
\\ &>& \lambda^2 (M-\lambda^2 m).
\end{eqnarray*}
As a result, we have for all $0<\lambda<\lambda_0<r_0<|w|$,
\[\frac{2m}{1+|w|^2}>\frac{2\lambda^2 M}{|w|^2+\lambda^4},\]
and hence
\begin{eqnarray*}
e^{(n-1)f}(w)	&=&\left(\frac{2}{1+|w|^2}\right)^{n-1}e^{(n-1)\widetilde{f}\circ\Psi}(w)
\\ &\geq &\left(\frac{2m}{1+|w|^2}\right)^{n-1}
\\ &>& \left(\frac{2\lambda^2 M}{|w|^2+\lambda^4}\right)^{n-1}
\\ &\geq & \left(\frac{2\lambda^2 }{|w|^2+\lambda^4}\right)^{n-1} e^{(n-1)\widetilde{f}\circ \Psi}\left(\frac{\lambda^2 w}{|w|^2}\right)
\\ &=& e^{(n-1)f_{\lambda,0}}(w).
\end{eqnarray*}

With the chosen $\lambda_0$, combining the two cases together, we get for all $0<\lambda<\lambda_0$ and all $|w|>\lambda$, we have
\[e^{(n-1)f}(w)>e^{(n-1)f_{\lambda,0}}(w).\]
\end{proof}

\begin{remark}
Note that
\[e^{(n-1)\flv}(w)< e^{(n-1)f}(w)\]	
for all $|w-v_0|>\lambda$ is equivalent to 
\[e^{(n-1)\flv}(w)> e^{(n-1)f}(w)\]
for all $|w-v_0|<\lambda$. Since for any $w$ such that $|w-v_0|<\lambda$, we can define $w_{\lambda,v_0}=\plv(w)$. Then we have
\begin{equation}\label{plvidentity}
\plv(w_{\lambda,v_0})=\plv\circ\plv (w)=w,	
\end{equation}

and
\[|w_{\lambda,v_0}-v_0|=|\plv(w)-v_0|=\frac{\lambda^2}{|w-v_0|}>\lambda.\]
As a result, using Theorem \ref{startf}, we have for $|w-v_0|<\lambda$
\begin{eqnarray*}
e^{(n-1)\flv}(w) &=&|\plv'(w)|^{n-1} e^{(n-1)f}(w_{\lambda,v_0})
\\ &>&|\plv'(w)|^{n-1} e^{(n-1)\flv}(w_{\lambda,v_0})	
\\ & =&|\plv'(w)|^{n-1}|\plv'(w_{\lambda,v_0})|^{n-1} e^{(n-1)f}(w)
\\ &=& e^{(n-1)f}(w).
\end{eqnarray*}
Note that in the last step we used 
\begin{equation}\label{plvjacobian}
	|\plv'(w)|^{n-1}|\plv'(w_{\lambda,v_0})|^{n-1} =1,
\end{equation}
 which follows from (\ref{plvidentity}) and chain rule.
\end{remark}

With the help of the previous remark, we can use maximum principle to prove similar results for $P_{2-n}f$.

\begin{proposition}\label{startu}
For any $\widetilde{f}\in C^\infty(\Sph^{n-1})$, define $f$ and $\flv$ as in (\ref{flimitdef}) and (\ref{flvlimitdef}) respectively.

For any $v_0\in \partial \R^n_+=\R^{n-1}$ there exists $\lambda_0>0$ such that for all $0<\lambda<\lambda_0$ we have
\[e^{nP_{2-n}f}(y)<e^{nP_{2-n}\flv}(y)\]
for all $|y-v_0|<\lambda$.
\end{proposition}

\begin{proof}

From Lemma \ref{hyperbolicharmonic} we see that $P_{2-n}f$ and $P_{2-n}\flv$ are Hyperbolic harmonic with boundary values $f$ and $\flv$ respctivly. By the discussion at the beginning of this section, all four functions $f$, $\flv$, $P_{2-n}f$ and $P_{2-n}\flv$ are continuous at the point $v_0\in\partial\R^n_+=\R^{n-1}$. Choose the same $\lambda_0$ as in Theorem \ref{startf}. Using the previous remark we have
\[e^{(n-1)f}(w)<e^{(n-1)\flv}(w),\]
for any $0<\lambda<\lambda_0$ and for all $w\in \R^{n-1}$ such that $|w-v_0|<\lambda$. Which is equivalent to
\[f(w)<\flv(w)\]
for all $w\in \R^{n-1}$ such that $|w-v_0|<\lambda$. On the other hand, using the definition of $\flv$, (\ref{PKphi}) and (\ref{Extension}), we can see 
\[P_{2-n}f(y)=P_{2-n}\flv (y)\]
for all $y\in\R^n_+$ such that $|y-v_0|=\lambda$. By maximum principle we can see 
\[P_{2-n}f(y)<P_{2-n}\flv (y)\]
	for all $|y-v_0|<\lambda$.
\end{proof}

Note that in this subsection we only used asymptotic behavior (\ref{asympf}), (\ref{asymppf}) and the maximum principle. We did not use the fact that $\widetilde{f}$ is a solution to the integral equation (\ref{ELball}).

\subsection{The Case $\bar{\lambda}_0=\infty$}
In the previous subsection we showed that $\lambda_0>0$ exists. In this subsection we will show that it can not go to infinity.

Define 
\begin{equation}\label{criticallambda}
\bar{\lambda}_0=\sup\{\lambda>0 \text{ such that }e^{(n-1)f}(w)<e^{(n-1)\flv}(w) \text{ for all }|w-v_0|<\lambda.\}	
\end{equation}
In the following lemma we show that $\bar{\lambda}_0$ can not equal to $\infty$.

\begin{lemma}
For any $f:\R^{n-1}\to\R$, such that $e^f\in L^{n-1}(\R^{n-1})$, and for all $v_0\in \partial \R^n_+=\R^{n-1}$, we have $\bar{\lambda}_0<\infty$.
\end{lemma}

\begin{proof}
We can prove this by contradiction following \cite{FKT}. Suppose for some $v_0\in \partial \R^n_+=\R^{n-1}$, we have $\bar{\lambda}_0=\infty$, then we can find a sequence $\lambda_i\to \infty$ such that
\[e^{(n-1)f}(w)<e^{(n-1)f_{\lambda_i,v_0}}(w)\]
for all $|w-v_0|<\lambda_i$ and for all $i$. For a given $i$, by the previous inequality, (\ref{plvjacobian}) and change of variable we have
\begin{eqnarray*}
\int_{\B^+_{\lambda_i,v_0}}e^{(n-1)f}(w)dw &<& \int_{\B^+_{\lambda_i,v_0}}	e^{(n-1)f_{\lambda_i,v_0}}(w)dw
\\ &=& \int_{\R^n_+\backslash \overline{\B}^+_{\lambda_i,v_0}} e^{(n-1)f}(w)dw.
\end{eqnarray*}
As a result, we have
\[0<\frac{1}{2}\int_{\R^{n-1}}e^{(n-1)f}(w)dw< \int_{|w-v_0|>\lambda_i}e^{(n-1)f}(w)dw.\]
But the right hand side of the inequality goes to zero as $\lambda_i\to \infty$ by dominated convergence theorem.
\end{proof}

Note that in the this proof we only need $e^f\in L^{n-1}(\R^{n-1})$ in order to use dominated convergence theorem. We don't need $\widetilde{f}$ to be a solution to the integral equation (\ref{ELball}).

\subsection{The Case $\bar{\lambda}_0<\infty$}\label{SubsectionEndofMovingSphere}
Note that since $\bar{\lambda}_0>0$, this is the last case we need to consider. In this subsection we will show that at the critical value $\bar{\lambda}_0$, we have
\[f_{\bar{\lambda}_0,v_0}=f.\]
But firstly we need to show three lemmas. The first lemma is about the kernel function $p_{2-n}(y,u)$.

\begin{lemma}\label{Kgreaterthanzero}
For any $v_0\in\partial\R^n_+=\R^{n-1}$	 and any $\lambda>0$, define $\plv$ as in (\ref{plvd}). Then for any $y\in \Blvp$ and any $w\in\partial\R^n_+=\R^{n-1}$ such that $|w-v_0|<\lambda$, we have
\begin{equation}
p_{2-n}(y,w)-p_{2-n}(\plv(y),w)>0
\end{equation}
\end{lemma}

\begin{proof}
	By (\ref{pkphi}), we have
	\[p_{2-n}(\plv(y),w)=p_{2-n}(y,\plv(w))|\plv'(w)|^{n-1},\]
	so we only need to prove 
	\[\frac{1}{|y-w|^{2n-2}}-\frac{|\plv'(w)|^{n-1}}{|y-\plv(w)|^{2n-2}}>0,\]
	 for any $y\in \Blvp$ and any $w\in\R^{n-1}$ such that $|w-v_0|<\lambda$.
	By direct calculation using (\ref{plvd}), we have
	\[|\plv'(w)|=\frac{\lambda^2}{|w-v_0|^2}.\]
So the proof follows from
\[
|y-\plv(w)|^2\frac{|w-v_0|^2}{\lambda^2}-|y-w|^2 =\frac{(\lambda^2-|w-v_0|^2)(\lambda^2-|y-v_0|^2)}{\lambda^2},
\]
which is positive when both $|w-v_0|<\lambda$ and $|y-v_0|<\lambda$.
	
\end{proof}

If we define 
\[K(v_0,\lambda;y,w)=p_{2-n}(y,w)-p_{2-n}(\plv(y),w)\]
 as in \cite{L} (right before Lemma 3.1 in \cite{L}), then we can show the following result about the derivative of $K$ with respect to $w$ (as in the proof of Lemma 3.2 in \cite{L}).
 
 \begin{lemma}\label{Kgradient}
\begin{equation}
\langle\overline{\nabla}_w K(v_0,\lambda;y,w), w-v_0\rangle\big|_{|w-v_0|=\lambda}=-\frac{2(n-1)c_{n,2-n}y_n^{n-1}}{|w-y|^{2n}} (|w-v_0|^2-|y-v_0|^2),
\end{equation}
	for all $y \in \R^n_+$.
 \end{lemma}

\begin{proof}
Here $\langle\overline{\nabla}_w K(v_0,\lambda;y,w), w-v_0\rangle$ denote the inner product in $\R^{n-1}$ (or $\R^n$) with respect to the Euclidean metric.
And $\overline{\nabla}_w$ denotes the gradient in $\R^{n-1}$ with the Euclidean metric. The subscript $w$ emphasizes the fact that the derivative is taken with respect to $w$.
The proof follows from direct calculation. As in the previous lemma, use \ref{pkphi} we can have
\begin{eqnarray*}
&&K(v_0,\lambda;y,w)
\\&=&c_{n,2-n}y_n^{n-1}\left(\frac{1}{|w-y|^{2n-2}}-\left(\frac{\lambda}{|w-v_0|}\right)^{2n-2}\frac{1}{\left|y-v_0-\frac{\lambda^2(w-v_0)}{|w-v_0|^2}\right|^{2n-2}}\right).	
\end{eqnarray*}
As a result, we can calculate $\overline{\nabla}_u K$ to be
\begin{eqnarray*}
	&&\frac{\overline{\nabla}_w K}{c_{n,2-n} y_n^{n-1}} 
	\\&=&-\frac{2(n-1)(w-y')}{|w-y|^{2n}}- \frac{2(n-1)\lambda^{2n-2}|w-v_0|^{2n-4}}{|(y-v_0)|w-v_0|^2-\lambda^2 (w-v_0)|^{2n-2}}(w-v_0)
	\\ &&+\frac{(n-1)\lambda^{2n-2}|w-v_0|^{2n-2}(4|y-v_0|^2|w-v_0|^2-4\lambda^2\langle y-v_0,w-v_0\rangle)}{|(y-v_0)|w-v_0|^2-\lambda^2 (w-v_0)|^{2n}}(w-v_0)
	\\ &&-\frac{2(n-1)\lambda^{2n}|w-v_0|^{2n-2}}{|(y-v_0)|w-v_0|^2-\lambda^2 (w-v_0)|^{2n}}(|w-v_0|^2(y-v_0)-\lambda^2 (w-v_0)).
\end{eqnarray*}
Here $\langle y-v_0, w-v_0\rangle$ denotes inner product in $\R^n$. Then for the inner product we have
\begin{eqnarray*}
\langle\overline{\nabla}_w K, w-v_0\rangle\big|_{|w-v_0|=\lambda}=	\frac{-2(n-1)c_{n,2-n}y_n^{n-1}}{|w-y|}(|w-v_0|^2-|y-v_0|^2).
\end{eqnarray*}

\end{proof}

In the next lemma, we use change of variable to rewrite equation $(\ref{ELup})$ into a new form.

\begin{lemma}\label{ELupdifferentform}
 For any $v_0\in\partial\R^n_+$ and any $\lambda>0$,  define the inversion $\plv$ as in (\ref{plvd}). For any $\widetilde{f}\in C^\infty(\R^{n-1})$ that is a solution to (\ref{ELball}). Define $f$ and $\flv$ as in (\ref{flimitdef}) and (\ref{flvlimitdef}) respectively,
 then we have
\begin{equation}
\begin{split}
&e^{(n-1)\flv}-e^{(n-1)f} \\ 
&=\int_{\Blvp}\left(e^{nP_{2-n\flv}}-e^{nP_{2-n}f}\right)(p_{2-n}(y,w)-p_{2-n}(\plv(y),w))dy	
\end{split}
\end{equation}

\end{lemma}

\begin{proof} Since $\widetilde{f}$ is a solution to \ref{ELball}, through change of variable we can see that $f$ is a solution to (\ref{ELup}).
Using equation (\ref{ELup}) we have
\begin{eqnarray*}
e^{(n-1)f} &=& \int_{\Blvp}	e^{nP_{2-n}f}p_{2-n}(y,w)dy+\int_{\R^n_+\backslash\Blvp}	e^{nP_{2-n}f}p_{2-n}(y,w)dy
\\&=&\int_{\Blvp}	\left(e^{nP_{2-n}f}p_{2-n}(y,w)+e^{nP_{2-n}\flv}p_{2-n}(\plv(y),w)\right)dy,
\end{eqnarray*}
where the second step follows from change of variable and (\ref{Extension}). Similarly
\begin{eqnarray*}
e^{(n-1)\flv} &=& 	\int_{\Blvp}e^{nP_{2-n}\flv}p_{2-n}(y,w)dy+\int_{\R^n_+\backslash\Blvp}e^{nP_{2-n}\flv}p_{2-n}(y,w)dy
\\ &=&\int_{\Blvp}\left(e^{nP_{2-n}\flv}p_{2-n}(y,w)+e^{nP_{2-n}f}p_{2-n} (\plv(y),w)\right)dy.
\end{eqnarray*}
Subtracting the first inequality from the second one, we get the desired result.

\end{proof}

We now prove $f=f_{\overline{\lambda}_0,v_0}$ by contradiction.

\begin{proposition}\label{PropositionMovingSphereStep3}
Suppose $f\in C^1(\R^{n-1})$ is a solution to (\ref{ELup}). For any $v_0\in \partial \R^n_+=\R^{n-1}$, with $\overline{\lambda}_0$	 defined as in (\ref{criticallambda}), $f$ and $f_{\overline{\lambda}_0,v_0}$ defined as in (\ref{flimitdef}) and (\ref{flvlimitdef}) respectively, we have
\[f(w)=f_{\overline{\lambda}_0,v_0}(w),\]
for all $w\in \R^{n-1}$.
\end{proposition}

\begin{proof}
We prove this by contradiction. Our proof is similar to a combination of Lemma 2.4 in \cite{LZ} and Lemma 3.2 in \cite{L}. Suppose there exists $v\in \partial \R^n_+=\R^{n-1}$ such that
\[f(v)<\fblv(v),\]
then by maximum principle (note that by Lemma \ref{hyperbolicharmonic} we can use maximum principle here), we have
\[P_{2-n}f(y)<P_{2-n}\fblv(y),\]
for all $y\in\Bblvp $. As a result, using (\ref{ELup}) and Lemma \ref{Kgreaterthanzero} we see that
\begin{equation}\label{fblvu}
f(w)<\fblv(w)	
\end{equation}

for all $w\in \partial \R^n_+=\R^{n-1}$ such that $|w-v_0|<\overline{\lambda}_0$. By (\ref{asympf}), $\fblv$ is continuous on $\partial\R^{n}_+$. Using compactness of $\overline{\B}^+_{\overline{\lambda}_0/2,v_0}$, there exists $\gamma>0$ such that
\[P_{2-n}\fblv(y)-P_{2-n}f(y)\geq \gamma>0\]
for all $y\in \overline{\B}^+_{\overline{\lambda}_0/2,v_0}$.

Consider a sequence $\{\lambda_i\}_{i=1}^\infty$ such that for each $i$ we have
\[\lambda_i>\lambda_{i+1}>\overline{\lambda}_0,\]
and that
\[\lambda_i\to \overline{\lambda}_0\]
as $i\to\infty$. We can also require that
\begin{equation}\label{gammahalfball}
P_{2-n}f_{\lambda_i,v_0}(y)-P_{2-n}f(y)\geq \frac{\gamma}{2}>0,	
\end{equation}
since by (\ref{asymppf}) we know that $P_{2-n}\flv(y)$ is continuous function of $\lambda$.

For each $i$ by compactness of $\overline{\B}^+_{\lambda_i,v_0}$, there exists $w_i\in \overline{\B}^+_{\lambda_i,v_0} $ such that
\[P_{2-n}f_{\lambda_i,v_0}(w_i)-P_{2-n}f(w_i)=\inf_{\overline{\B}^+_{\overline{\lambda}_0,v_0}} P_{2-n}f_{\lambda_i,v_0}-P_{2-n}f.\]
Since $\overline{\lambda}_0$ is the critical value, we must have
\[P_{2-n}f_{\lambda_i,v_0}(w_i)-P_{2-n}f(w_i)<0.\]
By maximum principle, we also know that $w_i\in \overline{\B}^+_{\lambda_i,v_0}\cap \partial \R^n_+ $, and that
\[f_{\lambda_i,v_0}(w_i)-f(w_i)=P_{2-n}f_{\lambda_i,v_0}(w_i)-P_{2-n}f(w_i)<0.\]
In addition, by \ref{gammahalfball} we have $\frac{\lambda_0}{2}\leq |w_i-v_0|<\lambda_i$. We have strict inequality because $f_{\lambda_i,v_0}(w)=f(w)$ for all $|w-v_0|=\lambda_i$.

Since $w_i$ is an interior minimum for $f_{\lambda_i,v_0}-f$ in $\overline{\B}^+_{\lambda_i,v_0}\cap\partial \R^n_+$, we have
\[\overline{\nabla}(f_{\lambda_i,v_0}-f)(w_i)=0.\]
Here we use $\overline{\nabla}$ to denote the gradient in $\R^{n-1}$ with the Euclidean metric. It is the same notation as in Lemma \ref{Kgradient}.

By compactness of $\overline{\B}^+_{\lambda_1,v_0}$, we can choose a subsequence (still denote it $\{w_i\}_{i=1}^\infty$ and $\{\lambda_i\}_{i=1}^{\infty}$) such that $w_i\to w_0$ for some $w_0\in \overline{\B}^+_{\overline{\lambda}_0,v_0}\cap \partial \R^n_+$. By (\ref{asympf}) we can see that both $\flv(w)$ and $\overline{\nabla} \flv(w)$ are continuous for $\lambda>0$ and $|w|\neq 0$. As a result, we can take limit $i\to \infty$ to get
\begin{equation}\label{fu0}
\fblv(w_0)-f(w_0)\leq 0	
\end{equation}
and
\begin{equation}\label{fu0gradient}
\overline{\nabla}(\fblv-f)(w_0)=0.
\end{equation}
Because of (\ref{fu0}) and (\ref{fblvu}) we have $|w_0-v_0|=\overline{\lambda}_0$. But by Lemma \ref{ELupdifferentform} and Lemma \ref{Kgradient}, we can see that
\begin{eqnarray*}
&& (n-1)e^{(n-1)f}(w_0)\langle \overline{\nabla}(\fblv-f)(w_0),w_0-v_0\rangle 
\\&=& \langle\overline{\nabla} (e^{(n-1)\fblv}-e^{(n-1)f})(w_0),w_0-v_0\rangle 
\\ &=&\int_{\Bblvp} \left(e^{nP_{2-n}\fblv}-e^{nP_{2-n}f}(y)\right)\langle\overline{\nabla}_w K(v_0,\overline{\lambda}_0;y,w_0), w_0-v_0\rangle dy
\\ &=& -2(n-1)c_{n,2-n}\int_{\Bblvp} \left(e^{nP_{2-n}\fblv}-e^{nP_{2-n}f}(y)\right) \frac{y_n^{n-1}}{|w_0-y|^{2n}} (\overline{\lambda}_0^2-|y-v_0|^2)dy
\\ &<& 0.
\end{eqnarray*}
Which is a contradiction to (\ref{fu0gradient}).

\end{proof}
\subsection{Proof of Theorem \ref{MainTheoremLimitCaseUniqueness}}
For the proof of Theorem \ref{MainTheoremLimitCaseUniqueness} we need the following Lemma proved by Li and Zhu \cite{LZ}
\begin{lemma}\cite[Lemma 2.5]{LZ}\label{LemmaLiZhu}
For any integer $n\geq 3$, suppose $f\in C^1(\R^{n-1})$ satisfying: for any $b\in \R^{n-1}$, there exists $\lambda_b\in \R$ such that
\[f(w)=\frac{\lambda_b^{n-2}}{|w-b|^{n-2}}f
\left(\frac{\lambda_b^2}{|w-b|^2}(w-b)+b\right),\text{ for all } w\in\R^{n-1}\backslash\{b\}.\]
Then for some $a\geq 0$, $d>0$, $w_0\in \R^{n-1}$,
\[f(w)=\left(\frac{a}{|w-w_0|^2+d}\right)^{(n-2)/2},\text{ for all }w\in \R^{n-1},\]
or
\[f(w)=-\left(\frac{a}{|w-w_0|^2+d}\right)^{(n-2)/2},\text{ for all }w\in \R^{n-1}.\]
\end{lemma}

We restate Theorem \ref{MainTheoremLimitCaseUniqueness} here for the convenience of the reader:
\begin{theorem}
For any integer $n\geq 2	$, if $\widetilde{f}\in L^\infty(\Sph^{n-1})$ satisfies the equation
\[e^{(n-1)\widetilde{f}(\xi)}=\int_{\B^n}e^{n\widetilde{I}_n +n\widetilde{P}_{2-n}\widetilde{f}}\widetilde{p}_{2-n}(x,\xi)dx,\]
then for all $\xi\in\Sph^{n-1}$
\[\widetilde{f}(\xi)=\ln \frac{1-|\zeta|^2}{|\xi-\zeta|^2}+C_n,\]
where $\zeta \in\B^{n}$ and $C_n=-\frac{1}{n-1} \ln \left|\Sph^{n-1}\right|$ is a constant. Here $\left|\Sph^{n-1}\right|$ denotes the volume of the standard sphere.
\end{theorem}

\begin{proof}
By Proposition \ref{AppendixPropositionRegularity2} in the appendix we know that $\widetilde{f}\in C^1(\Sph^{n-1})$. For $w\in \R^{n-1}$ define $f(w)$ as in (\ref{flimitdef}) then we have $f\in C^1(\R^{n-1})$. By discussion before (\ref{ELup}) we know that $f$ satisfies the Euler-Lagrange equation (\ref{ELup}). Then by Proposition \ref{PropositionMovingSphereStep3} we know that for any $v_0\in \partial \R^n_+=\R^{n-1}$, there exists $\overline{\lambda}_0>0$ depending on $v_0$, such that 
\[f(w)=f_{\overline{\lambda}_0,v_0}(w),\]
for all $w\in \R^{n-1}$.
Note that here 
\begin{equation*}
\begin{split}
f_{\overline{\lambda}_0,v_0}(w) &=f\circ \phi_{\overline{\lambda}_0,v_0}(w)+\ln|\phi_{\overline{\lambda}_0,v_0}'(w)|	\\
&=f\left(\frac{\overline{\lambda}_0^2}{|w-v_0|^2}(w-v_0)+v_0\right)+\ln\frac{\overline{\lambda}_0^2}{|w-v_0|^2}
\end{split}	
\end{equation*}
is as in (\ref{flvlimitdef}).

Define
\[\varphi(w)=e^{\frac{n-2}{2} f(w)},\]
then it is easy to check that $\varphi$ satisfies the assumption of Lemma \ref{LemmaLiZhu}. As a result we have
\[\varphi(w)=\left(\frac{a}{|w-w_0|^2+d}\right)^{(n-2)/2}\]
for some $a> 0$, $d>0$ and $w_0\in \R^{n-1}$. From this we know that
\[f(w)=-\ln (|w-w_0|^2+d)+\ln a.\]
Define
\[\zeta= \Psi(w_0,\sqrt{d}),\]
where $(w_0,\sqrt{d})$ is a point in $\R^n_+$. Then we have
\[\widetilde{f}(\xi)=\ln\frac{1-|\zeta|^2}{|\xi-\zeta|^2}+C_n, \text{ for all }\xi\in \Sph^{n-1}.\]
Note that here the constant $C_n=-\frac{1}{n-1} \ln \left|\Sph^{n-1}\right|$ is determined by the restriction 
\[\left\|e^{\widetilde{f}}\right\|_{L^{n-1}(\Sph^{n-1})}=1.\]
This finishes the proof.
\end{proof}

\appendix
\section{REgularity}
For any integer $n\geq 3$ any $\alpha\in (2-n,1)$ and any $p>\frac{2(n-1)}{n-2+\alpha}$ we want to prove regularity of the function $\widetilde{f}\in L^p(\Sph^{n-1})$ such that $\widetilde{f}$ is a solution to the integral equation for any $\xi\in\Sph^{n-1}$
\[\int_{\B^n}\pka(x,\xi)\left(\PKA \widetilde{f} \right)(x)^{\frac{n+2-\alpha}{n-2+\alpha}}=\widetilde{f}(\xi)^{p-1}.\]
Here $\pka (x,\xi)$ is the Poisson kernel in the unit ball as in (\ref{PoissonKernelUnitBall}). The normalizing constant $\cna$ is as in (\ref{cna}).

The corresponding integral equation in the upper half space is
\[\left(\widetilde{f}\circ\Psi(w)\right)^{p-1}=\int_{\R^n_+}\upka(y,w)|\Psi'(w)|^{\frac{\alpha-n}{2}}\left((\UPKA f)(y)\right)^{\frac{n+2-\alpha}{n-2+\alpha}}dy,\]
here $\upka(y,w)$ is the Poisson kernel in the upper half space is given in (\ref{PoissonKernelUpperHalfSpace}). 

The integral equation in the upper half space can be simplified to
\begin{equation}\label{IntegralEquationP}
\left(f(w)\right)^{p-1}|\Psi'(w)|^{\frac{n-\alpha}{2}-
\frac{(p-1)(n-2+\alpha)}{2}}=\int_{\R^n_+}\upka(y,w)\left((\UPKA f)(y)\right)^{\frac{n+2-\alpha}{n-2+\alpha}}dy,
\end{equation}

\subsection{From $L^p$ to $L^\infty$ }
\begin{proposition}
For any integer $n\geq 3$, any $\alpha\in (2-n,1)$ and any $\frac{2(n-1)}{n-2+\alpha}<p<\infty$ suppose $\widetilde{f}\in L^p(\Sph^{n-1})$ is a solution to the Euler-Lagrange equation
\[\int_{\B^n}\pka (x,\xi)\left(\left(\PKA \widetilde{f}\right)(x)\right)^{\frac{n+2-\alpha}{n-2+\alpha}}dx=\widetilde{f}(\xi)^{p-1},\]
then we have $\widetilde{f}\in L^\infty(\Sph^{n-1})$
\end{proposition}

\begin{proof}
Note that the operator $\PKA: L^p(\Sph^{n-1})\to L^{\frac{2n}{n-2+\alpha}}(\B^n)$ is bounded and compact when $\frac{2(n-1)}{n-2+\alpha}<p<\infty$. For $\widetilde{u}\in L^q(\B^n)$ where $1\leq q<n$, define the operator $T_\alpha$
\[\widetilde{T}_\alpha \widetilde{u}=\int_{\B^n}\pka (x,\xi)\widetilde{u}(x)dx. \]
Using a duality argument as in \cite{HWY2} we can prove
\begin{equation}\label{AppendixEquationDualityHWY2}
\|\widetilde{T}_\alpha \widetilde{u}\|_{L^{\frac{(n-1)q}{n-q}}(\Sph^{n-1})}\leq \|\widetilde{u}\|_{L^q(\B^n)}	
\end{equation}

Suppose $\widetilde{f}\in L^p(\Sph^{n-1})$ is a solution to the Euler-Lagrange equation
\[\int_{\B^n}\pka (x,\xi)\left(\left(\PKA \widetilde{f}\right)(x)\right)^{\frac{n+2-\alpha}{n-2+\alpha}}dx=\widetilde{f}(\xi)^{p-1}.\]
From Proposition \ref{WeakEstimate} we see that
\[\PKA \widetilde{f}\in L^{\frac{np}{n-1}}(\B^n).\]
As a result if we define $\gamma=\frac{n-2+\alpha}{n+2-\alpha}$, then we have
\[(\PKA \widetilde{f})^{1/\gamma}\in  L^{\frac{np\gamma}{n-1}}(\B^n). \]
Using (\ref{AppendixEquationDualityHWY2}) we have
\[\widetilde{f}^{p-1}\in L^{\frac{\gamma p(n-1)}{n-1-\gamma p}}(\Sph^{n-1}),\]
and hence
\[\widetilde{f}\in L^{\frac{\gamma p(p-1)(n-1)}{n-1-\gamma p}}(\Sph^{n-1})\] 

If we keep going we can have $\widetilde{f}\in L^\infty(\Sph^{n-1})$.
\end{proof}

\subsection{Derivative of $\pka$ with respect to $x$}
Next, we want to prove regularity for $f$ in using the same idea as in the book by Gilbart-Trudinger \cite[Chapter 4]{GT}. (The way they handle the Newtonian Potential) So firstly, we take derivative of $\pka$ with respect to $x$ and with respect to $\xi$.
Through direct calculation we get
\[
\partial_{y_i}\upka(y,w)=\begin{cases}
-(n-\alpha)\frac{y_n^{1-\alpha}(y_i-w_i)}{|y-w|^{n-\alpha+2}},\ i\neq n\\
-(n-\alpha)\frac{y_n^{2-\alpha}}{|y-w|^{n-\alpha+2}}+(1-\alpha)\frac{y_n^{-\alpha}}{|y-w|^{n-\alpha}},\ i= n
\end{cases}
\]

\[
\partial_{w_i}\upka(y,w)=
(n-\alpha)\frac{y_n^{1-\alpha}(y_i-w_i)}{|y-w|^{n-\alpha+2}},\ i=1,\ 2,...,\ n-1\]

\subsection{$C^\beta$ Regularity for $F$}

Suppose we have the integral equation
\[F(w)=\cna \int_{\R^n_+}\frac{y_n^{1-\alpha}}{|y-w|^{n-\alpha}}U(y)dy,\]

where $y=(y',y_n)\in \R^n_+$ and $y',w\in \R^{n-1}$. We assume $U(y)\in L^\infty(\R^n_+)$ and $U(y)>0$ for all $y\in \R^n_+$.

For any $R>0$ we can write
\begin{eqnarray*}
F(w)&=&\cna \int_0^R\int_{\R^{n-1}}\frac{y_n^{1-\alpha}}{|y-w|^{n-\alpha}}U(y)dy'dy_n\\
&+&\cna \int_R^\infty\int_{\R^{n-1}}\frac{y_n^{1-\alpha}}{|y-w|^{n-\alpha}}U(y)dy'dy_n.\\
\end{eqnarray*}
Define
\begin{equation}\label{FRdefinition}
F_R(w)=\cna \int_0^R\int_{\R^{n-1}}\frac{y_n^{1-\alpha}}{|y-w|^{n-\alpha}}U(y)dy'dy_n.	
\end{equation}

It is easy to see that $F-F_R\in C^\infty(\R^{n-1})$, since there is no local singularity, and the singularity at $\infty$ is still summable after taking derivatives.

\begin{lemma}\label{CalphaF}
For any $2-n\leq\alpha<1$. For any $R>0$ and for $U\in L^\infty (R^n_+)$, and $U>0$, define $F_R$ as in (\ref{FRdefinition}).
Then for any $\beta$ such that $0<\beta<1$, we have $F_R\in C^\beta_{loc} (\R^{n-1})$.
\end{lemma}

\begin{proof}
For any $v,\ w\in R^{n-1}$,  consider
\[F_R(w)-F_R(v)=\cna \int_0^R\int_{\R^{n-1}}\left(\frac{y_n^{1-\alpha}}{|y-w|^{n-\alpha}}-\frac{y_n^{1-\alpha}}{|y-v|^{n-\alpha}}\right)U(y)dy'dy_n.\]
Define $r=|v-w|$. Suppose we have $0<r<\frac{R}{2}$. Define
\[A=\{y\in\R^n_+: \text{ such that }|y-v|<2r \text{ and } |y-w|<2r\}.\]
Then we can write

\begin{eqnarray*}
F_R(w)-F_R(v)&=&\cna \int_0^R\int_{A}\left(\frac{y_n^{1-\alpha}}{|y-w|^{n-\alpha}}-\frac{y_n^{1-\alpha}}{|y-v|^{n-\alpha}}\right)U(y)dy'dy_n\\
	&+&\cna \int_0^R\int_{\R^{n-1}\backslash A}\left(\frac{y_n^{1-\alpha}}{|y-w|^{n-\alpha}}-\frac{y_n^{1-\alpha}}{|y-v|^{n-\alpha}}\right)U(y)dy'dy_n\\
	&:=& I+II
\end{eqnarray*}

For $I$ we have
\begin{eqnarray*}
 &&\bigg|\cna  \int_0^R\int_{A}\left(\frac{y_n^{1-\alpha}}{|y-w|^{n-\alpha}}-\frac{y_n^{1-\alpha}}{|y-v|^{n-\alpha}}\right)U(y)dy'dy_n	\bigg|\\
&\leq & \CNA\int_0^R\int_{A}\frac{y_n^{1-\alpha }|y-w|^{\beta}}{|y-w|^{n-\alpha+\beta}}U(y)dy'dy_n\\
 && +\CNA\int_0^R\int_{A}\frac{y_n^{1-\alpha }|y-v|^{\beta}}{|y-v|^{n-\alpha+\beta}}U(y)dy'dy_n\\
&\leq & 2^{\beta}|w-v|^{\beta}\CNA \|U\|_{L^\infty (\R^n_+)}\bigg(\int_0^R\int_{A}\frac{y_n^{1-\alpha }}{|y-w|^{n-\alpha+\beta}}dy'dy_n\\
 && +\int_0^R\int_{A}\frac{y_n^{1-\alpha }}{|y-v|^{n-\alpha+\beta}}dy'dy_n \bigg),\\
 \end{eqnarray*}
in the last inequality we used the fact that $|y-w|<2r=2|v-w|$, $|y-w|<2|v-w|$ and $U>0$ in $\R^n_+$. By change of variable, choose $z'=\frac{y'-w}{y_n}$ with $dz'=\frac{dy'}{y_n^{n-1}}$ then we have
\begin{eqnarray*}
\int_0^R\int_{A}\frac{y_n^{1-\alpha }}{|y-w|^{n-\alpha+\beta}}dy'dy_n &\leq& \int_0^R\frac{dy_n}{y_n^\beta}\int_{\R^{n-1}}\frac{dz'}{(|z'|^2+1)^{\frac{n-\alpha+\beta}{2}}} \\	
&=& |\Sph^{n-2}|\int_0^R\frac{dy_n}{y_n^\beta}\int_0^\infty \frac{r^{n-2}dr}{(r^2+1)^{\frac{n-(\alpha-\beta)}{2}}}\\
&=& |\Sph^{n-2}| \frac{R^{1-\beta}}{1-\beta}\cdot \frac{\Gamma\left(\frac{1-(\alpha-\beta)}{2}\right)\Gamma(\frac{n-1}{2})}{2\Gamma\left(\frac{n-(\alpha-\beta)}{2}\right)}.
\end{eqnarray*}
Note that when $2-n\leq\alpha<1$ and $0<\beta <1$, we have $\frac{1-(\alpha-\beta)}{2}>0$ and $\frac{n-(\alpha-\beta)}{2}>0$. As a result, we have
\[\int_0^R\int_{A}\frac{y_n^{1-\alpha }}{|y-w|^{n-\alpha+\beta}}dy'dy_n \leq C(n,\alpha,\beta,R).\]
Similarly for $v$ we have
\[\int_0^R\int_{A}\frac{y_n^{1-\alpha }}{|y-v|^{n-\alpha+\beta}}dy'dy_n \leq C(n,\alpha,\beta,R).\]
As a result, we have
\[I\leq C(n,\alpha,\beta,R)\|U\|_{L^\infty(\R^n_+)} |w-v|^\beta.\]

Now we consider $II$. Notice that 
\[|D_w \upka(y,w)|\leq (n-\alpha)c_{n,\alpha}\frac{y_n^{1-\alpha}}{|y-w|^{n-\alpha+1}}.\]
Using mean value theorem we have: for some $w_0$ lying on the line segment between $v$ and $w$
\[|\upka(y,w)-\upka(y,v)|\leq|D_w\upka(y,w_0)||w-v|.\]
As a result, we have
\begin{eqnarray*}
|II|&\leq & \CNA \|U\|_{L^\infty(\R^n_+)}	\int_0^R\int_{\R^{n-1}\backslash A} |D_w\upka(y,w_0)||w-v| dy'dy_n\\
&=& \CNA \|U\|_{L^\infty(\R^n_+)}	\int_0^R\int_{\R^{n-1}\backslash A} \frac{y_n^{1-\alpha}}{|y-w_0|^{n-\alpha+1}}|w-v|^{1-\beta}|w-v|^\beta dy'dy_n.
\end{eqnarray*}
In $\R^n_+\backslash A$, we have
\[|y-w_0|\geq |w-v|=r.\]
We can prove this by contradiction. Suppose we have
\[|y-w_0|<r,\]
then by triangle inequality, we have
\[|y-w|\leq |w-w_0|+|y-w_0|\leq|w-v|+|y-w_0| \leq 2r. \]
Similarly for $v$, we have
\[|y-v|\leq 2r.\]
As a result, we have $y\in A$, which is a contradiction.

Now we have
\begin{eqnarray*}
|II|&\leq &  	\CNA \|U\|_{L^\infty(\R^n_+)}	\int_0^R\int_{\R^{n-1}\backslash A} \frac{y_n^{1-\alpha}}{|y-w_0|^{n-\alpha+1}}|y-w_0|^{1-\beta}|w-v|^\beta dy'dy_n \\
\\ &=& \CNA \|U\|_{L^\infty(\R^n_+)}	|w-v|^\beta\int_0^R\int_{\R^{n-1}\backslash A} \frac{y_n^{1-\alpha}}{|y-w_0|^{n-\alpha+\beta}}dy'dy_n.
\end{eqnarray*}
By previous calculation we have
\begin{eqnarray*}
\int_0^R\int_{\R^{n-1}\backslash A} \frac{y_n^{1-\alpha}}{|y-w_0|^{n-\alpha+\beta}}dy' dy_n &\leq &  \int_0^R\int_{\R^{n-1}} \frac{y_n^{1-\alpha}}{|y-w_0|^{n-\alpha+\beta}}dy' dy_n .\\
 &\leq& C(n,\alpha,\beta,R)	
\end{eqnarray*}
As a result, we have
\[|II|\leq C(n,\alpha,\beta,R) \|U\|_{L^{\infty}(\R^{n}_+)}|w-v|^\beta.\]

All together, for any $0<\beta<1$ we have
\[|F_R(w)-F_R(v)|\leq C(n,\alpha,\beta,R)\|U\|_{L^\infty (\R^n_+)}|w-v|^\beta,\]
which means $F_R\in C^\beta_{loc}(\R^{n-1})$.
\end{proof}

\subsection{$C^\beta$ Regularity for $U$}
\begin{lemma}\label{CalphaU}
For some $2-n\leq\alpha<1$, $0<\beta<1$ such that $\alpha +\beta<1$ and for some $f\in C^\beta_{loc}(\R^{n-1})\cap L^\infty(\R^{n-1})$, define
\[u(y)=\UPKA f (y)=c_{n,\alpha}\int_{\R^{n-1}}\frac{y_n^{1-\alpha}}{|y-w|^{n-\alpha}}f(w)dw.\]
Then for any $y'\in R^{n-1}$ we have
\[\lim_{y_n\to 0}u(y',y_n)=f(y'),\]
and $u\in C^\beta_{loc}(\overline{\R}^n_+)$.	
\end{lemma}

\begin{proof}
Through change of variable we see that
\[|u(y)|=c_{n,\alpha}\left|\int_{\R^{n-1}}\frac{f(y_n w+y')}{(|w|^2+1)^{\frac{n-\alpha}{2}}}dw\right|\leq \|f\|_{L^\infty(\R^{n-1})}\cna\int_{\R^{n-1}}\frac{dw}{(|w|^2+1)^{\frac{n-\alpha}{2}}}.\]
From which we can get
\[\|u\|_{L^\infty(\R^n_+)}\leq \|f\|_{L^\infty(\R^{n-1})}.\]

In addition, by dominated convergence theorem, we get
\[\lim_{y_n\to 0}u(y',y_n)=\cna \int_{\R^{n-1}}\lim_{y_n\to 0}\frac{f(y_n w+y')}{(|w|^2+1)^{\frac{n-\alpha}{2}}}dw=f(y').\]

Since it is easy to see that $U\in C^\infty(\R^n_+)$, we only need to show that $U$ is H\"{o}lder continuous up to the boundary.

For any $y\in \R^n_+$ and any $v\in \R^{n-1}$, define $D=\{w\in \R^{n-1} \text{ such that } |(y_n w+y')-v|<1 \}$. Note that $D$ is a ball in $\R^{n-1}$ centered at $\frac{v-y'}{y_n}$ with radius $\frac{1}{y_n}$. Also, note that for any $w\in D$ we have
\[|y_n w+y'|\leq 1+|v|\]

Choose $R>0$ large enough such that
\[1+|v|<R\]
then for all $w\in D$ we have
\[|y_n w+y'|\leq 1+|v|<R.\]

Since $f\in C^\beta_{loc}(\R^{n-1})$, we have
\[\sup_{|w_1|<R,\ |w_2|<R} \frac{|f(w_1)-f(w_2)|}{|w_1-w_2|^\beta}=\kappa<\infty.\]

 Consider
\begin{eqnarray*}
|u(y)-u(v)|&=&c_{n,\alpha}\int_{\R^{n-1}}\left|\frac{f(y_n w+y')-f(v)}{(|w|^2+1)^{\frac{n-\alpha}{2}}}\right|dw\\
&\leq & c_{n,\alpha}\int_{D}\left|\frac{f(y_n w+y')-f(v)}{(|w|^2+1)^{\frac{n-\alpha}{2}}}\right|dw\\
& & +c_{n,\alpha}\int_{\R^{n-1}\backslash D}\left|\frac{f(y_n w+y')-f(v)}{(|w|^2+1)^{\frac{n-\alpha}{2}}}\right|dw\\
&\leq & \cna\kappa \int_{D}\frac{|y_n w+(y'-v)|^\beta}{(|w|^2+1)^{\frac{n-\alpha}{2}}} dw\\
&& + 2\cna\|f\|_{L^\infty(\R^{n-1})}\int_{\R^{n-1}\backslash D}\frac{|y_n w+(y'-v)|^\beta}{(|w|^2+1)^{\frac{n-\alpha}{2}}} dw\\
&\leq & C(n,\alpha,\beta,\kappa,\|f\|_{L^\infty})\int_{\R^{n-1}}\frac{|y_n w|^\beta+|y'-v|^\beta}{(|w|^2+1)^{\frac{n-\alpha}{2}}}dw\\
&\leq & C(n,\alpha,\beta,\kappa,\|f\|_{L^\infty})\int_{\R^{n-1}} \left(\frac{|y_n|^\beta |w|^\beta}{(|w|^2+1)^{\frac{n-\alpha}{2}}} +\frac{
|y'-v|^\beta}{(|w|^2+1)^{\frac{n-\alpha}{2}}}\right) dw\\
&\leq &C(n,\alpha,\beta,\kappa,\|f\|_{L^\infty})|y-v|^\beta \int_{\R^{n-1}} \frac{|w|^\beta}{(|w|^2+1)^{\frac{n-\alpha}{2}}} dw\\
& & +C(n,\alpha,\beta,\kappa,\|f\|_{L^\infty}) |y-v|^\beta \int_{\R^{n-1}} \frac{1}{(|w|^2+1)^{\frac{n-\alpha}{2}}} dw.
\end{eqnarray*}
Note that in the third inequality, we used subadditivity of concave function. Also, note that
\[\int_{\R^{n-1}} \frac{|w|^\beta}{(|w|^2+1)^{\frac{n-\alpha}{2}}} dw=|\Sph^{n-2}| \frac{\Gamma\left(\frac{1-(\alpha+\beta)}{2}\right)\Gamma\left(\frac{n+\beta-1}{2}\right)}{2\Gamma\left(\frac{n-\alpha}{2}\right)},\]
and
\[\int_{\R^{n-1}} \frac{|w|^\beta}{(|w|^2+1)^{\frac{n-\alpha}{2}}} dw=|\Sph^{n-2}| \frac{\Gamma\left(\frac{1-\alpha}{2}\right)\Gamma\left(\frac{n-1}{2}\right)}{2\Gamma\left(\frac{n-\alpha}{2}\right)}.\]
Both integrals are finite when $2-n\leq\alpha<1$, $0<\beta<1$ and $\alpha+\beta<1$. As a result, we have
\[|u(y)-u(v)|\leq C(n,\alpha,\beta,\kappa,\|f\|_{L^\infty})|y-v|^\beta\]
for all $v\in \R^{n-1}$ such that $|v|+1<R$ and for all $y\in R^n_+$. Note that since $\kappa$ depends on $R$, the function $u$ is only locally H\"{o}lder continuous. We have $u\in C^\beta_{loc}(\overline{\R}^n_+)$.
\end{proof}

\subsection{$C^1$ Regularity for $F$}

For any $v\in \R^{n-1}=\partial \R^n$ and for any $R>0$ define $\B_{R,v}=\{y\in\R^n: |y-v|<R\}$, and $\B_{R,v}^+= \B_{R,v}\cap \R^n_+$. We also define the $n-1$ dimensional ball as $\B^{n-1}_{R,v}= \{w\in\R^{n-1}: |w-v|<R\}$. We also use notation $\B_{R}$, $\B_{R}^+$ and $\B_{R}^{n-1}$ to denote  $\B_{R,0}$, $\B_{R,0}^+$ and $\B_{R,0}^{n-1}$ respectively.

We prove $C^1_{loc}(\R^{n-1})$ regularity using the same argument as in \cite[Lemma 4.2]{GT}.

\begin{lemma}\label{C1F}
For any $2-n\leq\alpha<1$, $0<\beta<1$ and for any $U\in L^\infty (\R^n_+)\cap C^\beta_{loc}(\overline{\R}^n_+)$, $U>0$ such that
\[F(w)=\cna \int_{\R^n_+}\frac{y_n^{1-\alpha}}{|y-w|^{n-\alpha}}U(y)dy\]
is well defined.
We have $F\in C^1_{loc}(\R^{n-1})$.
\end{lemma}

\begin{proof}
Note that for any $R>0$ we can write
\[F(w)=\cna \int_{\B^+_R}\frac{y_n^{1-\alpha}}{|y-w|^{n-\alpha}}U(y)dy+\cna \int_{\R^n_+\backslash \B^+_R}\frac{y_n^{1-\alpha}}{|y-w|^{n-\alpha}}U(y)dy,\]
where 
\[\cna \int_{\R^n_+\backslash \B_{R,w}}\frac{y_n^{1-\alpha}}{|y-w|^{n-\alpha}}U(y)dy\in C^\infty_{loc}(\B^{n-1}_R),\] 
so we only need to consider
\[\cna \int_{\B^+_R}\frac{y_n^{1-\alpha}}{|y-w|^{n-\alpha}}U(y)dy.\]

Extend $U(y)$ to $\R^n$ by defining 
\[\overline{U}(y',y_n)=\begin{cases}U(y',y_n),\text{ for } y_n\geq 0,\\
U(y',-y_n),\text{ for } y_n<0.
\end{cases}
\] 
for all $y'\in \R^{n-1}$. Then it is easy to see that $\overline{U}\in L^\infty (\R^n)\cap C^\beta_{loc}(\R^n)$.
Extend $\upka(y,w)$ to $\R^n$ in the $y$ variable by defining
\[\overline{p}_\alpha(y,w)=\cna \frac{|y_n|^{1-\alpha}}{|y-w|^{n-\alpha}}.\]
As a result we have
\[\int_{\B^+_R}\upka (y,w)U(y)dy=\frac{1}{2}\int_{\B_R}\overline{p}_\alpha(y,w)\overline{U}(y)dy,\]
for all $w\in \B^{n-1}_R$.
Now we only need to consider $\int_{\B_R}\overline{p}_\alpha(y,w)\overline{U}(y)dy$. We can use the same argument as in Lemma 4.2 of \cite{GT} to prove that
\begin{eqnarray*}
D_i\left(\int_{\B_R}\overline{p}_\alpha(y,w)\overline{U}(y)dy\right) &=&\int_{\B_R} D_i\overline{p}_\alpha(y,w) \left(\overline{U}(y)-\overline{U}(w)\right)dy\\
&&-U(w)\int_{\partial \B_R}\overline{p}_\alpha(y,w)\nu_i(y)dS_{y},
\end{eqnarray*}
for $i=1,2,...,n-1$. Here derivative is taken with respect to $w$, $\partial\B_R$ is the boundary of $\B_R$ in $\R^n$, $dS_y$ is the standard measure on $\partial \B_R$.
\end{proof}

\subsection{Application to the Non-limit case}
Now we are ready to prove regularity results which was used in Theorem \ref{maininequality}.
\begin{proposition}\label{AppendixPropositionRegularity1}
For any integer $n\geq 3$, for any $2-n<\alpha<1$ and any $p>\frac{2(n-1)}{n-2+\alpha}$ suppose we have $\widetilde{f}\in L^\infty (\Sph^{n-1})$, $\widetilde{f}\geq 0$ and that $\widetilde{f}$ is a solution to the Euler-Lagrange equation
\[\int_{\B^n}\pka (x,\xi)\left(\left(\PKA \widetilde{f}\right)(x)\right)^{\frac{n+2-\alpha}{n-2+\alpha}}dx=\widetilde{f}(\xi)^{p-1},\]
then $\widetilde{f}\in C^1(\Sph^{n-1})$.
\end{proposition}
\begin{proof}
If for any $w\in \R^{n-1}$ we define
\[f(w)=\widetilde{f}\circ \Psi \left(w\right) \left( \frac{2}{1+\left\vert w\right\vert ^{2}}\right)
^{\frac{n-2+\alpha }{2}}\]
as in \ref{PsiTransformation}, then by Proposition \ref{PropositionPoissonKernelConformalTransformationPsi} and change of variable we can see that $f$ satisfies the the integral equation
\[
\left(f(w)\right)^{p-1}|\Psi'(w)|^{\frac{n-\alpha}{2}-
\frac{(p-1)(n-2+\alpha)}{2}}=\int_{\R^n_+}\upka(y,w)\left((\UPKA f)(y)\right)^{\frac{n+2-\alpha}{n-2+\alpha}}dy.
\]
As an immediate result, we have $f(w)>0$ for all $w\in \R^{n-1}$, since otherwise we have $f=0$.
Where $\|\UPKA f\|_{L^\infty (\R^{n}_+)}\leq \|f\|_{L^\infty(\R^{n-1})}$, and hence $\left((\UPKA f)(y)\right)^{\frac{n+2-\alpha}{n-2+\alpha}}\in L^\infty (\R^n_+)$. 

Using Lemma \ref{CalphaF}, choose some $\beta$ such that $0<\beta<1$, $\beta<p-1$ and $\frac{\beta}{p-1}+\alpha<1$, we get that 
\[\left(f(w)\right)^{p-1}|\Psi'(w)|^{\frac{n-\alpha}{2}-
\frac{(p-1)(n-2+\alpha)}{2}}\in C^\beta_{loc} (\R^{n-1}).\] 
Since $|\Psi'(w)|=\frac{2}{1+|w|^2}$ is smooth as a function of $w$ and that it is always positive, we have
\[f^{p-1}\in C^\beta_{loc}(\R^{n-1}),\]
and as a reuslt
\[f\in C^{\frac{\beta}{p-1}}_{loc}(\R^{n-1}).\]

Now apply Lemma \ref{CalphaU}, we get
\[\UPKA f\in C^\frac{\beta}{p-1}_{loc}(\R^n_+).\]
Finally apply Lemma \ref{C1F} to get
\[\left(f(w)\right)^{p-1}|\Psi'(w)|^{\frac{n-\alpha}{2}-
\frac{(p-1)(n-2+\alpha)}{2}}\in C^1_{loc}(\R^n),\]
and hence
\[\left(f(w)\right)^{p-1}\in C^1_{loc}(\R^n).\]
Lastly, since $f(w)>0$ for all $w\in \R^{n-1}$, we have
\[f(w)\in C^1_{loc}(\R^{n-1}).\]
Transform $f$ back to the unit ball we see that
\[\widetilde{f}\in C^1(\Sph^{n-1}).\]
\end{proof}

\subsection{Application to the Limit Case $\alpha=2-n$}

\begin{proposition}\label{AppendixPropositionRegularity2}
For any integer $n\geq 3$, suppose we have $\widetilde{f}\in L^\infty (\Sph^{n-1})$ and that $\widetilde{f}$ is a solution to the Euler-Lagrange equation
\begin{equation}\label{ELballAppendix}
	e^{(n-1)\widetilde{f}(\xi)}=\int_{\B^n}e^{n\widetilde{I}_n+n\widetilde{P}_{2-n}\widetilde{f}}\widetilde{p}_{2-n}(x,\xi)dx
\end{equation}
then $\widetilde{f}\in C^1(\Sph^{n-1})$.
\end{proposition}

\begin{proof}
Suppose $\widetilde{f}\in L^\infty(\Sph^{n-1})$ satisfies the integral equation (\ref{ELballAppendix}). Define 
\[f(w)=\widetilde{f}\circ\Psi(w)+\ln|\Psi'(w)|,\]
then from the discussion before (\ref{ELup}) we see that $f(w)$ satisfy the following integral equation:
\begin{equation}\label{ELupAppendix}
e^{(n-1)f(w)}=\int_{\R^n_+}e^{nP_{2-n}f}p_{2-n}(y,w)dy.	
\end{equation}
Since $e^{nP_{2-n}f}(y)=|\Psi'(y)|^ne^{n\widetilde{I}_n+n\widetilde{P}_{2-n}\widetilde{f}}\circ\Psi(y)\in L^\infty (\R^n_+)$, we can apply Lemma \ref{CalphaF} to get
\[e^{(n-1)f(w)}\in C^\beta_{loc}(\R^{n-1}).\]
Since $e^{(n-1)f(w)}>0$ for all $w\in \R^{n-1}$, we have
\[f(w)\in C^\beta_{loc}(\R^{n-1})\]
Now from Lemma \ref{CalphaU} we know that $P_{2-n}f\in C^\beta_{loc}(\overline{\R}^n_+)$. Using Lemma \ref{C1F} and equation (\ref{ELupAppendix}) we eventually get $f\in C^1_{loc}(\R^{n-1})$.
Eventually, transform back to the unit ball we get $\widetilde{f}\in C^1(\Sph^{n-1})$.
\end{proof}

\section*{Acknowledgement}
The author would like to thank his thesis advisor Xiaodong Wang for his guidance throughout this project.

\newpage

\bibliographystyle{amsplain}
 \bibliography{Ballanalysis}

\providecommand{\bysame}{\leavevmode\hbox to3em{\hrulefill}\thinspace}
\providecommand{\MR}{\relax\ifhmode\unskip\space\fi MR }
\providecommand{\MRhref}[2]{%
  \href{http://www.ams.org/mathscinet-getitem?mr=#1}{#2}
}
\providecommand{\href}[2]{#2}
\begin{thebibliography}{10}

\bibitem{AC}
Antonio~G. Ache and Sun-Yung~Alice Chang, \emph{{Sobolev trace inequalities of
  order four}}, Duke Mathematical Journal \textbf{166} (2017), no.~14, 2719 --
  2748.

\bibitem{BR}
E.~F. Beckenbach and T.~Rado, \emph{Subharmonic functions and surfaces of
  negative curvature}, Transactions of the American Mathematical Society
  \textbf{35} (1933), no.~3, 662--674.

\bibitem{CS}
Luis Caffarelli and Luis Silvestre, \emph{An extension problem related to the
  fractional {Laplacian}}, Communications in Partial Differential Equations
  \textbf{32} (2007), no.~8, 1245--1260.

\bibitem{C}
Torsten Carleman, \emph{{Zur theorie der minimalfl{\"a}chen}}, Mathematische
  Zeitschrift \textbf{9} (1921), no.~1, 154--160.

\bibitem{CC}
Jeffrey~S. Case and Sun-Yung Alice~Chang, \emph{{On fractional GJMS
  operators}}, Communications on Pure and Applied Mathematics \textbf{69}
  (2016), no.~6, 1017--1061.

\bibitem{Chang}
Sun-Yung~Alice Chang, \emph{Conformal invariants and partial differential
  equations}, Bull. Amer. Math. Soc. (N.S.) \textbf{42} (2005), no.~3,
  365--393.

\bibitem{CG}
Sun-Yung~Alice Chang and Mar{\'\i}a del Mar~Gonz{\'a}lez, \emph{{Fractional
  Laplacian in conformal geometry}}, Advances in Mathematics \textbf{226}
  (2011), no.~2, 1410--1432.

\bibitem{Ch}
Shibing Chen, \emph{A new family of sharp conformally invariant integral
  inequalities}, International Mathematics Research Notices \textbf{2014}
  (2012), no.~5, 1205--1220.

\bibitem{ChenLiOu}
Wenxiong Chen, Congming Li, and Biao Ou, \emph{Classification of solutions for
  an integral equation}, Communications on Pure and Applied Mathematics
  \textbf{59} (2006), no.~3, 330--343.

\bibitem{FG}
Charles Fefferman and C.~Robin Graham, \emph{{Q-curvature and Poincar\'e
  metrics}}, Mathematical Research Letters \textbf{9} (2002), no.~2, 139--151.

\bibitem{FKT}
Rupert~L. Frank, Tobias K{\"o}nig, and Hanli Tang, \emph{{Classification of
  solutions of an equation related to a conformal log Sobolev inequality}},
  Advances in Mathematics \textbf{375} (2020), 107395.

\bibitem{GNN}
B.~Gidas, Wei~Ming Ni, and L.~Nirenberg, \emph{{Symmetry and related properties
  via the maximum principle}}, Communications in Mathematical Physics
  \textbf{68} (1979), no.~3, 209 -- 243.

\bibitem{GT}
David Gilbarg and Neil~S. Trudinger, \emph{{Elliptic Partial Differential
  Equations of Second Order}}, Springer-Verlag Berlin Heidelberg, 2001.

\bibitem{GuoWang}
Qianqiao Guo and Xiaodong Wang, \emph{Uniqueness results for positive harmonic
  functions on {$\overline{\mathbb{B}^n}$} satisfying a nonlinear boundary
  condition}, Calculus of Variations and Partial Differential Equations
  \textbf{59} (2020), no.~5, 146.

\bibitem{HWY2}
Fengbo Hang, Xiaodong Wang, and Xiaodong Yan, \emph{An integral equation in
  conformal geometry}, Annales de l'Institut Henri Poincar{\'e} C, Analyse non
  lin{\'e}aire \textbf{26} (2009), no.~1, 1--21.

\bibitem{L}
Yanyan Li, \emph{Remark on some conformally invariant integral equations: the
  method of moving spheres}, Journal of the European Mathematical Society
  \textbf{6} (2004), no.~2, 153--180.

\bibitem{LZ}
Yanyan Li and Meijun Zhu, \emph{{Uniqueness theorems through the method of
  moving spheres}}, Duke Mathematical Journal \textbf{80} (1995), no.~2, 383 --
  417.

\bibitem{WZ}
Lei Wang and Meijun Zhu, \emph{Liouville theorems on the upper half space},
  Discrete and Continuous Dynamical Systems \textbf{40} (2020), no.~9,
  5373--5381.

\bibitem{Yang}
Qiaohua {Yang}, \emph{{Sharp Sobolev trace inequalities for higher order
  derivatives}}, arXiv e-prints (2019), arXiv:1901.03945.

\end{thebibliography}

\end{document}